\documentclass[reqno]{amsart}
\usepackage{amsmath,amsthm,amssymb,amsfonts}
\usepackage[margin=1.1in]{geometry}
\usepackage{enumitem}
\usepackage{graphicx}
\usepackage{tikz}
  \usepackage{epsfig}
\usepackage[pagebackref, colorlinks=true]{hyperref}

\hypersetup{
    colorlinks=true,%
    citecolor=black,%
    filecolor=black,%
    linkcolor=black,%
    urlcolor=black
}

\def\doi#1{   {\href{http://dx.doi.org/#1}
   {{\mdseries\ttfamily DOI}}}}

\usepackage{graphics}
\def\Xint#1{\mathchoice
  {\XXint\displaystyle\textstyle{#1}}%
  {\XXint\textstyle\scriptstyle{#1}}%
  {\XXint\scriptstyle\scriptscriptstyle{#1}}%
  {\XXint\scriptscriptstyle\scriptscriptstyle{#1}}%
  \!\int}
\def\XXint#1#2#3{{\setbox0=\hbox{$#1{#2#3}{\int}$}
  \vcenter{\hbox{$#2#3$}}\kern-.5\wd0}}

\def\dashint{\Xint-}

\newcommand{\al}{\alpha}                
\newcommand{\om}{\Omega}                \newcommand{\pa}{\partial}
\newcommand{\va}{\varepsilon}           
\newcommand{\be}{\begin{equation}}      \newcommand{\ee}{\end{equation}}

\newcommand{\R}{\mathbb{R}}

\DeclareMathOperator{\dist}{dist}

\DeclareMathOperator{\diam}{diam}
\DeclareMathOperator{\BMO}{BMO}
\DeclareMathOperator{\VMO}{VMO}

\def\<{\langle}             \def\>{\rangle}
\def\({\left(}                 \def\){\right)}
\numberwithin{equation}{section}

\theoremstyle{plain}
\newtheorem{thm}{Theorem}[section]
\newtheorem{cor}[thm]{Corollary}

\newtheorem{lem}[thm]{Lemma}

\theoremstyle{definition}
\newtheorem{defn}[thm]{Definition}

\newtheorem{rem}[thm]{Remark}

\title[Uniform Estimates of Resolvents in Homogenization Theory of Elliptic Systems]
{Uniform Estimates of Resolvents in Homogenization Theory of Elliptic Systems}

\author{Wei Wang}
\address{School of Mathematical Science\\Peking University\\Bejing 100871, P. R. China}
		\email{2201110024@stu.pku.edu.cn}

\begin{document}
\bibliographystyle{plain}
\date{}

\maketitle

\begin{abstract}
\noindent
In this paper, we study the estimates of resolvents $ R(\lambda,\mathcal{L}_{\va})=(\mathcal{L}_{\va}-\lambda I)^{-1} $, where 
$$
\mathcal{L}_{\va}=-\operatorname{div}(A(x/\va)\nabla)
$$ 
is a family of second elliptic operators with symmetric, periodic and oscillating coefficients defined on a bounded domain $ \om $ with $ \va>0 $. For $ 1<p<\infty $, we will establish uniform $ L^p\to L^p $, $ L^p\to W_0^{1,p} $, $ W^{-1,p}\to L^p $ and $ W^{-1,p}\to W_0^{1,p} $ estimates by using the real variable method. Meanwhile, we use Green functions for operators $ \mathcal{L}_{\va}-\lambda I $ to study the asymptotic behavior of $ R(\lambda,\mathcal{L}_{\va}) $ and obtain convergence estimates in $ L^p\to L^p $, $ L^p\to W_0^{1,p} $ norm.
\end{abstract}

\section{Introduction and main results}
\noindent
The main purpose of this paper is to study estimates of resolvents for a family of elliptic operators with rapidly oscillating and symmetric coefficients. Precisely speaking, we consider the resolvents of operators
\begin{align}
\mathcal{L}_{\va}=-\operatorname{div}(A(x/\va)\nabla)=-\frac{\pa}{\pa x_i}\left\{a_{ij}^{\al\beta}\left(\frac{x}{\va}\right)\frac{\pa}{\pa x_j}\right\},\quad\va>0,\label{ho}
\end{align}
where $ 1\leq i,j\leq d $ and $ 1\leq\al,\beta\leq m $. Here, $ d\geq 2 $ denotes the dimension of Euclidean space and $ m\geq 1 $ is the number of equations in the system. The summation convention for repeated indices is used throughout the paper. In discussion, we will always assume that the measurable matrix-valued functions $ A(y)=(a_{ij}^{\al\beta}(y)):\R^d\to\R^{m^2\times d^2} $ satisfy the symmetry condition
\begin{align}
a_{ij}^{\al\beta}(y)=a_{ji}^{\beta\al}(y)\text{ for any }1\leq i,j\leq d,1\leq\al,\beta\leq m\text{ and } y\in\R^d,\label{sy}
\end{align}
the uniform ellipticity condition 
\begin{align}
\mu|\xi|^2\leq a_{ij}^{\al\beta}(y)\xi_{i}^{\al}\xi_{j}^{\beta}\leq \mu^{-1}|\xi|^2\text{ for any }y\in\R^d \text{ and }\xi=(\xi_{i}^{\al})\in\R^{m\times d},\label{el}
\end{align} 
where $ \mu>0 $ is a positive constant and the periodicity condition
\begin{align}
A(y+z)=A(y)\text{ for any }y\in\R^d\text{ and }z\in\mathbb{Z}^d.\label{pe}
\end{align}
To ensure $ L^p $, $ W^{1,p} $ and Lipschitz estimates of operators $ \mathcal{L}_{\va} $, in some situations, we need more smoothness conditions on the coefficients matrix $ A $, i.e. the Hölder regularity of $ A $,
\begin{align}
|A(x)-A(y)|\leq \tau|x-y|^{\nu}\text{ for any }x,y\in\R^d,\label{Hol}
\end{align}
where $ \tau>0,\nu\in(0,1) $ and the $ \VMO $ condition (vanishing mean oscillation condition),
\begin{align}
\sup_{x\in\R^d,0<\rho<t}\dashint_{B(x,\rho)}\left|A(y)-\dashint_{B(x,\rho)}A\right|dy\leq\omega(t)\text{ and } \lim_{t\to 0}\omega(t)=0,\label{VMO}
\end{align}
where $ \omega(t) $ is a continuous nondecreasing function. To make the notation simpler, we denote that
 $ A\in\VMO(\R^d) $ if $ A $ satisfies the $ \VMO $ condition $ \eqref{VMO} $. 

Assume that $ A(y) =(a_{ij}^{\al\beta}(y)) $, the coefficient matrix of $ \mathcal{L}_{\va}=-\operatorname{div}(A(x/\va)\nabla) $, satisfies $ \eqref{el} $ and $ \eqref{pe} $. Let $ \chi_{j}^{\beta}(y)=(\chi_{j}^{\al\beta}(y)) $ denote the matrix of correctors for $ \mathcal{L}_1=-\operatorname{div}(A(x)\nabla) $ in $ \mathbb{R}^d $, where $ \chi_j^{\beta}(y)=(\chi_j^{1\beta}(y),...,\chi_j^{m\beta}(y))\in H_{\operatorname{per}}^1(Y;\mathbb{R}^m) $ is defined by the following cell problem
\begin{align}
\left\{\begin{aligned}
&-\operatorname{div}(A(x)\nabla\chi_{j}^{\beta})=\operatorname{div}(A(x)\nabla P_{j}^{\beta})\text{ in }[0,1)^d,\\
&\chi_{j}^{\beta}\text{ is periodic with respect to }\mathbb{Z}^d\text{ and }\int_{Y}\chi_{j}^{\beta}dy=0,
\end{aligned}\right.\label{correctors}
\end{align}
where $ 1\leq j\leq d $, $ 1\leq\beta\leq m $, $ Y=[0,1)^d\cong \mathbb{R}^d/\mathbb{Z}^d $ and $ P_{j}^{\beta}=(P_{j}^{\al\beta}(x))=(x_j\delta^{\al\beta}) $ with $ \delta^{\al\beta}=1 $ if $ \al=\beta $, $ \delta^{\al\beta}=0 $ otherwise. The homogenized operator is defined by $ \mathcal{L}_{0}=-\operatorname{div}(\widehat{A}\nabla) $, where the coefficients $ \widehat{A}=(\widehat{a}_{ij}^{\al\beta}) $ are given by
\begin{align}
\widehat{a}_{ij}^{\al\beta}=\int_{Y}\left[a_{ij}^{\al\beta}(y)+a_{ik}^{\al\gamma}(y)\frac{\pa}{\pa y_k}\chi_{j}^{\gamma\beta}(y)\right]dy.\label{hoco}
\end{align}
It is easy to show that if $ A $ satisfies $ \eqref{sy} $, then $ \widehat{A} $ is also symmetric, that is
\begin{align}
\widehat{a}_{ij}^{\al\beta}=\widehat{a}_{ji}^{\beta\al}\text{ for any }1\leq i,j\leq d\text{ and } 1\leq\al,\beta\leq m.\label{syl0}
\end{align}
Moreover, if $ A $ satisfies $ \eqref{el} $, then there exists $ \mu_1>0 $ depending only on $ \mu $, such that
\begin{align}
\mu_1|\xi|^2\leq \widehat{a}_{ij}^{\al\beta}\xi_{i}^{\al}\xi_{j}^{\beta}\leq \mu_1^{-1}|\xi|^2\text{ for any }\xi=(\xi_{i}^{\al})\in\R^{m\times d}.\label{ell0}
\end{align} 
For the sake of simplicity, we set $ \min(\mu,\mu_1) $ as new $ \mu $ in the rest of the paper. Let
\begin{align}
b_{ij}^{\al\beta}(y)=\widehat{a}_{ij}^{\al\beta}-a_{ij}^{\al\beta}(y)-a_{ik}^{\al\gamma}(y)\frac{\partial}{\pa y_k}\chi_{j}^{\gamma\beta}(y),\label{Flux coefficients}
\end{align}
where $ 1\leq i,j\leq d $ and $ 1\leq\al,\beta\leq m $. It is easy to see that $ \int_{Y}b_{ij}^{\al\beta}(y)dy=0 $ by $ \eqref{hoco} $. Then in view of $ \eqref{correctors} $ and $ \eqref{hoco} $, there exists the flux corrector $ (F_{kij}^{\al\beta}(y))_{1\leq \al\leq m}\in H_{\operatorname{per}}^1(Y;\R^m) $ such that 
\begin{align}
b_{ij}^{\al\beta}(y)=\frac{\pa}{\partial y_k}\left\{F_{kij}^{\al\beta}(y)\right\}\quad\text{and}\quad  F_{kij}^{\al\beta}(y)=-F_{ikj}^{\al\beta}(y)\label{Flux correctors}
\end{align}
for any $ 1\leq i,j,k\leq d $ and $ 1\leq\al,\beta\leq m $. Moreover, $ (F_{kij}^{\al\beta})\in L^{\infty}(Y) $ if $ (\chi_{j}^{\al\beta}) $ is Hölder continuous. For details about the proof of this fact, one can refer to Chaper 2 of \cite{Shen2}.

Let $ A $ satisfy $ \eqref{sy} $, $ \eqref{el} $ and $ \eqref{pe} $. For $ F\in L^2(\om;\mathbb{C}^m) $ and $ \va\geq 0 $, we can define a linear operator $ T_{\va}:L^2(\om;\mathbb{C}^m)\to H_0^1(\om;\mathbb{C}^m)\subset L^2(\om;\mathbb{C}^m) $ by $ T_{\va}(F)=u_{\va}\in H_0^1(\om;\mathbb{C}^m) $ such that $ \mathcal{L}_{\va}(u_{\va})=F $ in $ \om $ and $ u_{\va}=0 $ on $ \pa\om $. Using standard arguments in \cite{Kenig3}, operators $ T_{\va} $  with $ \va\geq 0 $ are positive and self-adjoint. Therefore, applying the spectrum theory, it is natural to consider properties of resolvents $ R(\lambda,\mathcal{L}_{\va})=(\mathcal{L}_{\va}-\lambda I)^{-1} $ with $ \lambda\in\mathbb{C}\backslash(0,\infty) $. To simplify notations, we will use $ R(\lambda,\mathcal{L}) $ to denote the resolvent of the elliptic operator $ \mathcal{L} $ in the rest of the paper. In order to better characterize resolvents of operators, for $ \lambda=|\lambda|e^{i\theta}\in\mathbb{C}\backslash(0,\infty) $, we define a constant $ c(\lambda,\theta) $ by
\begin{align}
c(\lambda,\theta)=\left\{\begin{matrix}
1&\text{if}&\theta\in[\pi/2,3\pi/2]\text{ or }\lambda=0\\
|\sin\theta|^{-1}&\text{if}&\theta\in(0,\pi/2)\cup(3\pi/2,2\pi)\text{ and }|\lambda|>0.
\end{matrix}\right.\label{ac}
\end{align}
To present our results more conveniently and study the estimates of resolvents, we need to introduce the matrix of Dirichlet correctors $ \Phi_{\va}(x)=(\Phi_{\va,j}^{\beta}(x))_{1\leq j\leq d,1\leq \beta\leq m} $ in $ \om $, defined by
\begin{align}
\mathcal{L}_{\va}(\Phi_{\va,j}^{\beta}(x))=0\text{ in }\om\quad\text{and}\quad \Phi_{\va,j}^{\beta}(x)=P_j^{\beta}(x)\text{ on }\partial\om.\label{Dirichlet correctors}
\end{align}
The Dirichlet correctors were first introduced in \cite{Av1} to study the uniform Lipschitz estimates of homogenization problems. It is known that if $ A $ satisfies $ \eqref{el} $, $ \eqref{pe} $, $ \eqref{Hol} $ and $ \om $ is a bounded $ C^{1,\eta} $ $ \eta\in(0,1) $ domain in $ \R^d $ with $ d\geq 2 $, then
\begin{align}
\|\Phi_{\va,j}^{\beta}-P_j^{\beta}\|_{L^{\infty}(\om)}\leq C\va\quad\text{and}\quad \|\nabla\Phi_{\va}\|_{L^{\infty}(\om)}\leq C,\label{Estimate for Dirichlet correctors}
\end{align}
where $ C $ depends only on $ \mu,d,m,\tau,\nu,\eta $ and $ \om $. The following are the main results of this paper. For the sake of simplicity, we will denote $ \diam(\om) $, the diameter of the bounded domain $ \om $ in $ \R^d $ with $ d\geq 2 $ by $ R_0 $ throughout this paper. 

\begin{thm}[$ L^2 $ and $ H_0^1 $ convergence of resolvents]\label{Approximation 1}
Suppose that $ d\geq 2 $, $ A $ satisfies $ \eqref{sy} $, $ \eqref{el} $, $ \eqref{pe} $ and $ \eqref{Hol} $. Let $ \om $ be a bounded $ C^{1,1} $ domain in $ \mathbb{R}^d $ and $ \lambda=|\lambda|e^{i\theta}\in\mathbb{C}\backslash(0,\infty) $. For $ \va\geq 0 $ and $ F\in L^2(\om;\mathbb{C}^m) $, let $ u_{\va,\lambda}\in H_0^1(\om;\mathbb{C}^m) $ be the unique solution of the Dirichlet problem $
(\mathcal{L}_{\va}-\lambda I)(u_{\va,\lambda})=F $ in $ \om $ and $ u_{\va,\lambda}=0 $ on $ \pa\om $. Then
\begin{align}
\|u_{\va,\lambda}-u_{0,\lambda}\|_{L^2(\om)}&\leq C\va c^2(\lambda,\theta)(R_0^{-2}+|\lambda|)^{-\frac{1}{2}}\|F\|_{L^2(\om)},\label{Convergence rate 11}\\
\|u_{\va,\lambda}-u_{0,\lambda}-(\Phi_{\va,j}^{\beta}-P_{j}^{\beta})\frac{\pa u_{0,\lambda}^{\beta}}{\pa x_j}\|_{H_0^1(\om)}&\leq C\va c^2(\lambda,\theta)\|F\|_{L^2(\om)},\label{Convergence rate 1}
\end{align}
where $ C $ depends only on $ \mu,d,m,\tau,\nu,\om $ and $ \Phi_{\va} $ is given by $ \eqref{Dirichlet correctors} $. In operator forms,
\begin{align}
\|R(\lambda,\mathcal{L}_{\va})-R(\lambda,\mathcal{L}_{0})\|_{L^2(\om)\to L^2(\om)}&\leq C\va c^2(\lambda,\theta)(R_0^{-2}+|\lambda|)^{-\frac{1}{2}},\label{Operator estimate 11}\\
\|R(\lambda,\mathcal{L}_{\va})-R(\lambda,\mathcal{L}_{0})-K_{\va}(\lambda)\|_{L^2(\om)\to H_0^1(\om)}&\leq C\va c^2(\lambda,\theta),\label{Operator estimate 1}
\end{align}
where $ K_{\va}(\lambda) $ are the operator correctors given by the formula
\begin{align}
K_{\va}(\lambda)=\{K_{\va}^{\al}(\lambda)\}_{1\leq\al\leq m}=\{(\Phi_{\va,j}^{\al\beta}(x)-P_{j}^{\al\beta}(x))\pa_{x_j}R(\lambda,\mathcal{L}_{0})^{\beta}\}_{1\leq\al\leq m}.\label{Operator corrector}
\end{align}
\end{thm}

Theorem \ref{Approximation 1} is actually a quantitative result of the homogenization theory for the operator $ \mathcal{L}_{\va}-\lambda I $. In fact, if $ \om $ is a bounded Lipschitz domain in $ \R^d $ with $ d\geq 2 $ and $ A $ satisfies $ \eqref{el},\eqref{pe} $, it can be shown that $ u_{\va,\lambda}\to u_{0,\lambda} $ weakly in $ H_0^1(\om;\mathbb{C}^m) $ and strongly in $ L^2(\om;\mathbb{C}^m) $. For more details about the homogenization problems of elliptic systems, one can refer to \cite{Bensou} and \cite{Shen2}. 

In \cite{Su1}, assuming that $ \om $ is a bounded and $ C^{1,1} $ domain, the convergence rates of such resolvents are established, that is, for $ \lambda\in\mathbb{C}\backslash[0,\infty) $ and $ 0\leq\va<1 $,
\begin{align}
\|R(\lambda,\mathcal{L}_{\va})-R(\lambda,\mathcal{L}_{0})\|_{L^2(\om)\to L^2(\om)}&\leq Cc^2(\lambda,\theta)(\va|\lambda|^{-\frac{1}{2}}+\va^2),\label{su1l2}\\
\|R(\lambda,\mathcal{L}_{\va})-R(\lambda,\mathcal{L}_{0})-K_0(\va,\lambda)\|_{L^2(\om)\to H^1(\om)}&\leq Cc^2(\lambda,\theta)\va^{\frac{1}{2}},\label{su1l22}
\end{align}
where $ K_0(\va,\lambda) $ is some corrector of the homogenization problem and $ C $ depends only on $ \mu,d,m $ and $ \om $. These estimates are obtained by applying the results when $ \om=\R^d $ in \cite{Birman1}, \cite{Birman2} and some extension theorems. By using such approximation estimates for $ R(\lambda,\mathcal{L}_{\va}) $, \cite{Meshkova1} and \cite{Meshkova2} gave the convergence rates for homogenization problems of parabolic and hyperbolic systems in $ L^2 $ and $ H_0^1 $ space.  What is new for Theorem \ref{Approximation 1} is that we use different operator corectors and obtain a more brief proof under higher regularity assumptions of $ A $. The main method is developed in the proof of Theorem 2.4 in \cite{Kenig3}, which is used to deal with the case that $ \lambda=0 $. In Theorem 1.5 of \cite{Xu1}, the author generalized convergence results in \cite{Kenig3} to elliptic operators with lower order terms. Such results are somewhat similar to this paper, but this does not mean that the conclusions of this paper can be trivially covered. To some extent, Theorem \ref{Approximation 1} is a generalization of Theorem 2.4 in \cite{Kenig3} and Theorem 1.5 in \cite{Xu1}. The difference between this and Theorem 1.5 in \cite{Xu1} is that the constants on the right hand side of $ \eqref{Convergence rate 11} $ do not depend on the module of $ \lambda $, i.e. $ |\lambda| $. Moreover, we remark that these estimates in Theorem \ref{Approximation 1} can also be applied to evolution systems and obtain similar results given in \cite{Lin1}, \cite{Meshkova1} and \cite{Meshkova2}.

Given a sectorial domain
\begin{align}
\Sigma_{\theta_0}=\left\{\lambda=|\lambda|e^{i\theta}\in\mathbb{C}:|\lambda|>0,|\arg\theta|>\pi-\theta_0\right\},\label{td}
\end{align}
where $ \theta_0\in(0,\frac{\pi}{2}) $, we can obtain the following estimates. 

\begin{thm}[$ L^p $ and $ W_0^{1,p} $ estimates of resolvents]\label{Lp estimates of resolventsf}
Suppose that $ \va\geq 0 $ and $ d\geq 2 $. Let $ \lambda\in\Sigma_{\theta_0}\cup\{0\} $ with $ \theta_0\in(0,\frac{\pi}{2}) $ and $ \om $ be a bounded $ C^1 $ domain in $ \R^d $. Assume that $ A $ satisfies $ \eqref{sy} $, $ \eqref{el} $, $ \eqref{pe} $ and $ \eqref{VMO} $. Then for any $ 1<p<\infty $, $ F\in L^p(\om;\mathbb{C}^m) $ and $ f\in L^p(\om;\mathbb{C}^{m\times d}) $, there exists a unique $ u_{\va,\lambda}\in W_0^{1,p}(\om;\mathbb{C}^m) $ such that $ (\mathcal{L}_{\va}-\lambda I)(u_{\va,\lambda})=F+\operatorname{div}(f) $ in $ \om $, $ u_{\va,\lambda}=0 $ on $ \pa\om $
and satisfies the uniform estimates
\begin{align}
\|u_{\va,\lambda}\|_{L^p(\om)}&\leq C_{p,\theta_0}(R_0^{-2}+|\lambda|)^{-1}\|F\|_{L^p(\om)}+C_{p,\theta_0}(R_0^{-2}+|\lambda|)^{-\frac{1}{2}}\|f\|_{L^p(\om)},\label{Lpestiamesu3}\\
\|\nabla u_{\va,\lambda}\|_{L^p(\om)}&\leq C_{p,\theta_0}(R_0^{-2}+|\lambda|)^{-\frac{1}{2}}\|F\|_{L^p(\om)}+C_{p,\theta_0}\|f\|_{L^p(\om)},\label{LpW1pestiamesu3}
\end{align}
where $ C_{p,\theta_0} $ depends only on $ \mu,d,m,\omega(t),p,\theta_0 $ and $ \om $. In operator forms,
\begin{align}
\|R(\lambda,\mathcal{L}_{\va})\|_{L^p(\om)\to L^p(\om)}&\leq C_{p,\theta_0}(R_0^{-2}+|\lambda|)^{-1},\label{reLpLp}\\
\|R(\lambda,\mathcal{L}_{\va})\|_{W^{-1,p}(\om)\to L^p(\om)}&\leq C_{p,\theta_0}(R_0^{-2}+|\lambda|)^{-\frac{1}{2}},\label{reW-1pLp}\\
\|R(\lambda,\mathcal{L}_{\va})\|_{L^p(\om)\to W_0^{1,p}(\om)}&\leq C_{p,\theta_0}(R_0^{-2}+|\lambda|)^{-\frac{1}{2}},\label{reLpW1p}\\
\|R(\lambda,\mathcal{L}_{\va})\|_{W^{-1,p}(\om)\to W_0^{1,p}(\om)}&\leq C_{p,\theta_0},\label{reW-1pW1p}
\end{align} 
where $ W^{-1,p}(\om;\mathbb{C}^m)\triangleq(W_0^{1,p'}(\om;\mathbb{C}^m))^* $ with $ p'=\frac{p}{p-1} $ being the conjugate number of $ p $. 
\end{thm}

The $ L^p\to L^p $ estimates of resolvents for elliptic operators are widely studied and have abundant materials. We will list some results for non-homogenization problems. For $ m=1 $ and $ \BMO $ coefficients, see \cite{Kang1}; for constant coefficients and Dirichlet boundary conditions, see \cite{Shen1}; for constant coefficients and Neumann boundary conditions, see \cite{Wei2}; and for variable coefficients and Lipschitz domains, see \cite{Wei1}. It is noteworthy that in \cite{Wei1}, the authors derived the $ L^p\to L^p $ estimate of resolvents $ R(\lambda,\mathcal{L}_{1}) $ without any regularity assumptions on $ A $ with $ p $ being closed to $ 2 $ when $ \om $ is a bounded Lipschitz domain. These results can be used for homogenization problems. In this point of view, Theorem \ref{Lp estimates of resolventsf} gives stronger results under more assumptions of $ A $ and $ \om $. We see that the estimates established in Theorem \ref{Lp estimates of resolventsf} are sharp in view of the estimates for $ p=2 $, which will be given later. 

The main tool for the proof of Theorem \ref{Lp estimates of resolventsf} is real variable method, which was original used in \cite{Caf1} to deal with $ W^{1,p} $ estimates of elliptic equations. We remark that for the operator $ \mathcal{L}_{\va}-\lambda I $, the use of real variable method is slightly different from the case for $ \mathcal{L}_{\va} $. 

Besides the uniform $ W^{-1,p},L^p\to L^p,W_0^{1,p} $ estimates of resolvents, one may also concern about the convergence rates in $ L^p\to L^p $ or $ L^p\to W_0^{1,p} $ norm with $ 1<p<\infty $. For these topics, we have the following theorems. 

\begin{thm}[$ L^p\to L^p $ convergence rates of resolvents]\label{Lpconres}
Suppose that $ d\geq 2 $, $ \lambda\in\Sigma_{\theta_0}\cup\{0\} $ with $ \theta_0\in(0,\frac{\pi}{2}) $ and $ A $ satisfies $ \eqref{sy} $, $ \eqref{el} $, $ \eqref{pe} $, $ \eqref{Hol} $. Let $ \om $ be a bounded $ C^{1,1} $ domain in $ \mathbb{R}^d $. For $ \va\geq 0 $ and $ F\in L^p(\om;\mathbb{C}^m) $, let $ u_{\va,\lambda}\in W_0^{1,p}(\om;\mathbb{C}^m) $ be the unique solution of the Dirichlet problem $ (\mathcal{L}_{\va}-\lambda I)(u_{\va,\lambda})=F $ in $ \om $ and $ u_{\va,\lambda}=0 $ on $ \pa\om $. Then for any $ 1<p<\infty $,
\begin{align}
\|u_{\va,\lambda}-u_{0,\lambda}\|_{L^p(\om)}&\leq C_{p,\theta_0}\va (R_0^{-2}+|\lambda|)^{-\frac{1}{2}}\|F\|_{L^p(\om)},\label{Convergence rate Lp}
\end{align}
where $ C_{p,\theta_0} $ depends only on $ \mu,d,m,\tau,\nu,p $ and $ \om $. In operator forms,
\begin{align}
\|R(\lambda,\mathcal{L}_{\va})-R(\lambda,\mathcal{L}_{0})\|_{L^p(\om)\to L^p(\om)}&\leq C_{p,\theta_0}\va (R_0^{-2}+|\lambda|)^{-\frac{1}{2}}.\label{**-}
\end{align}
\end{thm}

\begin{thm}[$ L^p\to W_0^{1,p} $ convergence rates of resolvents]\label{LpW1pconreso}
Suppose that $ d\geq 2 $, $ \lambda\in\Sigma_{\theta_0}\cup\{0\} $ with $ \theta_0\in(0,\frac{\pi}{2}) $ and $ A $ satisfies $ \eqref{sy} $, $ \eqref{el} $, $ \eqref{pe} $, $ \eqref{Hol} $. Let $ \om $ be a bounded $ C^{2,1} $ domain in $ \mathbb{R}^d $. For $ \va\geq 0 $ and $ F\in L^p(\om;\mathbb{C}^m) $, let $ u_{\va,\lambda}\in W_0^{1,p}(\om;\mathbb{C}^m) $ be the unique solution of the Dirichlet problem $ (\mathcal{L}_{\va}-\lambda I)(u_{\va,\lambda})=F $ in $ \om $ and $  u_{\va,\lambda}=0 $ on $ \pa\om $. Then for any $ 1<p<\infty $,
\begin{align}
\|u_{\va,\lambda}-u_{0,\lambda}-(\Phi_{\va,j}^{\beta}-P_j^{\beta})\frac{\pa u_{0,\lambda}^{\beta}}{\pa x_j}\|_{W_0^{1,p}(\om)}&\leq C_{p,\theta_0}\va\left\{\ln[\va^{-1}R_0+2]\right\}^{4|\frac{1}{2}-\frac{1}{p}|}\|F\|_{L^p(\om)},\label{Convergence rate LpW1p}
\end{align}
where $ C_{p,\theta_0} $ depends only on $ \mu,d,m,\tau,\nu,p $ and $ \om $. In operator forms,
\begin{align}
\|R(\lambda,\mathcal{L}_{\va})-R(\lambda,\mathcal{L}_{0})-K_{\va}(\lambda)\|_{L^p(\om)\to W_0^{1,p}(\om)}&\leq C_{p,\theta_0}\va\left\{\ln[\va^{-1}R_0+2]\right\}^{4|\frac{1}{2}-\frac{1}{p}|},\label{***-}
\end{align}
where $ K_{\va}(\lambda) $ is defined by $ \eqref{Operator corrector} $.
\end{thm}

For the operator $ \mathcal{L}_{\va} $, the estimates of $ L^p $ and $ W_0^{1,p} $ convergence rates are established in \cite{Kenig4}. These estimates are established by obtaining the convergence rates of Green functions of operators $ \mathcal{L}_{\va} $. The existences and pointwise estimates of Green functions for operators $ \mathcal{L}_{\va} $ are well-known and one can refer to \cite{Hofmann1} and \cite{Taylor} for details. However, for operators $ \mathcal{L}_{\va}-\lambda I $ with $ \lambda\in\Sigma_{\theta_0} $ and $ \theta_0\in(0,\frac{\pi}{2}) $, the constructions and estimates of Green functions are still unknown. We will construct the Green functions for operators $ \mathcal{L}_{\va}-\lambda I $ in this paper and establish the uniform estimates of them. We point out that the constructions of Green functions for operators $ \mathcal{L}_{\va}-\lambda I $ with $ d\geq 3 $ are similar to which of the Green functions of elliptic operators with lower order terms in \cite{Xu1}. There are some differences between these two. The first is that the problems considered in this paper are about complex valued functions. This makes calculations a little more complicated. The second difference is that the impact of the parameter $ \lambda $ should be calculated explicitly. On the other hand, for $ d=2 $, we need to employ arguments in \cite{Dong2}, \cite{Dong1} and \cite{Taylor}. The constructions of Green functions for $ d\geq 3 $ and $ d=2 $ are rather different and to deal with the case $ d=2 $, we need some properties of $ \BMO $ space and Hardy space. Moreover, proofs of uniform estimates on two dimensional Green functions are much more difficult than that of the case $ d\geq 3 $. These estimates are new and are the most important innovations of this paper. What is essential is that all the uniform regularity estimates of Green functions in this paper are scaling invariant (see \cite{Xu2} for such estimates of fundamental solutions related to generalized elliptic operators with lower order terms). 

After constructing Green functions for the operator $ \mathcal{L}_{\va}-\lambda I $, we will use almost the same methods in \cite{Kenig4} to prove Theorem \ref{Lpconres}-\ref{LpW1pconreso}. In view of the convergence rates with $ p=2 $, we can also see that estimates $ \eqref{**-} $ and $ \eqref{***-} $ are sharp.

The rest of this paper is organized as follows. In Section \ref{Preliminaries}, we will give some basic ingredients in the proof, including the $ L^2 $ regularity theory, Caccioppoli's inequality and  some properties about $ \BMO $, Hardy space. In Section \ref{W1pand Lipschitz estimates for}, we will establish the $ W^{1,p} $, Hölder and Lipschitz estimates for the operator $ \mathcal{L}_{\va}-\lambda I $. In Section \ref{Lp estimates of resolventssection}, we will first prove Theorem \ref{Lp estimates of resolventsf} by using real variable methods. Then we will construct Green functions $ G_{\va,\lambda}(x,y) $ for operators $ \mathcal{L}_{\va}-\lambda I $ and obtain the regularity estimates for them. In Section \ref{Estimates of convergence of resolvents}, firstly, we will use standard methods given in \cite{Kenig3} to prove Theorem \ref{Approximation 1}. Next, we will calculate the convergence rates of Green functions for $ (\mathcal{L}_{\va}-\lambda I) $. Using these tools, we can prove Theorem \ref{Lpconres}-\ref{LpW1pconreso}. 

\section{Preliminaries}\label{Preliminaries}
\subsection{Energy estimates and Caccioppoli's inequality}

A notable observation gives that for any $ \xi=\xi^{(1)}+i\xi^{(2)}\in \mathbb{C}^{m\times d} $, where $ \xi^{(1)},\xi^{(2)}\in\R^{m\times d} $,
\begin{align}
a_{kj}^{\al\beta}(y)\xi_{k}^{\al}\overline{\xi_{j}^{\beta}}&=a_{kj}^{\al\beta}(y)\left(\xi_{k}^{(1)\al}+i\xi_{k}^{(2)\al}\right)\left(\xi_{j}^{(1)\beta}-i\xi_{j}^{(2)\beta}\right)=a_{kj}^{\al\beta}(y)\xi_{k}^{(1)\al}\xi_{j}^{(1)\beta}+a_{kj}^{\al\beta}(y)\xi_{k}^{(2)\al}\xi_{j}^{(2)\beta}.\nonumber
\end{align}
This, together with $ \eqref{el} $, gives the ellipticity condition for $ A $ with complex variables, i.e.
\begin{align}
2\mu|\xi|^2\leq a_{ij}^{\al\beta}(y)\xi_{i}^{\al}\overline{\xi_{j}^{\beta}}\leq 2\mu^{-1}|\xi|^2\text{ for any }  y\in\mathbb{R}^d\text{ and }\xi=(\xi_i^{\al})\in\mathbb{C}^{m\times d}.\label{ellcom}
\end{align}
Similarly, in view of $ \eqref{syl0} $ and $ \eqref{ell0} $, we can also infer that
\begin{align}
2\mu|\xi|^2\leq \widehat{a}_{ij}^{\al\beta}\xi_{i}^{\al}\overline{\xi_{j}^{\beta}}\leq 2\mu^{-1}|\xi|^2\text{ for any }  y\in\mathbb{R}^d\text{ and }\xi=(\xi_i^{\al})\in\mathbb{C}^{m\times d}.\label{ellcoml0}
\end{align}
Moreover, for $ d\geq 2 $, let $ \om $ be a bounded domain in $ \R^d $, $ \va\geq 0 $, $ \lambda=|\lambda|e^{i\theta}\in\mathbb{C}\backslash(0,\infty) $ and $ A $ satisfy $ \eqref{sy},\eqref{el} $. Define a bilinear form $ B_{\va,\lambda,\om}[\cdot,\cdot]:H_0^1(\om;\mathbb{C}^m)\times H_0^1(\om;\mathbb{C}^m)\to\mathbb{C} $ by
\begin{align}
B_{\va,\lambda,\om}[u,v]=\int_{\om}A_{\va}(x)\nabla u(x)\overline{\nabla v(x)}dx-\lambda\int_{\om}u(x)\overline{v(x)}dx,\label{bilinearBva}
\end{align}
where $ A_{\va}(x)=A(x/\va) $ for $ \va>0 $ and $ A_0(x)=\widehat{A} $. For $ \lambda\in\mathbb{C}\backslash[0,\infty) $ and $ u\in H_0^1(\om;\mathbb{C}^m) $,
\begin{align}
B_{\va,\lambda,\om}[u,u]=\int_{\om}A_{\va}(x)\nabla u(x)\overline{\nabla u(x)}dx-\lambda\int_{\om}|u(x)|^2dx.\label{buu}
\end{align}
If $ \operatorname{Re}(\lambda)\geq 0 $, we can take imaginary parts of both sides of $ \eqref{buu} $ and obtain that
\begin{align}
|B_{\va,\lambda,\om}[u,u]|\geq|\operatorname{Im}(\lambda)|\int_{\om}|u(x)|^2dx\Rightarrow \|u\|_{L^2(\om)}^2\leq c(\lambda,\theta)|\lambda|^{-1}|B_{\va,\lambda,\om}[u,u]|.\label{buuim}
\end{align}
If $ \operatorname{Re}(\lambda)<0 $, we can take real parts of both sides of $ \eqref{buu} $ and get that
\begin{align}
|B_{\va,\lambda,\om}[u,u]|\geq|\operatorname{Re}(\lambda)|\int_{\om}|u(x)|^2dx.\label{buure}
\end{align}
Adding $ \eqref{buure} $ by $ \eqref{buuim} $, it can be easily shown that
\begin{align}
\|u\|_{L^2(\om)}^2\leq Cc(\lambda,\theta)|\lambda|^{-1}|B_{\va,\lambda,\om}[u,u]|\text{ for any }\lambda\in\mathbb{C}\backslash[0,\infty).\label{ubuuboun}
\end{align}
This, together with $ \eqref{ellcom} $, $ \eqref{ellcoml0} $ and $ \eqref{buu} $, implies that
\begin{align}
\mu\|\nabla u\|_{L^2(\om)}^2\leq |B_{\va,\lambda,\om}[u,u]|+|\lambda|\|u\|_{L^2(\om)}^2\leq 2c(\lambda,\theta)|B_{\va,\lambda,\om}[u,u]|.\label{laxmil}
\end{align}
If $ \lambda=0 $, it is obvious that $ \eqref{laxmil} $ is still true. By using Poincaré's inequality, then
\begin{align}
\|u\|_{L^2(\om)}^2\leq CR_0^2\|\nabla u\|_{L^2(\om)}^2\leq Cc(\lambda,\theta)R_0^2|B_{\va,\lambda,\om}[u,u]|,\nonumber
\end{align}
which, together with $ \eqref{ubuuboun} $, implies that for any $ \lambda\in\mathbb{C}\backslash(0,\infty) $,
\begin{align}
\|u\|_{L^2(\om)}^2\leq Cc(\lambda,\theta)(R_0^{-2}+|\lambda|)^{-1}|B_{\va,\lambda,\om}[u,u]|,\label{ABCDE}
\end{align}
where $ C $ depends only on $ \mu,d,m $ and $ \om $. Another important fact is that if $ A $ satisfies $ \eqref{sy} $, the adjoint operator of $ \mathcal{L}_{\va}-\lambda I $ is $ \mathcal{L}_{\va}-\overline{\lambda} I $. Indeed,
\begin{align}
\langle(\mathcal{L}_{\va}-\lambda I)(u),v\rangle_{H^{-1}(\om)\times H_0^1(\om)}&=B_{\va,\lambda,\om}[u,v]=\overline{B_{\va,\overline{\lambda},\om}[v,u]}=\langle u,(\mathcal{L}_{\va}-\overline{\lambda} I)(v)\rangle_{H_0^1(\om)\times H^{-1}(\om)}.\nonumber
\end{align}
In view of $ \eqref{laxmil} $ and well-known Lax-Milgram theorem, it is easy to show the following theorem for existence of solutions corresponding to operators $ \mathcal{L}_{\va}-\lambda I $.
\begin{thm}\label{existencethm}
Let $ \om $ be a bounded $ C^1 $ domain in $ \R^d $ with $ d\geq 2 $. Assume that $ A $ satisfies $ \eqref{sy} $ and $ \eqref{el} $. Then for any $ \va\geq 0 $, $ F\in H^{-1}(\om;\mathbb{C}^m) $, there is a unique weak solution $ u_{\va,\lambda}\in H_0^1(\om;\mathbb{C}^m) $ satisfying the Dirichlet problem $
(\mathcal{L}_{\va}-\lambda I)(u_{\va,\lambda})=F $ in $ \om $ and $ u_{\va,\lambda}=0 $ on $ \pa\om $. That is, for any $ \varphi\in H_0^1(\om;\mathbb{C}^m) $,
\begin{align}
B_{\va,\lambda,\om}[u_{\va,\lambda},\varphi]=\langle F,\varphi\rangle_{H^{-1}(\om)\times H_0^1(\om)}.\nonumber
\end{align}
\end{thm}

\begin{lem}\label{l2reses}
Suppose that $ d\geq 2 $ and $ \lambda=|\lambda|e^{i\theta}\in\mathbb{C}\backslash(0,\infty) $. Let $ \om $ be a bounded $ C^{1} $ domain in $ \mathbb{R}^d $. Assume that $ A $ satisfies $ \eqref{sy} $ and $ \eqref{el} $. Then for any $ \va\geq 0 $, $ F\in L^2(\om;\mathbb{C}^m) $ and $ f\in L^2(\om;\mathbb{C}^{m\times d}) $, there exists a unique weak solution $ u_{\va,\lambda}\in H_0^1(\om;\mathbb{C}^m) $ of the Dirichlet problem $ (\mathcal{L}_{\va}-\lambda I)(u_{\va,\lambda})=F+\operatorname{div}(f) $ in $ \om $ and $ u_{\va,\lambda}=0 $ on $ \pa\om $, satisfying the uniform energy estimates
\be
\begin{aligned}
\|u_{\va,\lambda}\|_{L^2(\om)}&\leq Cc(\lambda,\theta)\left\{(R_0^{-2}+|\lambda|)^{-1}\|F\|_{L^2(\om)}+(R_0^{-2}+|\lambda|)^{-\frac{1}{2}}\|f\|_{L^2(\om)}\right\},\\
\|\nabla u_{\va,\lambda}\|_{L^2(\om)}&\leq Cc(\lambda,\theta)\left\{(R_0^{-2}+|\lambda|)^{-\frac{1}{2}}\|F\|_{L^2(\om)}+\|f\|_{L^2(\om)}\right\},
\end{aligned}\label{L2uva}
\ee
where $ C $ depends only on $ \mu,d,m $ and $ \om $. Moreover, if $ \om $ is a bounded $ C^{1,1} $ domain, $ \va=0 $ and $ f\equiv 0 $, then
\begin{align}
\|\nabla^2 u_{0,\lambda}\|_{L^2(\om)}\leq Cc(\lambda,\theta)\|F\|_{L^2(\om)}.\label{L2n2u0}
\end{align}
In operator forms, 
\begin{align}
\|R(\lambda,\mathcal{L}_{\va})\|_{L^2(\om)\to L^2(\om)}&\leq Cc(\lambda,\theta)(R_0^{-2}+|\lambda|)^{-1},\label{L2o1}\\
\max\left\{\|R(\lambda,\mathcal{L}_{\va})\|_{H^{-1}(\om)\to L^2(\om)},\|R(\lambda,\mathcal{L}_{\va})\|_{L^2(\om)\to H_0^1(\om)}\right\}&\leq Cc(\lambda,\theta)(R_0^{-2}+|\lambda|)^{-\frac{1}{2}},\label{L2o2}\\
\max\left\{\|R(\lambda,\mathcal{L}_{0})\|_{L^2(\om)\to H^2(\om)},\|R(\lambda,\mathcal{L}_{\va})\|_{H^{-1}(\om)\to H_0^1(\om)}\right\}&\leq Cc(\lambda,\theta).\label{L2o4}
\end{align}
\end{lem}
\begin{proof}
The existence is ensured by Theorem \ref{existencethm}, we only need to show $ \eqref{L2uva} $ and $ \eqref{L2n2u0} $. Firstly, we can choose $ u_{\va,\lambda}^{(1)} $ and $ u_{\va,\lambda}^{(2)} $ in $ H_0^1(\om;\mathbb{C}^m) $ such that
\begin{align}
\left\{\begin{matrix}
(\mathcal{L}_{\va}-\lambda I)(u_{\va,\lambda}^{(1)})=F&\text{in}&\om,\\
u_{\va,\lambda}^{(1)}=0&\text{on}&\pa\om,
\end{matrix}\right.\quad\text{and}\quad \left\{\begin{matrix}
(\mathcal{L}_{\va}-\lambda I)(u_{\va,\lambda}^{(2)})=\operatorname{div}(f)&\text{in}&\om,\\
u_{\va,\lambda}^{(2)}=0&\text{on}&\pa\om.
\end{matrix}\right.\nonumber
\end{align}
The next thing is to estimate $ u_{\va,\lambda}^{(1)} $ and $ u_{\va,\lambda}^{(2)} $ respectively. For $ u_{\va,\lambda}^{(1)} $, we need to prove that
\begin{align}
\|\nabla u_{\va,\lambda}^{(1)}\|_{L^2(\om)}+(R_0^{-2}+|\lambda|)^{\frac{1}{2}}\|u_{\va,\lambda}^{(1)}\|_{L^2(\om)}&\leq Cc(\lambda,\theta)(R_0^{-2}+|\lambda|)^{-\frac{1}{2}}\|F\|_{L^2(\om)}.\label{ww71}
\end{align}
Choosing $ u_{\va,\lambda}^{(1)} $ as the test function, it can be obtained that
\begin{align}
B_{\va,\lambda,\om}[u_{\va,\lambda}^{(1)},u_{\va,\lambda}^{(1)}]=\int_{\om}F(x)\overline{u_{\va,\lambda}^{(1)}(x)}dx=(F,u_{\va,\lambda}^{(1)})_{L^2(\om)\times L^2(\om)}.\label{wef1}
\end{align}
In view of $ \eqref{laxmil} $ and $ \eqref{ABCDE} $, we have
\begin{align}
\|u_{\va,\lambda}^{(1)}\|_{L^2(\om)}^2&\leq Cc(\lambda,\theta)(R_0^{-2}+|\lambda|)^{-1}\|u_{\va,\lambda}^{(1)}\|_{L^2(\om)}\|F\|_{L^2(\om)},\nonumber\\
\mu\|\nabla u_{\va,\lambda}^{(1)}\|_{L^2(\om)}^2&\leq Cc(\lambda,\theta)\|u_{\va,\lambda}^{(1)}\|_{L^2(\om)}\|F\|_{L^2(\om)},\nonumber
\end{align}
which complete the proof of $ \eqref{ww71} $. For $ u_{\va,\lambda}^{(2)} $, it suffices to show that
\begin{align}
\|\nabla u_{\va,\lambda}^{(2)}\|_{L^2(\om)}+(R_0^{-2}+|\lambda|)^{\frac{1}{2}}\|u_{\va,\lambda}^{(2)}\|_{L^2(\om)}&\leq Cc(\lambda,\theta)\|f\|_{L^2(\om)}.\label{flam-1}
\end{align}
Choosing $ u_{\va,\lambda}^{(2)} $ as the test function, then
\begin{align}
B_{\va,\lambda,\om}[u_{\va,\lambda}^{(2)},u_{\va,\lambda}^{(2)}]=-\int_{\om}f(x)\overline{\nabla u_{\va,\lambda}^{(2)}(x)}dx=-(f,\nabla u_{\va,\lambda}^{(2)})_{L^2(\om)\times L^2(\om)}.\label{wef2}
\end{align}
Again, by applying $ \eqref{laxmil} $ and $ \eqref{ABCDE} $, it is easy to get that
\begin{align}
\|u_{\va,\lambda}^{(2)}\|_{L^2(\om)}^2&\leq Cc(\lambda,\theta)(R_0^{-2}+|\lambda|)^{-1}\|\nabla u_{\va,\lambda}^{(2)}\|_{L^2(\om)}\|f\|_{L^2(\om)},\nonumber\\
\mu\|\nabla u_{\va,\lambda}^{(2)}\|_{L^2(\om)}^2&\leq Cc(\lambda,\theta)\|\nabla u_{\va,\lambda}^{(2)}\|_{L^2(\om)}\|f\|_{L^2(\om)},\nonumber
\end{align}
which implies $ \eqref{flam-1} $. At last, for the proof of $ \eqref{L2n2u0} $, since $ \om $ is $ C^{1,1} $, one can deduce that
\begin{align}
\|\nabla^2u_{0,\lambda}\|_{L^2(\om)}\leq C\left\{\|F\|_{L^2(\om)}+|\lambda|\|u_{0,\lambda}\|_{L^2(\om)}\right\}\leq Cc(\lambda,\theta)\|F\|_{L^2(\om)},\nonumber
\end{align}
where, for the first inequality, we have used standard $ H^2 $ estimates for elliptic systems with constant coefficients and for the second inequality, we have used $ \eqref{L2uva} $.
\end{proof}

\begin{rem}
For $ d\geq 3 $ and $ \lambda\in\mathbb{C}\backslash(0,\infty) $, using Hölder's inequality and Sobolev embedding theorem $ H_0^1(\om)\subset L^{\frac{2d}{d-2}}(\om) $, we have
\begin{align}
|(F,u_{\va,\lambda}^{(1)})_{L^2(\om)\times L^2(\om)}|\leq\|F\|_{L^{\frac{2d}{d+2}}(\om)}\|u_{\va,\lambda}^{(1)}\|_{L^{\frac{2d}{d-2}}(\om)}\leq \|F\|_{L^{\frac{2d}{d+2}}(\om)}\|\nabla u_{\va,\lambda}^{(1)}\|_{L^2(\om)}.\nonumber
\end{align}
By using the same arguments in the proof of Theorem \ref{l2reses}, it is easy to find that for any $ \va\geq 0 $,
\begin{align}
\|\nabla u_{\va,\lambda}\|_{L^2(\om)}+(R_0^{-2}+|\lambda|)^{\frac{1}{2}}\|u_{\va,\lambda}\|_{L^2(\om)}\leq Cc(\lambda,\theta)\left\{\|F\|_{L^{\frac{2d}{d+2}}(\om)}+\|f\|_{L^{2}(\om)}\right\},\label{u2dd+2}
\end{align}
where $ C $ depends only on $ \mu,d,m $ and $ \om $. Similarly, If $ d=2 $, in view of the Sobolev embedding $ H_0^1(\om)\subset L^{q'}(\om) $ with $ q'=\frac{q}{q-1}>1 $, it can be easily seen that for any $ q>1 $,
\begin{align}
\|\nabla u_{\va,\lambda}\|_{L^2(\om)}+(R_0^{-2}+|\lambda|)^{\frac{1}{2}}\|u_{\va,\lambda}\|_{L^2(\om)}\leq Cc(\lambda,\theta)\left\{R_0^{2(1-\frac{1}{q})}\|F\|_{L^{q}(\om)}+\|f\|_{L^{2}(\om)}\right\},\label{uq}
\end{align}
where $ C $ depends only on $ \mu,m,q $ and $ \om $.
\end{rem}

Next, we will establish the Caccioppoli's inequality for the operator $ \mathcal{L}_{\va}-\lambda I $. In some point of view, Caccioppoli's inequality can be seen as the localization of $ \eqref{u2dd+2} $. In this paper, we need to obtain Caccioppoli's inequality that is scaling invariant. This means that $ \lambda $ cannot be regarded as a constant and its influence on the constants of this inequality should be calculated explicitly. The scaling invariant Caccioppoli's inequality plays a vital role in the proofs of other scaling invariant inequalities. The idea comes from \cite{Shen1,Xu2} and we combine the methods in both of them to complete the proof for the sake of completeness. To simplify the notations, we can define
\begin{align}
\om(x_0,R)=\om\cap B(x_0,R)\quad\text{and}\quad \Delta(x_0,R)=\pa\om\cap B(x_0,R)\label{bn}
\end{align}
for any $ 0<R<R_0 $ and $ x_0\in\om $. Sometimes we will use $ \om_r $ and $ \Delta_r $ to denote $ \om(x_0,r) $ and $ \Delta(x_0,r) $ if no confusion would be made.

\begin{lem}[Caccioppoli's inequality]\label{Caccio}
Let $ \va\geq 0 $, $ d\geq 2 $, $ \lambda\in\Sigma_{\theta_0}\cup\{0\} $ with $ \theta_0\in(0,\frac{\pi}{2}) $ and $ \om $ be a bounded $ C^1 $ domain in $ \R^d $. Assume that $ A $ satisfies $ \eqref{sy} $ and $ \eqref{el} $. Suppose that $ 0<R<R_0 $, $ x_0\in\om $, $ f\in L^2(\om(x_0,2R);\mathbb{C}^m) $ and $ F\in L^{q}(\om(x_0,2R);\mathbb{C}^m) $, where $ q=\frac{2d}{d+2} $ if $ d\geq 3 $ and $ q>1 $ if $ d=2 $. If $ \Delta(x_0,2R)\neq\emptyset $, assume that $ u_{\va,\lambda}\in H^1(\om(x_0,2R),\mathbb{C}^m) $ is the weak solution of the boundary problem 
\begin{align}
(\mathcal{L}_{\va}-\lambda I)(u_{\va,\lambda})=F+\operatorname{div}(f)\text{ in }\om(x_0,2R)\quad\text{and}\quad 
u_{\va,\lambda}=0\text{ on }\Delta(x_0,2R).\nonumber
\end{align}
If $ \Delta(x_0,2R)=\emptyset $, assume that $ u_{\va,\lambda}\in H^1(B(x_0,2R);\mathbb{C}^m) $ is the weak solution of the interior problem 
\begin{align}
(\mathcal{L}_{\va}-\lambda I)(u_{\va,\lambda})=F+\operatorname{div}(f)\text{ in } B(x_0,2R).\nonumber  
\end{align}
Then for any $ k\in\mathbb{N}_+ $,
\be
\begin{aligned}
\left(\dashint_{\om(x_0,R)}|u_{\va,\lambda}|^2\right)^{\frac{1}{2}}&\leq \frac{C_{k,\theta_0}}{(1+|\lambda|R^2)^{k}}\left(\dashint_{\om(x_0,2R)}|u_{\va,\lambda}|^2\right)^{\frac{1}{2}}\\
&\quad+\frac{C_{k,\theta_0}R}{(1+|\lambda|R^2)^{\frac{1}{2}}}\left\{R\left(\dashint_{\om(x_0,2R)}|F|^q\right)^{\frac{1}{q}}+\left(\dashint_{\om(x_0,2R)}|f|^2\right)^{\frac{1}{2}}\right\},
\end{aligned}\label{Cau1}
\ee
\be
\begin{aligned}
\left(\dashint_{\om(x_0,R)}|\nabla u_{\va,\lambda}|^2\right)^{\frac{1}{2}}&\leq \frac{C_{k,\theta_0}}{(1+|\lambda|R^2)^{k}R}\left(\dashint_{\om(x_0,2R)}|u_{\va,\lambda}|^2\right)^{\frac{1}{2}}\\
&\quad+C_{k,\theta_0}\left\{R\left(\dashint_{\om(x_0,2R)}|F|^q\right)^{\frac{1}{q}}+\left(\dashint_{\om(x_0,2R)}|f|^2\right)^{\frac{1}{2}}\right\},
\end{aligned}\label{Cau2}
\ee
where $ C_{k,\theta_0} $ depends only on $ \mu,d,m,q,k,\theta_0,\om $.
\end{lem}

\begin{proof}
We will only prove the case for $ d\geq 3 $, since the case $ d=2 $ follows from almost the same methods by substituting the Sobolev embedding theorem $ H_0^1\subset L^{\frac{2d}{d-2}} $ with $ H_0^1\subset L^{q'} $, $ 1<q<\infty $. Firstly, we can assume that $ R=1 $ and $ x_0=0 $ and for general case, $ \eqref{Cau1} $ and $ \eqref{Cau2} $ follows by rescaling and translation. For $ \lambda =0 $, the results are obvious by standard Caccioppoli's inequality for the operator $ \mathcal{L}_{\va} $. Then we can assume that $ \lambda\neq 0 $. Choose $ \varphi\in C_0^{\infty}(B(0,2);\R) $ such that $ \varphi\equiv 1 $ for $ x\in B(0,1) $, $ \varphi\equiv 0 $ for $ x\in B(0,\frac{3}{2})^c $, $ |\nabla\varphi|\leq C $ and $ 0\leq\varphi\leq 1 $. We can set $ \psi=\varphi^2u_{\va,\lambda} $ as the test function. By applying the definition of weak solutions, we have
\begin{align}
\int_{\om_2}A_{\va}\nabla u_{\va,\lambda}\overline{\nabla(\varphi^2u_{\va,\lambda})}-\lambda \int_{\om_2}\varphi^2|u_{\va,\lambda}|^2=\int_{\om_2}F\overline{\varphi^2u_{\va,\lambda}}-\int_{\om_2}f\overline{\nabla(\varphi^2u_{\va,\lambda})},\label{ww1}
\end{align}
where $ A_{\va}=A(x/\va) $ if $ \va>0 $ and $ A_{0}=\widehat{A} $. In the following calculations, we use $ C_{\theta_0} $ to denote a constant depending on $ \theta_0 $ and $ C_{k,\theta_0} $ a constant depending on $ \theta_0 $ and $ k $. If $ \operatorname{Re}\lambda\geq 0 $, we can take the imaginary parts of both sides of $ \eqref{ww1} $, then 
\be
\begin{aligned}
&|\operatorname{Im}\lambda|\int_{\om_2}\varphi^2|u_{\va,\lambda}|^2\\
&\quad\quad\leq C\left\{\int_{\om_2}\varphi|\nabla u_{\va,\lambda}||\nabla\varphi||u_{\va,\lambda}|+\int_{\om_2}\varphi^2|F||u_{\va,\lambda}|+\int_{\om_2}\varphi|f|\left(|\nabla\varphi||u_{\va,\lambda}|+\varphi|\nabla u_{\va,\lambda}|\right)\right\}.
\end{aligned}\label{ww42}
\ee
Noticing that $ |\lambda|=c(\lambda,\theta)|\operatorname{Im}\lambda| $, it actually means that
\begin{align}
\int_{\om_2}\varphi^2|u_{\va,\lambda}|^2&\leq \frac{C_{\theta_{0}}}{|\lambda|}\left\{\int_{\om_2}|\nabla\varphi|^2|u_{\va,\lambda}|^2+\delta\int_{\om_2}\varphi^2|\nabla u_{\va,\lambda}|^2dx+\int_{\om_2}|\varphi F||\varphi u_{\va,\lambda}|+\int_{\om_2}\varphi^2|f|^2\right\},\nonumber
\end{align}
where we have used Cauchy's inequality and the following inequality
\begin{align}
ab\leq \delta a^2+\frac{1}{4\delta}b^2\text{ for any }\delta>0\text{ and }a,b>0,\label{inte}
\end{align}
with $ \delta>0 $ being sufficiently small. By using Hölder's inequality, it is easy to see that
\be
\begin{aligned}
\int_{\om_2}\varphi^2|u_{\va,\lambda}|^2&\leq \frac{C_{\theta_{0}}}{|\lambda|}\left\{\int_{\om_2}|\nabla\varphi|^2|u_{\va,\lambda}|^2+\delta\int_{\om_2}\varphi^2|\nabla u_{\va,\lambda}|^2\right.\\
&\quad\left.+\left(\int_{\om_2}|\varphi F|^{\frac{2d}{d+2}}\right)^{\frac{d+2}{2d}}\left(\int_{\om_2}|\varphi u_{\va,\lambda}|^{\frac{2d}{d-2}}\right)^{\frac{d-2}{2d}}+\int_{\om_2}\varphi^2|f|^2\right\}.
\end{aligned}\label{ww6}
\ee
If $ \operatorname{Re}\lambda<0 $, we can take real parts of both sides of $ \eqref{ww1} $ and conclude that
\be
\begin{aligned}
&|\operatorname{Re}\lambda|\int_{\om_2}\varphi^2|u_{\va,\lambda}|^2\\
&\quad\quad\leq C\left\{\int_{\om_2}\varphi|\nabla u_{\va,\lambda}||\nabla\varphi||u_{\va,\lambda}|+\int_{\om_2}\varphi^2|F||u_{\va,\lambda}|+\int_{\om_2}\varphi|f|\left(|\nabla\varphi||u_{\va,\lambda}|+\varphi|\nabla u_{\va,\lambda}|\right)\right\}.
\end{aligned}\label{ww43}
\ee
By adding $ \eqref{ww42} $ to $ \eqref{ww43} $, we can also derive $ \eqref{ww6} $ for the case $ \operatorname{Re}\lambda<0 $. Then $ \eqref{ww6} $ is true for any $ \lambda\in\Sigma_{\theta_0} $. Owing to $ H_0^1(\om_2)\subset L^{\frac{2d}{d-2}}(\om_2) $ and $ \eqref{inte} $, it can be seen that
\begin{align}
&\left(\int_{\om_2}|\varphi F|^{\frac{2d}{d+2}}\right)^{\frac{d+2}{2d}}\left(\int_{\om_2}|\varphi u_{\va,\lambda}|^{\frac{2d}{d-2}}\right)^{\frac{d-2}{2d}}\leq C\left(\int_{\om_2}|\varphi F|^{\frac{2d}{d+2}}\right)^{\frac{d+2}{d}}+\delta\left(\int_{\om_2}|\varphi u_{\va,\lambda}|^{\frac{2d}{d-2}}\right)^{\frac{d-2}{d}}\nonumber\\
&\quad\quad\quad\quad\quad\quad\leq C\left(\int_{\om_2}|\varphi F|^{\frac{2d}{d+2}}\right)^{\frac{d+2}{d}}+\delta\left(\int_{\om_2}(|\nabla\varphi u_{\va,\lambda}|+|\varphi\nabla u_{\va,\lambda}|)^{2}\right)^{\frac{1}{2}}.\nonumber
\end{align}
This, together with $ \eqref{ww6} $ and the properties of $ \varphi $, leads to
\begin{align}
\left(\int_{\om_1}|u_{\va,\lambda}|^2\right)^{\frac{1}{2}}\leq \frac{C_{\theta_0}}{|\lambda|^{\frac{1}{2}}}\left\{\left(\int_{\om_2}|u_{\va,\lambda}|^2\right)^{\frac{1}{2}}+\left(\int_{\om_2}|F|^{\frac{2d}{d+2}}\right)^{\frac{d+2}{2d}}+\left(\int_{\om_2}|f|^2\right)^{\frac{1}{2}}\right\}.\label{ww5}
\end{align}
Combining $ \eqref{ww1} $ and $ \eqref{ww5} $, we have
\begin{align}
\left(\int_{\om_1}|\nabla u_{\va,\lambda}|^2\right)^{\frac{1}{2}}\leq C_{\theta_0}\left\{\left(\int_{\om_2}|u_{\va,\lambda}|^2\right)^{\frac{1}{2}}+\left(\int_{\om_2}|F|^{\frac{2d}{d+2}}\right)^{\frac{d+2}{2d}}+\left(\int_{\om_2}|f|^2\right)^{\frac{1}{2}}\right\}.\label{ww4}
\end{align}
In view of the obvious fact that $ \|u_{\va,\lambda}\|_{L^2(\om_1)}\leq C\|u_{\va,\lambda}\|_{L^2(\om_2)} $, we can deduce from $ \eqref{ww5} $ that for any $ \lambda\in\Sigma_{\theta_0} $,
\begin{align}
\left(\int_{\om_1}|u_{\va,\lambda}|^2\right)^{\frac{1}{2}}\leq \frac{C_{\theta_0}}{(1+|\lambda|)^{\frac{1}{2}}}\left\{\left(\int_{\om_2}|u_{\va,\lambda}|^2\right)^{\frac{1}{2}}+\left(\int_{\om_2}|F|^{\frac{2d}{d+2}}\right)^{\frac{d+2}{2d}}+\left(\int_{\om_2}|f|^2\right)^{\frac{1}{2}}\right\}.\nonumber
\end{align}
This, together with $ \eqref{ww4} $, gives the proof by repeating the procedure for $ 2k $ times.
\end{proof}

\begin{rem}
Choose $ \varphi\in C_0^{\infty}(B(0,2R);\R) $ such that $ \varphi\equiv 1 $ for $ x\in B(x_0,R) $, $ \varphi\equiv 0 $ for $ x\in B(x_0,\frac{3}{2}R)^c $, $ |\nabla\varphi|\leq C/R $ and $ 0\leq\varphi\leq 1 $. According to $ \eqref{ww42} $, $ \eqref{inte} $ and $ \eqref{ww43} $, it can be easily seen that, if $ \lambda\neq 0 $, $ f\equiv 0 $ and $ F\equiv 0 $, then
\begin{align}
\int_{\om(x_0,2R)}\varphi^2|u_{\va,\lambda}|^2&\leq\frac{C_{\theta_0}}{|\lambda|}\int_{\om(x_0,2R)}\varphi|\nabla u_{\va,\lambda}||\nabla\varphi||u_{\va,\lambda}|\nonumber\\
&\leq \frac{C_{\theta_0}}{|\lambda|^2}\int_{\om(x_0,2R)}|\nabla\varphi|^2|\nabla u_{\va,\lambda}|^2+\frac{1}{2}\int_{\om(x_0,2R)}\varphi^2|u_{\va,\lambda}|^2,\nonumber
\end{align}
which, together with the properties of $ \varphi $, implies that
\begin{align}
\left(\dashint_{\om(x_0,R)}|u_{\va,\lambda}|^2\right)^{\frac{1}{2}}\leq\frac{C_{\theta_0}}{|\lambda|}\left(\dashint_{\om(x_0,2R)}|\nabla\varphi|^2|\nabla u_{\va,\lambda}|^2\right)^{\frac{1}{2}}\leq \frac{C_{\theta_0}}{|\lambda|R}\left(\dashint_{\om(x_0,2R)}|\nabla u_{\va,\lambda}|^2\right)^{\frac{1}{2}}.\label{fanxiangzy}
\end{align}
\end{rem}

\subsection{$ \operatorname{BMO} $ and Hardy space}
To deal with the Green functions with $ d=2 $, we will need some basic knowledge on real analysis, mainly about $ \BMO $ space and Hardy space. 
\begin{defn}[$ \operatorname{BMO} $ space and atom functions]\label{BMO atom}
Let $ d\geq 2 $. Assume that $ x_0\in\mathbb{R}^d $, $ r>0 $ and $ \om $ is a bounded domain in $ \mathbb{R}^d $. Denote $ \om(x_0,r)=\om\cap B(x_0,r) $ as before. $ \operatorname{BMO}(\om;\mathbb{C}^m) $ is a space containing all measurable, $ \mathbb{C}^m $-valued functions such that
\begin{align}
\left\|u\right\|_{\BMO(\om)}=\sup\left\{\dashint_{\om(x_0,r)}|u-u_{x_0,r}|:x_0\in\overline{\om},r>0\right\}<\infty,\label{BMOmo}
\end{align}
where
\begin{align}
u_{x_0,r}:=\left\{\begin{array}{ccc}
0 & \text{ if }  & r\geq \delta(x_0)=\operatorname{dist}(x_0,\pa\om), \\
\dashint_{\om(x_0,r)}u & \text{ if } & r< \delta(x_0)=\operatorname{dist}(x_0,\pa\om).
\end{array}\right.\label{ww35}
\end{align}
We call a bounded measurable function $ a(x) $ as an atom function in $ \om $ if $ \operatorname{supp}(a)\subset\om(x_0,r) $ with $ x_0\in\overline{\om} $, $ 0<r<R_0 $ and
\begin{align}
\left\|a\right\|_{L^{\infty}(\om)}\leq\frac{1}{|\om(x_0,r)|},\quad a_{x_0,r}=0. \nonumber
\end{align}
\end{defn}

\begin{defn}[Hardy space] Let $ \om $ be a $ C^1 $ domain in $ \mathbb{R}^d $ with $ d\geq 2 $. A function $ f $ is an element in the Hardy space $ \mathcal{H}^1(\om;\mathbb{C}^m) $, if there exist a sequence of atoms $ \left\{a_i\right\}_{i=1}^{\infty} $ and a sequence of complex numbers $ \left\{\eta_{i}\right\}_{i=1}^{\infty}\subset l^1(\mathbb{C}) $, such that $ f(x)=\sum_{i=1}^{\infty}\eta_ia_i(x) $. We define the norm in this space by
\begin{align}
\left\|f\right\|_{\mathcal{H}^1(\om)}=\inf\left\{\sum_{i=1}^{\infty}|\eta_i|:f(x)=\sum_{i=1}^{\infty}\eta_ia_i(x)\right\}.\nonumber
\end{align}
We notice the expression
\begin{align}
\sup\left\{\int_{\om}a(y)u(y)dy:a=a(x)\text{ is an atom in }\om\right\},\nonumber
\end{align}
gives the equivalent norm of $ \operatorname{BMO}(\om) $. This is because that $ \operatorname{BMO}(\om) $ is the dual space of $ \mathcal{H}^1 $.
\end{defn}

\section{Uniform regularity estimates}\label{W1pand Lipschitz estimates for}
It is well-known that for the homogenization problem with real vector values, the $ W^{1,p} $ estimates for $ u_{\va} $ are established (see \cite{Av1} and \cite{Shen3}). By simple observations, we can obtain similar results for the problem with complex vector values.
 
\begin{lem}[$ W^{1,p} $ estimates for the operator $ \mathcal{L}_{\va} $]
For $\va\geq 0 $ and $ d\geq 2 $, let $ 2<p<\infty $ and $ \om $ be a bounded $ C^1 $ domain in $ \R^d $. Suppose that $ A $ satisfies $ \eqref{sy} $, $ \eqref{el} $, $ \eqref{pe} $ and $ \eqref{VMO} $. Assume that $ f=(f_{i}^{\al})\in L^p(\om;\mathbb{C}^{m\times d}) $ and $ F\in L^q(\om;\mathbb{C}^m) $ with $ q=\frac{pd}{p+d} $. Then the weak solution of $ \mathcal{L}_{\va}(u_{\va})=F+\operatorname{div}(f) $ in $ \om $ and $ u_{\va}=0 $ on $ \pa\om $ satisfies the uniform estimate
\begin{align}
\|\nabla u_{\va}\|_{L^p(\om)}\leq C\left\{\|f\|_{L^p(\om)}+\|F\|_{L^q(\om)}\right\},\label{W1pom}    
\end{align}
where $ C $ depends only on $ \mu,d,m,p,q,\omega(t) $ and $ \om $.
\end{lem}
\begin{proof}
For the case that $ f\in L^p(\om;\R^{m\times d}) $ and $ F\in L^q(\om;\R^m) $, the proof is trivial. For $ f\in L^p(\om;\mathbb{C}^{m\times d}) $ and $ F\in L^q(\om;\mathbb{C}^m) $, we can write that $ f=g+ih $ and $ F=H+iG $ with $ g,h\in L^p(\om;\R^{m\times d}) $, $ G,H\in L^q(\om;\R^m) $ and use the $ W^{1,p} $ estimates for the case of real functions.
\end{proof}

\begin{thm}[Localization of $ W^{1,p} $ estimates for the operator $ \mathcal{L}_{\va}-\lambda I $]\label{LoW1p}
For $ \va\geq 0 $ and $ d\geq 2 $, let $ \lambda\in\Sigma_{\theta_0}\cup\{0\} $ with $ \theta_0\in(0,\frac{\pi}{2}) $, $ \om $ be a bounded $ C^1 $ domain in $ \R^d $ and $ 2<p<\infty $. Suppose that $ A $ satisfies $ \eqref{sy} $, $ \eqref{el} $, $ \eqref{pe} $ and $ \eqref{VMO} $. Assume that $ x_0\in\om $, $ 0<R<R_0 $, $ f\in L^p(\om(x_0,2R);\mathbb{C}^{m\times d}) $, $ F\in L^q(\om(x_0,2R);\mathbb{C}^m)$ and $ q=\frac{pd}{p+d} $. If $ \Delta(x_0,2R)\neq\emptyset $, assume that $ u_{\va,\lambda}\in H^1(\om(x_0,2R);\mathbb{C}^m) $ is the weak solution of the boundary problem
\begin{align}
(\mathcal{L}_{\va}-\lambda I)(u_{\va,\lambda})=F+\operatorname{div}(f)\text{ in }\om(x_0,2R)\quad\text{and}\quad 
u_{\va,\lambda}=0\text{ on }\Delta(x_0,2R).\nonumber
\end{align}
If $ \Delta(x_0,2R)=\emptyset $, assume that $ u_{\va,\lambda}\in H^1(B(x_0,2R);\mathbb{C}^m) $ is the weak solution of the interior problem 
\begin{align}
(\mathcal{L}_{\va}-\lambda I)(u_{\va,\lambda})=F+\operatorname{div}(f)\text{ in } B(x_0,2R).\nonumber  
\end{align}
Then there exists $ n\in\mathbb{N}_+ $, a constant integer depending only on $ d $, such that for any $ k\in\mathbb{N}_+ $,
\be 
\begin{aligned}
\left(\dashint_{\om(x_0,R)}|\nabla u_{\va,\lambda}|^p\right)^{\frac{1}{p}}&\leq \frac{C_{k,\theta_0}}{(1+|\lambda|R^2)^{k}R}\left(\dashint_{\om(x_0,2R)}|u_{\va,\lambda}|^2\right)^{\frac{1}{2}}\\
&\quad+C_{k,\theta_0}(1+|\lambda|R^2)^{n}\left\{R\left(\dashint_{\om(x_0,2R)}|F|^q\right)^{\frac{1}{q}}+\left(\dashint_{\om(x_0,2R)}|f|^p\right)^{\frac{1}{p}}\right\},
\end{aligned}\label{loW1pnu}
\ee
\be
\begin{aligned}
\left(\dashint_{\om(x_0,R)}|u_{\va,\lambda}|^p\right)^{\frac{1}{p}}&\leq \frac{C_{k,\theta_0}}{(1+|\lambda|R^2)^{k}}\left(\dashint_{\om(x_0,2R)}|u_{\va,\lambda}|^2\right)^{\frac{1}{2}}\\
&\quad+C_{k,\theta_0}(1+|\lambda|R^2)^{n}\left\{R^2\left(\dashint_{\om(x_0,2R)}|F|^q\right)^{\frac{1}{q}}+R\left(\dashint_{\om(x_0,2R)}|f|^p\right)^{\frac{1}{p}}\right\},
\end{aligned}\label{loW1pu}
\ee
where $ C_{k,\theta_0} $ depends only on $ \mu,d,m,k,\theta_0,p,\omega(t) $ and $\om $.
\end{thm}
\begin{proof}
By rescaling and translation, it can be assumed that $ R=1 $ and $ x_0=0 $. We can choose $ \varphi\in C_0^{\infty}(B(0,2);\R) $ such that $ 0\leq\varphi\leq 1 $, $ \varphi\equiv 1 $ in $ B(0,1) $, $ \varphi\equiv 0 $ in $ B(0,\frac{3}{2})^c $ and $ |\nabla\varphi|\leq C $. Then by setting $ A_{\va}$ as $ A(x/\va) $ if $ \va>0 $ and $ \widehat{A} $ if $ \va=0 $, we have
\begin{align}
\mathcal{L}_{\va}(\varphi u_{\va,\lambda})=F_{\va}+\operatorname{div}(f_{\va})\text{ in }\om_2\quad\text{and}\quad \varphi u_{\va,\lambda}=0\text{ on } \pa(\om_2),\nonumber
\end{align}
where $ F_{\va},f_{\va} $ are defined by
\begin{align}
F_{\va}=\lambda \varphi u_{\va,\lambda}+\varphi F-f\nabla\varphi-A_{\va}\nabla u_{\va,\lambda}\nabla\varphi\quad\text{and}\quad f_{\va}=\varphi f-A_{\va}\nabla \varphi u_{\va,\lambda}.\nonumber
\end{align} 
Then owing to $ \eqref{W1pom} $ and Hölder's inequality, it follows that
\begin{align}
&\|\nabla(\varphi u_{\va,\lambda})\|_{L^p(\om_{3/2})}\leq C\left\{\|F_{\va}\|_{L^q(\om_{3/2})}+\|f_{\va}\|_{L^p(\om_{3/2})}\right\}\nonumber\\
&\quad\quad\leq C\left\{|\lambda|\|u_{\va,\lambda}\|_{L^q(\om_{3/2})}+\|F\|_{L^q(\om_{3/2})}+\|u_{\va,\lambda}\|_{L^p(\om_{3/2})}+\|\nabla u_{\va,\lambda}\|_{L^q(\om_{3/2})}+\|f\|_{L^p(\om_{3/2})}\right\},\nonumber
\end{align}
where $ q=\frac{pd}{p+d} $. By using Hölder's inequality again, this implies that
\begin{align}
\|\nabla u_{\va,\lambda}\|_{L^p(\om_1)}&\leq C\left\{(1+|\lambda|)\|u_{\va,\lambda}\|_{L^p(\om_{3/2})}+\|F\|_{L^q(\om_{3/2})}+\|\nabla u_{\va,\lambda}\|_{L^q(\om_{3/2})}+\|f\|_{L^p(\om_{3/2})}\right\}\nonumber\\
&\leq C\left\{(1+|\lambda|)\left(\|\nabla u_{\va,\lambda}\|_{L^q(\om_2)}+\|u_{\va,\lambda}\|_{L^2(\om_2)}\right)+\|F\|_{L^q(\om_2)}+\|f\|_{L^p(\om_2)}\right\}.\nonumber
\end{align}
Here, we have used the inequality that for any $ 2<p<\infty $ and $ u\in W^{1,q}(\om;\mathbb{C}^m) $ with $ q=\frac{pd}{p+d} $,
\begin{align}
\|u\|_{L^p(\om)}\leq C\|u\|_{W^{1,q}(\om)}\leq C\left\{\|\nabla u\|_{L^q(\om)}+\|u\|_{L^2(\om)}\right\},\label{aux2}
\end{align}
where $ C $ depends only on $ p,d,m,\om $. One can refer to \cite{RA} for details about $ \eqref{aux2} $. By iterating for finite times (depending only on $ d $), we can get $ n=n(d)\in\mathbb{N}_+ $, such that
\begin{align}
\|\nabla u_{\va,\lambda}\|_{L^p(\om_1)}\leq C(1+|\lambda|)^{n}\|u_{\va,\lambda}\|_{L^2(\om_2)}+C(1+|\lambda|)^{n-1}\left\{\|F\|_{L^q(\om_2)}+\|f\|_{L^p(\om_2)}\right\}.\nonumber  
\end{align}
In view of $ \eqref{Cau1} $, this yields that for any $ k\in\mathbb{N}_+ $,
\begin{align}
\|\nabla u_{\va,\lambda}\|_{L^p(\om_1)}&\leq \frac{C_{k,\theta_0}(1+|\lambda|)^{n}}{(1+|\lambda|)^{k+n}}\|u_{\va,\lambda}\|_{L^2(\om_2)}+C_{k,\theta_0}(1+|\lambda|)^{n}\left\{\|F\|_{L^q(\om_2)}+\|f\|_{L^p(\om_2)}\right\}\nonumber\\
&\leq C_{k,\theta_0}(1+|\lambda|)^{-k}\|u_{\va,\lambda}\|_{L^2(\om_2)}+C_{k,\theta_0}(1+|\lambda|)^{n}\left\{\|F\|_{L^q(\om_2)}+\|f\|_{L^p(\om_2)}\right\}.\nonumber
\end{align}
This inequality implies $ \eqref{loW1pnu} $. On the other hand, $ \eqref{loW1pu} $ follows from Poincaré's inequality and the same arguments of localization.
\end{proof}

By using Sobolev embedding theorem and convex arguments (see \cite{Shen4}), we can use $ W^{1,p} $ estimates to obtain the Hölder and $ L^{\infty} $ estimates as follows. The proofs are trivial and for more details, one can refer to \cite{Xu1}.

\begin{cor}[Localization of Hölder and $ L^{\infty} $ estimates for the operator $ \mathcal{L}_{\va}-\lambda I $]
For $ \va\geq 0 $ and $ d\geq 2 $, let $ \lambda\in\Sigma_{\theta_0}\cup\{0\} $ with $ \theta_0\in(0,\frac{\pi}{2}) $ and $ \om $ be a bounded $ C^1 $ domain in $ \R^d $. Suppose that $ A $ satisfies $ \eqref{sy} $, $ \eqref{el} $, $ \eqref{pe} $ and $ \eqref{VMO} $. Assume that $ x_0\in \om $, $ 0<R<R_0 $, $ f\in L^p(\om(x_0,2R);\mathbb{C}^{m\times d})$, $ F\in L^q(\om(x_0,2R);\mathbb{C}^m)$ and $ q=\frac{pd}{p+d} $. If $ \Delta(x_0,2R)\neq\emptyset $, assume that $ u_{\va,\lambda}\in H^1(\om(x_0,2R);\mathbb{C}^m) $ is the weak solution of the boundary problem
\begin{align}
(\mathcal{L}_{\va}-\lambda I)(u_{\va,\lambda})=F+\operatorname{div}(f)\text{ in }\om(x_0,2R)\quad\text{and}\quad 
u_{\va,\lambda}=0\text{ on }\Delta(x_0,2R).\nonumber
\end{align}
If $ \Delta(x_0,2R)=\emptyset $, assume that $ u_{\va,\lambda}\in H^1(B(x_0,2R);\mathbb{C}^m) $ is the weak solution of the interior problem
\begin{align}
(\mathcal{L}_{\va}-\lambda I)(u_{\va,\lambda})=F+\operatorname{div}(f)\text{ in } B(x_0,2R).\nonumber  
\end{align}
Then there exists $ n\in\mathbb{N}_+ $, a constant integer depending only on $ d $, such that for $ \gamma=1-\frac{d}{p} $, any $ k\in\mathbb{N}_+ $ and $ 0<s<\infty $,
\be
\begin{aligned}
\left[u_{\va,\lambda}\right]_{C^{0,\gamma}(\om(x_0,R))}&\leq \frac{C_{k,\theta_0}}{(1+|\lambda|R^2)^{k}R^{\gamma}}\left(\dashint_{\om(x_0,2R)}|u_{\va,\lambda}|^2\right)^{\frac{1}{2}}\\
&\quad+\frac{C_{k,\theta_0}(1+|\lambda|R^2)^{n}R}{R^{\gamma}}\left\{R\left(\dashint_{\om(x_0,2R)}|F|^q\right)^{\frac{1}{q}}+\left(\dashint_{\om(x_0,2R)}|f|^p\right)^{\frac{1}{p}}\right\},
\end{aligned}\label{Holder}
\ee
\be
\begin{aligned}
\|u_{\va,\lambda}\|_{L^{\infty}(\om(x_0,R))}&\leq \frac{C_{k,\theta_0}}{(1+|\lambda|R^2)^{k}}\left(\dashint_{\om(x_0,2R)}|u_{\va,\lambda}|^s\right)^{\frac{1}{s}}\\
&\quad+C_{k,\theta_0}(1+|\lambda|R^2)^{n}\left\{R^{2}\left(\dashint_{\om(x_0,2R)}|F|^q\right)^{\frac{1}{q}}+R\left(\dashint_{\om(x_0,2R)}|f|^p\right)^{\frac{1}{p}}\right\},
\end{aligned}\label{Linfty}
\ee
where $ C_{k,\theta_0} $ depends only on $ \mu,d,m,k,\theta_0,p,s,\omega(t) $ and $ \om $.
\end{cor}

\begin{thm}[Localization of Lipschitz estimates for the operator $ \mathcal{L}_{\va} $]\label{Lipschitz estimate localization}
For $ \va\geq 0 $ and $ d\geq 2 $, let $ \om $ be a bounded $ C^{1,\eta} $ domain in $ \mathbb{R}^d $ with $ 0<\eta<1 $. Suppose that $ A $ satisfies $ \eqref{el} $, $ \eqref{pe} $ and $ \eqref{Hol} $. Assume that $ x_0\in\om $, $ 0<R<R_0 $ and $ F\in L^p(\om(x_0,2R);\mathbb{C}^m) $ with $ p>d $. If $ \Delta(x_0,2R)\neq\emptyset $, assume that $ u_{\va}\in H^1(\om(x_0,2R);\mathbb{C}^m) $ is the weak solution of the boundary problem
\begin{align}
\mathcal{L}_{\va}(u_{\va})=F\text{ in }\om(x_0,2R)\quad\text{and}\quad u_{\va}=0\text{ on }\Delta(x_0,2R).\nonumber
\end{align}
If $ \Delta(x_0,2R)=\emptyset $, assume that $ u_{\va}\in H^1(B(x_0,2R);\mathbb{C}^m) $ is the weak solution of the interior problem
\begin{align}
\mathcal{L}_{\va}(u_{\va})=F\text{ in } B(x_0,2R).\nonumber
\end{align}
then
\begin{align}
\|\nabla u_{\va}\|_{L^{\infty}(\om(x_0,R))}\leq \frac{C}{R}\left\{\left(\dashint_{\om(x_0,2R)}|u_{\va}|^2\right)^{\frac{1}{2}}+R^2\left(\dashint_{\om(x_0,2R)}|F|^p\right)^{\frac{1}{p}}\right\},\label{LipforL}
\end{align}
where $ C $ depends only on $ \mu,d,m,\tau,\nu,\eta $ and $ \om $.
\end{thm}
\begin{proof}
See \cite{Av2} or Theorem 3.1.1 and Theorem 4.5.1 of \cite{Shen2}.
\end{proof}

\begin{thm}[Localization of Lipschitz estimates for the operator $ \mathcal{L}_{\va}-\lambda I $]
For $ \va\geq 0 $ and $ d\geq 2 $, Let $ \lambda\in\Sigma_{\theta_0}\cup\{0\} $ with $ \theta_0\in(0,\frac{\pi}{2}) $ and $ \om $ be a bounded $ C^{1,\eta} $ domain in $ \mathbb{R}^d $ with $ 0<\eta<1 $. Suppose that $ A $ satisfies $ \eqref{sy} $,$ \eqref{el} $, $ \eqref{pe} $ and $ \eqref{Hol} $. Assume that $ x_0\in\om $, $ 0<R<R_0 $ and $ F\in L^p(\om(x_0,2R);\mathbb{C}^m) $ with $ p>d $. If $ \Delta(x_0,2R)\neq\emptyset $, assume that $ u_{\va,\lambda}\in H^1(\om(x_0,2R);\mathbb{C}^m) $ is the weak solution of the boundary problem
\begin{align}
(\mathcal{L}_{\va}-\lambda I)(u_{\va,\lambda})=F\text{ in }\om(x_0,2R)\quad\text{and}\quad 
u_{\va,\lambda}=0\text{ on }\Delta(x_0,2R).\nonumber
\end{align}
If $ \Delta(x_0,2R)=\emptyset $, assume that $ u_{\va,\lambda}\in H^1(B(x_0,2R);\mathbb{C}^m) $ is the weak solution of the interior problem
\begin{align}
(\mathcal{L}_{\va}-\lambda I)(u_{\va,\lambda})=F\text{ in }B(x_0,2R).\nonumber
\end{align}
Then there exists $ n\in\mathbb{N}_+ $, a constant integer depending only on $ d $, such that for any $ k\in\mathbb{N}_+ $,
\be
\begin{aligned}
\|\nabla u_{\va,\lambda}\|_{L^{\infty}(\om(x_0,R))}&\leq \frac{C_{k,\theta_0}}{(1+|\lambda|R^2)^kR}\left(\dashint_{\om(x_0,2R)}|u_{\va,\lambda}|^2\right)^{\frac{1}{2}}\\
&\quad+C_{k,\theta_0}(1+|\lambda|R^2)^{n}R\left(\dashint_{\om(x_0,2R)}|F|^p\right)^{\frac{1}{p}},
\end{aligned}\label{Lipesu}
\ee
where $ C_{k,\theta_0} $ depends only on $ \mu,d,m,k,\theta_0,\tau,\nu,\eta $ and $ \om $.
\end{thm}
\begin{proof}

In view of $ \eqref{Lipschitz estimate localization} $, we can infer that
\begin{align}
\|\nabla u_{\va,\lambda}\|_{L^{\infty}(\om_R)}&\leq \frac{C}{R}\left\{\left(\dashint_{\om_{3/2R}}|u_{\va,\lambda}|^2\right)^{\frac{1}{2}}+R^2\left(\dashint_{\om_{3/2R}}|F|^p\right)^{\frac{1}{p}}+|\lambda|R^2\left(\dashint_{\om_{3/2R}}|u_{\va,\lambda}|^p\right)^{\frac{1}{p}}\right\}.\nonumber
\end{align}
By applying $ \eqref{loW1pu} $ to $ u_{\va,\lambda} $ for index $ p $, this yields that for any $ k\in\mathbb{N}_+ $,
\begin{align}
\|\nabla u_{\va,\lambda}\|_{L^{\infty}(\om_R)}&\leq \frac{C}{R}\left(\dashint_{\om_{3/2R}}|u_{\va,\lambda}|^2\right)^{\frac{1}{2}}+CR\left(\dashint_{\om_{3/2R}}|F|^p\right)^{\frac{1}{p}}\nonumber\\
&\quad+\frac{C_{k,\theta_0}}{(1+|\lambda|R^2)^{k-1}}\left\{\left(\dashint_{\om_{2R}}|u_{\va,\lambda}|^p\right)^{\frac{1}{2}}+(1+|\lambda|R^2)^{k+n}R^2\left(\dashint_{\om_{2R}}|F|^p\right)^{\frac{1}{p}}\right\}.\nonumber
\end{align}
This, together with $ \eqref{Cau1} $, implies $ \eqref{Lipesu} $.
\end{proof}

\section{Green functions of operators}\label{Lp estimates of resolventssection}

\subsection{Proof of Theorem \ref{Lp estimates of resolventsf} and relevant estimates}

To begin with, we will use the well-known real variable method to prove Theorem \ref{Lp estimates of resolventsf}. Unlike what had been done in the proof of $ W^{1,p} $ estimates for $ \mathcal{L}_{\va} $, the operators $ \mathcal{L}_{\va}-\lambda I $ do not have the homogeneous property, that is, if $ \lambda\neq 0 $ and $ 0\neq c\in\mathbb{C}^m $ is a constant vector, $ (\mathcal{L}_{\va}-\lambda I)(c)\neq 0 $. For this reason, we need to make some adjustments for the original method employed on $ \mathcal{L}_{\va} $. For simplicity, we use $ B=B(x,r) $ to denote a ball in $ \R^d $ ($ d\geq 2 $) and $ tB=B(x,tr) $ ($ t\in \R_+ $) to denote balls with center $ x $ and radius $ tr $ if no confusion would be caused. 

\begin{thm}[Real variable method]\label{Real variable method}
Let $ q>2 $ and $ \om $ be a bounded Lipschitz domain in $ \R^d $ with $ d\geq 2 $. Let $ F\in L^{2}(\om;\mathbb{C}^m) $ and $ f\in L^{p}(\om;\mathbb{C}^{m\times d}) $ for some $ 2<p<q $. Suppose that for each ball $ B $ with the property that $ |B|\leq c_{0}|\om| $ and either $ 4B\subset\om $ or $ B $ is centered on $ \pa\om $, there exist two measurable functions $ F_{B} $ and $ R_{B} $ on $ \om\cap 2B $, such that $ |F|\leq|F_{B}|+|R_{B}|$ on $ \om\cap 2B $,
\begin{align}
&\left(\dashint_{\om\cap 2B}|R_{B}|^{q}\right)^{\frac{1}{q}} \leq N_{1}\left\{\left(\dashint_{\om \cap 4B}|F|^{2}\right)^{\frac{1}{2}}+\sup _{4 B_{0} \supset B^{\prime} \supset B}\left(\dashint_{\om \cap B^{\prime}}|f|^{2}\right)^{\frac{1}{2}}\right\},\label{rvc1}\\
&\left(\dashint_{\om \cap 2 B}|F_{B}|^{2}\right)^{\frac{1}{2}} \leq N_{2} \sup _{4 B_{0} \supset B^{\prime} \supset B}\left(\dashint_{\om \cap B^{\prime}}|f|^{2}\right)^{\frac{1}{2}}+\eta\left(\dashint_{\om \cap 4B}|F|^{2}\right)^{\frac{1}{2}},\label{rvc2}
\end{align}
where $ N_{1}, N_{2}>0 $ and $ 0<c_{0}<1 $. Then there exists $ \eta_{0}>0 $, depending only on $ N_{1},N_{2},c_{0} $, $ p,q $ and the Lipschitz character of $ \om $, with the property that if $ 0 \leq \eta<\eta_{0} $, then $ F\in L^{p}(\om;\mathbb{C}^m) $ and
\begin{align}
\left(\dashint_{\om}|F|^{p}\right)^{\frac{1}{p}} \leq C\left\{\left(\dashint_{\om}|F|^{2}\right)^{\frac{1}{2}}+\left(\dashint_{\om}|f|^{p}\right)^{\frac{1}{p}}\right\},\label{rer}
\end{align}
where $ C $ depends at most on $ N_{1},N_{2},c_{0}, p,q $ and the Lipschitz character of $ \om $.
\end{thm}
\begin{proof}
It is proved for the case that $ F\in L^2(\om;\R^m) $ and $ f\in L^p(\om;\R^m) $ in Theorem 3.2.6 of \cite{Shen2}. It is easy to generalize it to the complex case through almost the same arguments. 
\end{proof}

\begin{proof}[Proof of Theorem \ref{Lp estimates of resolventsf}]
Firstly, we assume that $ \lambda\neq 0 $, since the case $ \lambda=0 $ is trivial in view of the results for $ W^{1,p} $ estimates of $ \mathcal{L}_{\va} $. If $ p=2 $, the results are given by Lemma \ref{l2reses}. For $ 2<p<\infty $, choose $ q=p+1 $. Consider functions
\begin{align}
&H=|u_{\va,\lambda}|,\,\,h=(R_0^{-2}+|\lambda|)^{-1}|F|+(R_0^{-2}+|\lambda|)^{-\frac{1}{2}}|f|,\,\,G=|\nabla u_{\va,\lambda}|\nonumber\\
&\quad\quad\text{and }g=(R_0^{-2}+|\lambda|)^{-\frac{1}{2}}|F|+|f|.\nonumber
\end{align}
Now we will apply Theorem \ref{Real variable method} to complete the proof. For each ball $ B $ with the property that $ |B|\leq \frac{1}{100}|\om| $ and $ 4B\subset \om $, we write $ u_{\va,\lambda}=u_{\va,\lambda,1}+u_{\va,\lambda,2} $ in $ 2B $, where $ u_{\va,\lambda,1}\in H_0^1(4B;\mathbb{C}^m) $ is the weak solution of $ \mathcal{L}_{\va}(u_{\va,\lambda,1})-\lambda u_{\va,\lambda,1}=F+\operatorname{div}(f) $ in $ 4B $ and $ u_{\va,\lambda,1}=0 $ on $ \pa(4B) $. Let
\begin{align}
H_{B}=|u_{\va,\lambda,1}|,\,\, R_{B}=|u_{\va,\lambda,2}|,\,\, G_{B}=|\nabla u_{\va,\lambda,1}|\text{ and }T_{B}=|\nabla u_{\va,\lambda,2}|.\nonumber
\end{align}
Obviously, $ |H|\leq H_{B}+R_{B} $ and $ |G|\leq G_B+T_B $ in $ 2B $. It follows from Lemma \ref{l2reses} that
\be
\begin{aligned}
\left(\dashint_{4B}|H_{B}|^2\right)^{\frac{1}{2}}&\leq \frac{C_{\theta_0}}{r^{-2}+|\lambda|}\left(\dashint_{4B}|F|^2\right)^{\frac{1}{2}}+\frac{C_{\theta_0}}{(r^{-2}+|\lambda|)^{\frac{1}{2}}}\left(\dashint_{4B}|f|^2\right)^{\frac{1}{2}}\\
&\leq C_{\theta_0}\left\{\dashint_{4B}\left(\frac{|F|+(R_0^{-2}+|\lambda|)^{\frac{1}{2}}|f|}{R_0^{-2}+|\lambda|}\right)^2\right\}^{\frac{1}{2}}\leq C_{\theta_0}\left(\dashint_{4B}|h|^2\right)^{\frac{1}{2}}
\end{aligned}\label{HBestimates}
\ee
and
\be
\begin{aligned}
\left(\dashint_{4B}|G_{B}|^2\right)^{\frac{1}{2}}&\leq \frac{C_{\theta_0}}{(r^{-2}+|\lambda|)^{\frac{1}{2}}}\left(\dashint_{4B}|F|^2\right)^{\frac{1}{2}}+C_{\theta_0}\left(\dashint_{4B}|f|^2\right)^{\frac{1}{2}}\\
&\leq C_{\theta_0}\left\{\dashint_{4B}\left(\frac{|F|}{(R_0^{-2}+|\lambda|)^{\frac{1}{2}}}+|f|\right)^2\right\}^{\frac{1}{2}}\leq C_{\theta_0}\left(\dashint_{4B}|g|^2\right)^{\frac{1}{2}},
\end{aligned}\label{GBestimates}
\ee
where $ r $ is the radius of $ B $. Moreover, we note that $ u_{\va,\lambda,2}\in H^1(4B;\mathbb{C}^m) $ and $ \mathcal{L}_{\va}(u_{\va,\lambda,2})-\lambda u_{\va,\lambda,2}=0 $ in $ 4B $. Owing to $ \eqref{loW1pu} $, \eqref{HBestimates} and Lemma \ref{l2reses}, we can obtain that
\begin{align}
\left(\dashint_{2B}|R_B|^q\right)^{\frac{1}{q}}&\leq C_{\theta_0}\left(\dashint_{4B}|u_{\va,\lambda,2}|^2\right)^{\frac{1}{2}}\leq C_{\theta_0}\left\{\left(\dashint_{4B}|u_{\va,\lambda}|^2\right)^{\frac{1}{2}}+\left(\dashint_{4B}|u_{\va,\lambda,1}|^2\right)^{\frac{1}{2}}\right\}\nonumber\\
&\leq C_{\theta_0}\left\{\left(\dashint_{4B}|H|^2\right)^{\frac{1}{2}}+\left(\dashint_{4B}|h|^2\right)^{\frac{1}{2}}\right\}.\nonumber
\end{align}
To estimate $ T_B $, we first note that
\begin{align}
(\mathcal{L}_{\va}-\lambda I)\left(u_{\va,\lambda,2}-\dashint_{3B}u_{\va,\lambda,2}\right)=\lambda\left(\dashint_{3B}u_{\va,\lambda,2}\right)\text{ in }3B.\nonumber
\end{align}
Then it is easy to get that
\begin{align}
\left(\dashint_{2B}|T_B|^q\right)^{\frac{1}{q}}&\leq \frac{C_{\theta_0}}{r}\left(\dashint_{3B}\left|u_{\va,\lambda,2}-\dashint_{3B}u_{\va,\lambda,2}\right|^2\right)^{\frac{1}{2}}+C_{\theta_0}(1+|\lambda|r^2)^n|\lambda|r\left(\dashint_{3B}|u_{\va,\lambda,2}|\right)\nonumber\\
&\leq C_{\theta_0}\left\{\left(\dashint_{3B}|\nabla u_{\va,\lambda,2}|^2\right)^{\frac{1}{2}}+|\lambda|r\left(\dashint_{7/2 B}|u_{\va,\lambda,2}|^2\right)^{\frac{1}{2}}\right\}\leq C_{\theta_0}\left(\dashint_{4B}|\nabla u_{\va,\lambda,2}|^2\right)^{\frac{1}{2}}\nonumber\\
&\leq C_{\theta_0}\left\{\left(\dashint_{4B}|\nabla u_{\va,\lambda,1}|^2\right)^{\frac{1}{2}}+\left(\dashint_{4B}|\nabla u_{\va,\lambda}|^2\right)^{\frac{1}{2}}\right\}\leq C_{\theta_0}\left\{\left(\dashint_{4B}|G|^2\right)^{\frac{1}{2}}+\left(\dashint_{4B}|g|^2\right)^{\frac{1}{2}}\right\},\nonumber
\end{align}
where for the second inequality, we have used $ \eqref{Cau1} $ with $ k=n $, Poincaré's inequality, Hölder's inequality and for the third inequality, we have used $ \eqref{fanxiangzy} $.

For ball $ B=B(x,r) $ such that it is centered at $ \pa\om $ with $ 0<r<\frac{1}{16}R_0 $, write
$ u_{\va,\lambda}=u_{\va,\lambda,3}+u_{\va,\lambda,4} $ in $ 4B\cap\om $, where $ u_{\va,\lambda,3}\in H_0^1(4B;\mathbb{C}^m) $ is the weak solution of $ \mathcal{L}_{\va}(u_{\va,\lambda,3})-\lambda u_{\va,\lambda,3}=F+\operatorname{div}(f) $ in $ 4B\cap\om $ and $ u_{\va,\lambda,3}=0 $ on $ \pa(4B\cap\om) $. Let
\begin{align}
H_{B}=|u_{\va,\lambda,3}|,\,\, R_{B}=|u_{\va,\lambda,4}|,\,\, G_{B}=|\nabla u_{\va,\lambda,3}|\text{ and }T_{B}=|\nabla u_{\va,\lambda,4}|.\nonumber
\end{align}
By using almost the same arguments, we can obtain $ |H|\leq H_{B}+R_{B} $, $ |G|\leq G_B+T_B $ in $ 2B\cap\om $,
\begin{align}
\left(\dashint_{2B\cap\om}|R_B|^q\right)^{\frac{1}{q}}&\leq C_{\theta_0}\left(\dashint_{4B\cap\om}|H|^2\right)^{\frac{1}{2}}+C_{\theta_0}\left(\dashint_{4B\cap\om}|h|^2\right)^{\frac{1}{2}},\nonumber\\
\left(\dashint_{4B\cap\om}|G_{B}|^2\right)^{\frac{1}{2}}&\leq  C_{\theta_0}\left(\dashint_{4B\cap\om}|g|^2\right)^{\frac{1}{2}}\text{ and }
\left(\dashint_{4B\cap\om}|H_{B}|^2\right)^{\frac{1}{2}}\leq  C_{\theta_0}\left(\dashint_{4B\cap\om}|h|^2\right)^{\frac{1}{2}}.\nonumber
\end{align}
Moreover, since $ u_{\va,\lambda,4}=u_{\va,\lambda}-u_{\va,\lambda,3}=0 $ on $ \pa\om\cap 4B $, then in view of Poincaré's inequality and $ \eqref{loW1pnu} $, we have
\begin{align}
\left(\dashint_{2B\cap\om}|T_B|^q\right)^{\frac{1}{q}}&\leq \frac{C_{\theta_0}}{r}\left(\dashint_{4B\cap\om}|u_{\va,\lambda,4}|^2\right)^{\frac{1}{2}}\leq C_{\theta_0}\left(\dashint_{4B\cap\om}|\nabla u_{\va,\lambda,4}|^2\right)^{\frac{1}{2}}\nonumber\\
&\leq C_{\theta_0}\left\{\left(\dashint_{4B\cap\om}|\nabla u_{\va,\lambda,3}|^2\right)^{\frac{1}{2}}+\left(\dashint_{4B\cap\om}|\nabla u_{\va,\lambda}|^2\right)^{\frac{1}{2}}\right\}\nonumber\\
&\leq C_{\theta_0}\left\{\left(\dashint_{4B\cap\om}|G|^2\right)^{\frac{1}{2}}+\left(\dashint_{4B\cap\om}|g|^2\right)^{\frac{1}{2}}\right\}.\nonumber
\end{align}
Then by using Theorem \ref{Real variable method}, we have, for any $ 2<p<\infty $,
\begin{align}
\left(\dashint_{\om}|u_{\va,\lambda}|^{p}\right)^{\frac{1}{p}}&\leq C_{\theta_0}\left\{\left(\dashint_{\om}|u_{\va,\lambda}|^{2}\right)^{\frac{1}{2}}+\left(\dashint_{\om}|h|^{p}\right)^{\frac{1}{p}}\right\},\nonumber\\
 \left(\dashint_{\om}|\nabla u_{\va,\lambda}|^{p}\right)^{\frac{1}{p}}&\leq C_{\theta_0}\left\{\left(\dashint_{\om}|\nabla u_{\va,\lambda} |^{2}\right)^{\frac{1}{2}}+\left(\dashint_{\om}|g|^{p}\right)^{\frac{1}{p}}\right\}.\nonumber
\end{align}
In view of Lemma \ref{l2reses}, definitions of $ g,h $ and Hölder's inequality, we can complete the proof for the case $ 2<p<\infty $. For $ 1<p<2 $, the results follow form duality arguments. Using the same definitions of $ u_{\va,\lambda}^{(1)} $ and $ u_{\va,\lambda}^{(2)} $ in the proof of Lemma \ref{l2reses}, we only need to show that for any $ 1<p<2 $,
\begin{align}
\|\nabla u_{\va,\lambda}^{(1)}\|_{L^p(\om)}+(R_0^{-2}+|\lambda|)^{\frac{1}{2}}\|u_{\va,\lambda}^{(1)}\|_{L^p(\om)}&\leq C_{\theta_0}(R_0^{-2}+|\lambda|)^{-\frac{1}{2}}\|F\|_{L^p(\om)},\nonumber\\
\|\nabla u_{\va,\lambda}^{(2)}\|_{L^p(\om)}+(R_0^{-2}+|\lambda|)^{\frac{1}{2}}\|u_{\va,\lambda}^{(2)}\|_{L^p(\om)}&\leq C_{\theta_0}\|f\|_{L^p(\om)}.\nonumber
\end{align}
For $ F_1\in L^{p'}(\om;\mathbb{C}^m) $ and $ f_1\in L^{p'}(\om;\mathbb{C}^{m\times d}) $ with $ p'=\frac{p}{p-1} $ being the conjugate number of $ p $, let $ v_{\va,\lambda}^{(1)} $ and $ v_{\va,\lambda}^{(2)} $ be solutions of Dirichlet problems:
\begin{align}
\left\{\begin{matrix}
(\mathcal{L}_{\va}-\overline{\lambda} I)(v_{\va,\lambda}^{(1)})=F_1&\text{in}&\om,\\
v_{\va,\lambda}^{(1)}=0&\text{on}&\pa\om,
\end{matrix}\right.\quad\text{and}\quad \left\{\begin{matrix}
(\mathcal{L}_{\va}-\overline{\lambda} I)(v_{\va,\lambda}^{(2)})=\operatorname{div}(f_1)&\text{in}&\om,\\
v_{\va,\lambda}^{(2)}=0&\text{on}&\pa\om.
\end{matrix}\right.\nonumber
\end{align}
Then it follows from direct calculations and the definition of $ B_{\va,\lambda,\om}[\cdot,\cdot] $ that
\begin{align}
\int_{\om}F_1\overline{u_{\va,\lambda}^{(1)}}dx&=B_{\va,\overline{\lambda},\om}[v_{\va,\lambda}^{(1)},u_{\va,\lambda}^{(1)}]=\overline{B_{\va,\lambda,\om}[u_{\va,\lambda}^{(1)},v_{\va,\lambda}^{(1)}]}=\int_{\om}\overline{F}v_{\va,\lambda}^{(1)}dx,\nonumber\\
-\int_{\om}f_1\overline{\nabla u_{\va,\lambda}^{(1)}}dx&=B_{\va,\overline{\lambda},\om}[v_{\va,\lambda}^{(2)},u_{\va,\lambda}^{(1)}]=\overline{B_{\va,\lambda,\om}[u_{\va,\lambda}^{(1)},v_{\va,\lambda}^{(2)}]}=\int_{\om}\overline{F}v_{\va,\lambda}^{(2)}dx,\nonumber\\
\int_{\om}F_1\overline{u_{\va,\lambda}^{(2)}}dx&=B_{\va,\overline{\lambda},\om}[v_{\va,\lambda}^{(1)},u_{\va,\lambda}^{(2)}]=\overline{B_{\va,\lambda,\om}[u_{\va,\lambda}^{(2)},v_{\va,\lambda}^{(1)}]}=-\int_{\om}\overline{f}\nabla v_{\va,\lambda}^{(1)}dx,\nonumber\\
-\int_{\om}f_1\overline{\nabla u_{\va,\lambda}^{(2)}}dx&=B_{\va,\overline{\lambda},\om}[v_{\va,\lambda}^{(2)},u_{\va,\lambda}^{(2)}]=\overline{B_{\va,\lambda,\om}[u_{\va,\lambda}^{(2)},v_{\va,\lambda}^{(2)}]}=-\int_{\om}\overline{f}\nabla v_{\va,\lambda}^{(2)}dx.\nonumber
\end{align}
These, together with the results for $ 2<p<\infty $, imply that
\begin{align}
\left|\int_{\om}F_1\overline{u_{\va,\lambda}^{(1)}}dx\right|&\leq \|F\|_{L^p(\om)}\|v_{\va,\lambda}^{(1)}\|_{L^{p'}(\om)}\leq C_{\theta_0}(R_0^{-2}+|\lambda|)^{-1}\|F\|_{L^p(\om)}\|F_1\|_{L^{p'}(\om)},\nonumber\\
\left|\int_{\om}f_1\overline{\nabla u_{\va,\lambda}^{(1)}}dx\right|&\leq\|F\|_{L^p(\om)}\|v_{\va,\lambda}^{(2)}\|_{L^{p'}(\om)}\leq C_{\theta_0}(R_0^{-2}+|\lambda|)^{-\frac{1}{2}}\|F\|_{L^p(\om)}\|f_1\|_{L^{p'}(\om)},\nonumber\\
\left|\int_{\om}F_1\overline{u_{\va,\lambda}^{(2)}}dx\right|&\leq \|f\|_{L^p(\om)}\|\nabla v_{\va,\lambda}^{(1)}\|_{L^{p'}(\om)}\leq C_{\theta_0}(R_0^{-2}+|\lambda|)^{-\frac{1}{2}}\|f\|_{L^p(\om)}\|F_1\|_{L^{p'}(\om)},\nonumber\\
\left|\int_{\om}f_1\overline{\nabla u_{\va,\lambda}^{(2)}}dx\right|&\leq \|f\|_{L^p(\om)}\|\nabla v_{\va,\lambda}^{(2)}\|_{L^{p'}(\om)}\leq C_{\theta_0}\|f\|_{L^p(\om)}\|f_1\|_{L^{p'}(\om)},\nonumber
\end{align}
which give the proof.
\end{proof}

\begin{cor}
Assume that $ \om $ is a bounded $ C^{1,1} $ domain in $ \mathbb{R}^d $ with $ d\geq 2 $ and $ \lambda\in\Sigma_{\theta_0}\cup\{0\} $ with $ \theta_0\in (0,\frac{\pi}{2}) $. If $ u_{0,\lambda}\in W^{2,p}(\om;\mathbb{C}^m) $ is the unique weak solution for the Dirichlet problem
\begin{align}
(\mathcal{L}_{0}-\lambda I)(u_{0,\lambda})=F\text{ in }\om\quad\text{and}\quad u_{0,\lambda}=0\text{ on }\pa\om\nonumber
\end{align}
with $ 1<p<\infty $ and $ F\in L^p(\om;\mathbb{C}^m) $. Then
\begin{align}
\|\nabla^2u_{0,\lambda}\|_{L^p(\om)}\leq C_{\theta_0}\|F\|_{L^p(\om)},\label{W2p}
\end{align}
where $ C_{\theta_0} $ depends only on $ \mu,d,m,p,\theta_0 $ and $ \om $. In operator forms,
\begin{align}
\|R(\lambda,\mathcal{L}_0)\|_{L^p(\om)\to W^{2,p}(\om)}\leq C_{\theta_0}.\label{W2pop}
\end{align}
\end{cor}
\begin{proof}
The results follow from the $ W^{2,p} $ estimates for $ \mathcal{L}_0 $, i.e.
\begin{align}
\|\nabla u_{0,\lambda}\|_{L^p(\om)}\leq C\left\{|\lambda|\|u_{0,\lambda}\|_{L^p(\om)}+\|F\|_{L^p(\om)}\right\}\nonumber
\end{align}
and $ \|u_{0,\lambda}\|_{L^p(\om)}\leq C_{\theta_0}(R_0^{-2}+|\lambda|)^{-1}\|F\|_{L^p(\om)} $, which is given by Theorem \ref{Lp estimates of resolventsf}.
\end{proof}

\begin{lem}
Assume that $ \om $ is $ C^{1,1} $ domain in $ \mathbb{R}^d $ with $ d\geq 2 $, $ x_0\in\om $ and $ 0<R<R_0 $. If $ \Delta(x_0,2R)\neq\emptyset $, assume that $ u_{0,\lambda}\in H^2(\om(x_0,2R);\mathbb{C}^m) $ is the weak solution of the boundary problem
\begin{align}
(\mathcal{L}_0-\lambda I)(u_{0,\lambda})=0\text{ in }\om(x_0,2R)\quad\text{and}\quad u_{0,\lambda}=0\text{ on }\Delta(x_0,2R),\nonumber
\end{align}
if $ \Delta(x_0,2R)=\emptyset $, assume that $ u_{0,\lambda}\in H^2(B(x_0,2R);\mathbb{C}^m) $ is the weak solution of the interior problem
\begin{align}
(\mathcal{L}_0-\lambda I)(u_{0,\lambda})=0\text{ in }B(x_0,2R),\nonumber
\end{align}
with $ x_0\in\om $, $ 0<R<R_0 $ and $ 1<p<\infty $. Then there exists $ n\in\mathbb{N}_+ $, a constant depending only on $ d $, such that for any $ k\in\mathbb{N}_+ $,  
\be
\begin{aligned}
\left(\dashint_{\om(x_0,R)}|\nabla^2u_{0,\lambda}|^p\right)^{\frac{1}{p}}&\leq\frac{C_{k,\theta_0}}{(1+|\lambda|R^2)^kR^2}\left(\dashint_{\om(x_0,2R)}|u_{0,\lambda}|^2\right)^{\frac{1}{2}}.
\end{aligned}\label{W2p local}
\ee
If we further assume that $ \om $ is a bounded $ C^{2,1} $ domain and $ \rho\in(0,1) $, then
\begin{align}
\|\nabla^2u_{0,\lambda}\|_{L^{\infty}(\om(x_0,R))}&\leq\frac{C_{k,\theta_0}}{(1+|\lambda|R^2)^kR^{2}}\left(\dashint_{\om(x_0,2R)}|u_{0,\lambda}|^2\right)^{\frac{1}{2}},\label{ww80}\\
\left[\nabla^2u_{0,\lambda}\right]_{C^{0,\rho}(\om(x_0,R))}&\leq\frac{C_{k,\theta_0}}{(1+|\lambda|R^2)^kR^{2+{\rho}}}\left(\dashint_{\om(x_0,2R)}|u_{0,\lambda}|^2\right)^{\frac{1}{2}},\label{ww81}
\end{align}
where $ C $ depends on $ \mu,d,m,k,\theta_0,\rho $ and $ \om $.
\end{lem}
\begin{proof}
By applying $ C^{2,\rho} $ estimates for the constant elliptic system (see Theorem 5.23 in \cite{MG}), $ \eqref{Cau1} $, $ \eqref{Cau2} $, $ \eqref{W2p} $ and Theorem \ref{Lp estimates of resolventsf}. 
\end{proof}

\begin{rem}\label{remimp}
Under conditions that $ A $ satisfies $ \eqref{sy} $, $ \eqref{el} $, $ \eqref{pe} $, $ \eqref{VMO} $, $ \va\geq 0 $ and $ \om $ is a bounded $ C^1 $ domain in $ \R^d $ with $ d=2 $, if $ u_{\va,\lambda} $ is the unique solution for the Dirichlet problem $ (\mathcal{L}_{\va}-\lambda I)(u_{\va,\lambda})=F $ in $ \om $ and $ u_{\va,\lambda}=0 $ on $ \pa\om $, then in view of $ \eqref{W1pom} $ and Theorem \ref{Lp estimates of resolventsf}, we have
\begin{align}
\|\nabla u_{\va,\lambda}\|_{L^p(\om)}\leq C\left\{|\lambda|\|u_{\va,\lambda}\|_{L^{\frac{2p}{p+2}}(\om)}+\|F\|_{L^{\frac{2p}{p+2}}(\om)}\right\}\leq C_{\theta_0}\|F\|_{L^{\frac{2p}{p+2}}(\om)},\label{impoestim}
\end{align}
for $ 2<p<\infty $. If $ v_{\va,\lambda} $ is the unique solution for the Dirichlet problem $ (\mathcal{L}_{\va}-\lambda I)(v_{\va,\lambda})=\operatorname{div}(f) $ in $ \om $ and $ v_{\va,\lambda}=0 $ on $ \pa\om $, it is not hard to see that
\begin{align}
\|v_{\va,\lambda}\|_{L^{\frac{2p}{p-2}}(\om)}\leq C_{\theta_0}\|f\|_{L^{\frac{p}{p-1}}(\om)},\quad \text{for any }2<p<\infty,\label{impoestim2}
\end{align}
Indeed, for any $ F\in C_0^{\infty}(\om;\mathbb{C}^m) $, we can choose $ w_{\va,\lambda} $ such that $ (\mathcal{L}_{\va}-\overline{\lambda} I)(w_{\va,\lambda})=F $ in $ \om $ and $ w_{\va,\lambda}=0 $ on $ \pa\om $. Then it can be easily seen by duality that
\begin{align}
\int_{\om}F\overline{v_{\va,\lambda}}dx&=B_{\va,\overline{\lambda},\om}[w_{\va,\lambda},v_{\va,\lambda}]=\overline{B_{\va,\lambda,\om}[v_{\va,\lambda},w_{\va,\lambda}]}=-\int_{\om}\overline{f}\nabla w_{\va,\lambda}dx\nonumber
\end{align}
In view of Hölder's inequality and $ \eqref{impoestim} $, it can be got that
\begin{align}
\left|\int_{\om}F\overline{v_{\va,\lambda}}dx\right|\leq \|f\|_{L^{\frac{p}{p-1}}(\om)}\|\nabla w_{\va,\lambda}\|_{L^p(\om)}\leq C_{\theta_0}\|f\|_{L^{\frac{p}{p-1}}(\om)}\|F\|_{L^{\frac{2p}{p+2}}(\om)},\nonumber
\end{align}
which directly implies $ \eqref{impoestim2} $. Moreover, for $ F\in \dot{W}^{-1,p'}(\om) $, where $ \dot{W}^{-1,p'}(\om) $ denotes the dual space for the homogeneous Sobolev space $ \dot{W}_0^{1,p}(\om) $ and $ p' $ is the conjugate number of $ p $, we have
\begin{align}
\left\|u_{\va,\lambda}\right\|_{L^{\frac{2p}{p-2}}(\om)}\leq C_{\theta_0}\left\|F\right\|_{\dot{W}^{-1,p'}(\om)}. \label{Dual estimates}
\end{align}
This conclusion is also proved by using the duality arguments. Actually, for all $ g\in C_0^1(\om;\mathbb{C}^m) $ we can choose $ v_{\va,\lambda} $ such that $ (\mathcal{L}_{\va}-\overline{\lambda} I)(v_{\va,\lambda})=g $ in $ \om $ and $ v_{\va,\lambda}=0 $ on $ \pa\om $. Thus,
\begin{align}
\int_{\om}g\overline{u_{\va,\lambda}}dx=B_{\va,\overline{\lambda},\om}[v_{\va,\lambda},u_{\va,\lambda}]=\overline{B_{\va,\lambda,\om}[u_{\va,\lambda},v_{\va,\lambda}]}=\overline{\langle F,v_{\va,\lambda}\rangle_{\dot{W}^{-1,p'}(\om)\times\dot{W}_0^{1,p'}(\om)}}. \nonumber
\end{align}
By using Hölder's inequality and $ \eqref{impoestim} $, this shows that
\begin{align}
\left|\int_{\om}g\overline{u_{\va,\lambda}}dx\right|&=\left|\langle F,v_{\va,\lambda}\rangle_{\dot{W}^{-1,p'}(\om)\times\dot{W}_0^{1,p'}(\om)}\right|\leq \left\|F\right\|_{\dot{W}^{-1,p'}(\om)}\left\|v_{\va,\lambda}\right\|_{\dot{W}_0^{1,p}(\om)}\nonumber\\
&\leq\left\|F\right\|_{\dot{W}^{-1,p'}(\om)}\left\|\nabla v_{\va,\lambda}\right\|_{L^p(\om)}\leq C_{\theta_0}\left\|F\right\|_{\dot{W}^{-1,p'}(\om)}\left\|g\right\|_{L^{\frac{2p}{p+2}}(\om)},\nonumber
\end{align}
which gives the proof of $ \eqref{Dual estimates} $.
\end{rem}

\begin{lem}
Assume that $ A $ satisfies $ \eqref{sy} $, $ \eqref{el} $, $ \eqref{pe} $, $ \eqref{VMO} $, $ \lambda\in\Sigma_{\theta_0}\cup\{0\} $ with $ \theta_0\in(0,\frac{\pi}{2}) $ and $ \om $ is a $ C^{1} $ bounded domain in $ \mathbb{R}^2 $. Let $ a=a(x) $ be an atom function in $ \om $. If $ u_{\va,\lambda} $ is the unique weak solution for the Dirichlet problem $ (\mathcal{L}_{\va}-\lambda I)(u_{\va,\lambda})=a $ in $ \om $ and $ u_{\va,\lambda}=0 $ on $ \pa\om $, then there exists a constant $ C_{\theta_0} $ depending only on $ \mu,\omega(t),m,\theta_0 $ and $ \om $ such that
\begin{align}
\left\|u_{\va,\lambda}\right\|_{L^{\infty}(\om)}\leq C_{\theta_0}.\label{Linftyatom}
\end{align}
\end{lem}
\begin{proof}
For atom function $ a=a(\cdot) $, assume that
\begin{align}
\operatorname{supp}(a)\subset\om(x_0,\rho)\text{ and }\left\|a\right\|_{L^{\infty}(\om)}\leq\frac{1}{|\om(x_0,\rho)|}\nonumber
\end{align}
with $ x_0\in\om $ and $ 0<\rho< R_0 $. Fix $ z\in\om $, we can choose $ 2<p<\infty $. Then
\begin{align}
|u(z)| &\leq  |u(z)-u_{z,\rho}|+|u_{z,\rho}|\leq C\rho^{1-\frac{2}{p}}[u]_{C^{0,1-\frac{2}{p}}(\om(z,\rho))}+|u_{z,\rho}|\nonumber\\
 &\leq C\left\{\rho^{1-\frac{2}{p}}\left\|\nabla u_{\va,\lambda}\right\|_{L^p(\om)}+\rho^{\frac{2}{p}-1}\left\|u_{\va,\lambda}\right\|_{L^{\frac{2p}{p-2}}(\om)}\right\},\nonumber
\end{align}
due to Morrey's inequality and Hölder's inequality. Using $ \eqref{impoestim} $ and $ \eqref{Dual estimates} $, we have
\begin{align}
\left\|\nabla u_{\va,\lambda}\right\|_{L^p(\om)}&\leq C_{\theta_0}\left\|a\right\|_{L^{\frac{2p}{p+2}}(\om)}\leq C_{\theta_0}\rho^{\frac{2}{p}-1}\text{ and } \left\|u_{\va,\lambda}\right\|_{L^{\frac{2p}{p-2}}(\om)}\leq C_{\theta_0}\left\|a\right\|_{\dot{W}^{-1,p'}(\om)}.\nonumber
\end{align}
We claim that $ \left\|a\right\|_{\dot{W}^{-1,p'}(\om)}\leq \rho^{1-\frac{2}{p}} $. It is because that for all $ v\in \dot{W}_0^{1,p}
(\om;\mathbb{R}^m) $,
\begin{align}
\left|\int_{\om}a^{\alpha}(y)v^{\alpha}(y)dy\right|&\leq\left|\int_{\om}a^{\alpha}(y)(v^{\al}(y)-v_{x_0,\rho}^{\al})dy\right|\leq \|a\|_{L^1(\om)}\|v-v_{x_0,\rho}\|_{L^{\infty}(\om(x_0,\rho))}\nonumber\\
&\leq C\rho^{1-\frac{2}{p}}[v]_{C^{0,1-\frac{2}{p}}(\om(x_0,\rho))}\leq C\rho^{1-\frac{2}{p}}\left\|\nabla v\right\|_{L^p(\om)},\nonumber
\end{align}
where we have used Morrey's theory again for the last inequality. This completes the proof.
\end{proof}
\begin{rem}
The key point of $ \eqref{Linftyatom} $ is that $ C_{\theta_0} $ does not depend on the module of $ \lambda $. Such estimates are extremely important in the constructions of Green functions with $ d=2 $. If $ C_{\theta_0} $ depends on the module of $ \lambda $, we will not be able to obtain related estimates similar to the case that $ d\geq 3 $.
\end{rem}

\subsection{Green functions with dimension no less than three}

\begin{thm}[Green functions of $ \mathcal{L}_{\va}-\lambda I $ with $ d\geq 3 $]\label{Green3}
For $ \va\geq 0 $ and $ d\geq 3 $, let $ \lambda\in\Sigma_{\theta_0}\cup\{0\} $ with $ \theta_0\in(0,\frac{\pi}{2}) $ and $ \om $ be a bounded $ C^1 $ domain in $ \R^d $. Suppose that $ A $ satisfies $ \eqref{sy} $, $ \eqref{el} $, $ \eqref{pe} $ and $ \eqref{VMO} $. Then there exists a unique Green function, $ G_{\va,\lambda}=(G_{\va,\lambda}^{\al\beta}(\cdot,\cdot)):\om\times\om\to \mathbb{C}^{m^2}\cup\left\{\infty\right\} $ with $ 1\leq\al,\beta\leq m $, such that $ G_{\va,\lambda}(\cdot,y) \in H^1(\om\backslash B(y,r);\mathbb{C}^{m^2})\cap W_0^{1,s}(\om;\mathbb{C}^{m^2}) $ for any $ s\in [1,\frac{d}{d-1}) $, $ y\in\om $ and $ 0<r<R_0 $. $ G_{\va,\lambda} $ satisfies 
\begin{align}
B_{\va,\lambda,\om}[G_{\va,\lambda}^{\gamma}(\cdot,y),\phi(\cdot)]=\phi^{\gamma}(y),\label{defnG}
\end{align}
for any $ 1\leq \gamma\leq m $, $ \phi\in W_0^{1,p}(\om;\mathbb{C}^m) $ with $ p>d $. Particularly, if $ F\in L^q(\om;\mathbb{C}^m) $ with $ q>d/2 $,
\begin{align}
u_{\va,\lambda}(x)=\int_{\om}G_{\va,\lambda}(x,y)\overline{F(y)}dy,\label{repre}    
\end{align}
satisfies the Dirichlet problem $ (\mathcal{L}_{\va}-\lambda I)(u_{\va,\lambda})=F $ in $ \om $ and $ u_{\va,\lambda}=0 $ on $ \pa\om $. Moreover, let $ G_{\va,\overline{\lambda}}(x,y) $ be the Green function of the operator $ \mathcal{L}_{\va}-\overline{\lambda} I $, then 
\begin{align}
G_{\va,\lambda}(x,y)=[\overline{G_{\va,\overline{\lambda}}(y,x)}]^T,\label{duality}
\end{align}
which means that $ G_{\va,\lambda}^{\al\beta}(x,y)=\overline{G_{\va,\overline{\lambda}}^{\beta\al}(y,x)} $ for any $ 1\leq\al,\beta\leq m $,  $ x,y\in\om $ and $ x\neq y $. For any $ \sigma_1,\sigma_2\in(0,1) $ and $ k\in\mathbb{N}_+ $, the following estimates
\begin{align}
|G_{\va,\lambda}(x,y)|\leq\frac{C_{k,\theta_0}}{(1+|\lambda||x-y|^2)^{k}|x-y|^{d-2}}\min\left\{1,\frac{[\delta(x)]^{\sigma_1}}{|x-y|^{\sigma_1}},\frac{[\delta(y)]^{\sigma_2}}{|x-y|^{\sigma_2}}\right\}\label{Greene}
\end{align}
 hold for any $ x,y\in\om $ and $ x\neq y$, where $ \delta(x)=\dist(x,\pa\om) $ denotes the distance from $ x $ to the boundary of $ \om $ and $ C_{k,\theta_0} $ depends only on $ \mu,d,m,k,\theta_0,\omega(t),\sigma_1,\sigma_2 $ and $ \om $.
\end{thm}
\begin{proof}
For $ \lambda=0 $, the results follow from constructions of Green functions of standard elliptic operators $ \mathcal{L}_{\va} $. Then we only need to consider the case that $ \lambda\neq 0 $. First of all, let $ I(u)=\dashint_{\om(x,\rho)}u^{\gamma}(\cdot) $, then $ I\in H^{-1}(\om;\mathbb{C}^m) $. Precisely speaking, for any $ u\in H_0^1(\om;\mathbb{C}^m) $, we can easily get by using Sobolev embedding theorem $ L^{\frac{2d}{d-2}}(\om)\subset H_0^1(\om) $ that
\begin{align}
|I(u)|\leq C|\om(x,\rho)|^{-\frac{d-2}{2d}}\|u\|_{L^{\frac{2d}{d-2}}(\om)}\leq C|\om(x,\rho)|^{-\frac{d-2}{2d}}\|u\|_{H_0^1(\om)}\leq C\rho^{-\frac{d-2}{2}}\|u\|_{H_0^1(\om)}.\label{I u}
\end{align}
Consider the approximating Green function $ G_{\rho,\va,\lambda}(\cdot,y)=(G_{\rho,\va,\lambda}^{\al\beta}(\cdot,y))\in H_0^1(\om;\mathbb{C}^{m^2}) $ with $ \rho>0 $ and $ 1\leq\al,\beta\leq m $, such that
\begin{align}
B_{\va,\lambda,\om}[G_{\rho,\va,\lambda}^{\gamma}(\cdot,y),u(\cdot)]=\dashint_{\om(y,\rho)}u^{\gamma}(\cdot)\text{ for any }u\in H_0^1(\om;\mathbb{C}^m),\label{apf1}    
\end{align}
where $ 1\leq\gamma\leq m $. We see that the existence of $ G_{\rho,\va,\lambda}^{\gamma}(\cdot,y) $ is a direct consequence of Theorem \ref{existencethm}. Choose $ G_{\rho,\va,\lambda}^{\beta}(\cdot,y) $ $ (1\leq\beta\leq m) $ itself as the test function, we have
\begin{align}
B_{\va,\lambda,\om}[G_{\rho,\va,\lambda}^{\gamma}(\cdot,y),G_{\rho,\va,\lambda}^{\beta}(\cdot,y)]=I(G_{\rho,\va,\lambda}^{\gamma\beta}(\cdot,y)).\label{ww9}
\end{align}
In view of $ \eqref{ubuuboun} $, $ \eqref{laxmil} $ and $ \eqref{I u} $, it is not hard to see that
\begin{align}
\|G_{\rho,\va,\lambda}^{\gamma}(\cdot,y)\|_{L^2(\om)}^2&\leq C_{\theta_0}|\lambda|^{-1}\rho^{-\frac{d-2}{2}}\|\nabla_1 G_{\rho,\va,\lambda}^{\gamma}(\cdot,y)\|_{L^2(\om)},\nonumber\\
\|\nabla_1G_{\rho,\va,\lambda}^{\gamma}(\cdot,y)\|_{L^2(\om)}^2&\leq C_{\theta_0}\rho^{-\frac{d-2}{2}}\|\nabla_1 G_{\rho,\va,\lambda}^{\gamma}(\cdot,y)\|_{L^2(\om)},\nonumber
\end{align}
where $ \nabla_i $ ($ i=1,2 $) denote the derivatives for the first or second variable of Green functions and we will use these notations throughout this paper. Then
\begin{align}
\|\nabla_1 G_{\rho,\va,\lambda}^{\gamma}(\cdot,y)\|_{L^2(\om)}\leq C_{\theta_0}\rho^{-\frac{d-2}{2}}\text{ and } \|G_{\rho,\va,\lambda}^{\gamma}(\cdot,y)\|_{L^2(\om)}\leq C_{\theta_0}|\lambda|^{-\frac{1}{2}}\rho^{-\frac{d-2}{2}}.\label{ww11}
\end{align}
On the other hand, using Poincaré's inequality, we can deduce that
\begin{align}
\|G_{\rho,\va,\lambda}^{\gamma}(\cdot,y)\|_{L^2(\om)}\leq CR_0\|\nabla_1 G_{\rho,\va,\lambda}^{\gamma}(\cdot,y)\|_{L^2(\om)}\leq C_{\theta_0}R_0\rho^{-\frac{d-2}{2}}.\nonumber
\end{align}
This, combined with $ \eqref{ww11} $ and the case of $ \lambda=0 $, gives the following estimate
\begin{align}
\|G_{\rho,\va,\lambda}^{\gamma}(\cdot,y)\|_{L^2(\om)}\leq C_{\theta_0}(R_0^{-2}+|\lambda|)^{-\frac{1}{2}}\rho^{-\frac{d-2}{2}},\label{ww27}
\end{align}
for any $ \lambda\in\Sigma_{\theta_0}\cup\{0\} $ and $ 1\leq\gamma\leq m $. Let $ F\in C_0^{\infty}(\om;\mathbb{C}^m) $, consider the Dirichlet problem $ (\mathcal{L}_{\va}-\overline{\lambda}I)(u_{\va,\lambda})=F $ in $ \om $ and $ u_{\va,\lambda}=0 $ on $ \pa\om $. If we choose $ G_{\rho,\va,\lambda}^{\gamma}(\cdot,y) $ as the test function, it follows from simple calculations that
\begin{align}
&\int_{\om}\overline{G_{\rho,\va,\lambda}^{\gamma}(\cdot,y)}F(\cdot)=B_{\va,\overline{\lambda},\om}[u_{\va,\lambda}(\cdot),G_{\rho,\va,\lambda}^{\gamma}(\cdot,y)]=\overline{B_{\va,\lambda,\om}[G_{\rho,\va,\lambda}^{\gamma}(\cdot,y),u_{\va,\lambda}(\cdot)]}=\dashint_{\om(y,\rho)}\overline{u_{\va,\lambda}^{\gamma}(\cdot)}.\nonumber
\end{align}
Suppose that $ \operatorname{supp}F\subset B(y,R)\subset\om $, by using $ \eqref{Linfty} $, it can be got that 
\be
\begin{aligned}
&\left|\int_{\om}\overline{G_{\rho,\va,\lambda}^{\gamma}(\cdot,y)}F(\cdot)\right|\leq \|u_{\va,\lambda}\|_{L^{\infty}(\om(y,\rho))}\leq\|u_{\va,\lambda}\|_{L^{\infty}(B(y,\frac{R}{4}))}\\
&\quad\quad\quad\quad\leq C_{\theta_0}\left(\dashint_{B(y,\frac{R}{2})}|u_{\va,\lambda}|^2\right)^{\frac{1}{2}}+(1+|\lambda|R^2)^{n}R^{2}\left(\dashint_{B(y,\frac{R}{2})}|F|^q\right)^{\frac{1}{q}}
\end{aligned}\label{ww15}
\ee
for any $ \rho<\frac{R}{4} $ and $ q>\frac{d}{2} $. In view of $ \eqref{ww15} $, we need to estimate $ \left(\dashint_{B(y,\frac{R}{2})}|u_{\va,\lambda}|^2\right)^{\frac{1}{2}} $. Owing to $ \eqref{u2dd+2} $, we can obtain without difficulty that
\begin{align}
\|\nabla u_{\va,\lambda}\|_{L^2(\om)}\leq C_{\theta_0}\|F\|_{L^{\frac{2d}{d+2}}(\om)}\text{ and }\|u_{\va,\lambda}\|_{L^2(\om)}\leq C_{\theta_0}(R_0^{-2}+|\lambda|)^{-\frac{1}{2}}\|F\|_{L^{\frac{2d}{d+2}}(\om)}.\nonumber
\end{align}
These, together with Sobolev embedding theorem that $ H_0^1(\om)\subset L^{\frac{2d}{d-2}}(\om) $, imply that
\be
\begin{aligned}
\left(\dashint_{B(y,\frac{R}{2})}|u_{\va,\lambda}|^2\right)^{\frac{1}{2}}&\leq\left(\dashint_{B(y,\frac{R}{2})}|u_{\va,\lambda}|^{\frac{2d}{d-2}}\right)^{\frac{d-2}{2d}}\leq CR^{1-\frac{d}{2}}\left(\int_{\om}|u_{\va,\lambda}|^{\frac{2d}{d-2}}\right)^{\frac{d-2}{2d}}\\
&\leq CR^{1-\frac{d}{2}}\left(\int_{\om}|\nabla u_{\va,\lambda}|^{2}\right)^{\frac{1}{2}}\leq C_{\theta_0}R^{1-\frac{d}{2}}\left(\int_{\om}|F|^{\frac{2d}{d+2}}\right)^{\frac{d+2}{2d}}\\
&\leq C_{\theta_0}R^{1-\frac{d}{2}}\left(\int_{B(y,R)}|F|^{\frac{2d}{d+2}}\right)^{\frac{d+2}{2d}}=C_{\theta_0}R^2\left(\dashint_{B(y,R)}|F|^{\frac{2d}{d+2}}\right)^{\frac{d+2}{2d}}.
\end{aligned}\label{l2uvlab}
\ee
Since $ d\geq 3 $, we have $ q>\frac{d}{2}>\frac{2d}{d+2} $. Then by using Hölder's inequality and $ \eqref{l2uvlab} $, it is not hard to see that
\begin{align}
\left(\dashint_{B(y,\frac{R}{2})}|u_{\va,\lambda}|^2\right)^{\frac{1}{2}}&\leq C_{\theta_0}R^2\left(\dashint_{B(y,R)}|F|^{\frac{2d}{d+2}}\right)^{\frac{d+2}{2d}}\leq C_{\theta_0}R^2\left(\dashint_{B(y,R)}|F|^{q}\right)^{\frac{1}{q}}.\label{ww14}
\end{align}
This, together with $ \eqref{ww15} $, gives that
\begin{align}
\left|\int_{\om}\overline{G_{\rho,\va,\lambda}^{\gamma}(\cdot,y)}F(\cdot)\right|&\leq C_{\theta_0}(1+|\lambda|R^2)^{n}R^{2}\left(\dashint_{B(y,R)}|F|^{q}\right)^{\frac{1}{q}}.\label{ww16}
\end{align}
By using duality arguments, we can get that
\begin{align}
&\left(\dashint_{B(y,R)}|G_{\rho,\va,\lambda}^{\gamma}(\cdot,y)|^s\right)^{\frac{1}{s}}\leq C_{\theta_0}(1+|\lambda|R^2)^{n}R^{2-d},\label{ww17}
\end{align}
for any $ \rho<\frac{R}{4} $, $ R\leq\delta(x) $ and $ s\in (1,\frac{d}{d-2}) $. For $ x,y\in\om $ such that $ x\neq y $, set $ r=|x-y| $. If $ r\leq\frac{1}{2}\delta(y) $, choosing $ p\in(1,\frac{d}{d-2}) $, we have $ (\mathcal{L}_{\va}-\lambda I)(G_{\rho,\va,\lambda}^{\gamma}(\cdot,y))=0 $ in $ B(x,\frac{r}{2}) $ for any $ \rho<\frac{1}{4}r $. Then by using $ \eqref{Linfty} $ and $ \eqref{ww17} $, it is easy to show that for any $ s\in (1,\frac{d}{d-2}) $,
\be
\begin{aligned}
|G_{\rho,\va,\lambda}^{\gamma}(x,y)|&\leq \frac{C_{\theta_0}}{(1+|\lambda|r^2)^{n}}\left(\dashint_{B(x,\frac{r}{2})}|G_{\rho,\va,\lambda}^{\gamma}(\cdot,y)|^s\right)^{\frac{1}{s}}\\
&\leq\frac{C_{\theta_0}(1+|\lambda|r^2)^{n}r^{2-d}}{(1+|\lambda|r^2)^{n}}\leq \frac{C_{\theta_0}}{|x-y|^{d-2}},
\end{aligned}\label{ww18}
\ee
where we have chosen $ k=n\in\mathbb{N}_+ $. Assume that $ R<\frac{1}{4}\delta(y) $, then for any $ \rho<\frac{R}{4} $, we have $ (\mathcal{L}_{\va}-\lambda I)(G_{\rho,\va,\lambda}^{\gamma}(\cdot,y))=0 $ in $ \om\backslash B(x,R) $. Choose $ \varphi\in C_0^1(\om;\mathbb{R}) $, such that $ 0\leq\varphi\leq 1 $, $ \varphi\equiv 1 $ in $ \om\backslash B(y,2R) $, $ \varphi\equiv 0 $ in $ B(y,R) $ and $ |\nabla\varphi|\leq\frac{C}{R} $. Set $ u(z)=\varphi^2G_{\rho,\va,\lambda}^{\gamma}(z,y) $ as the test function, then it is not hard to obtain that
\be
\begin{aligned}
&\int_{\om}\varphi^2A_{\va}(\cdot)\nabla_1 G_{\rho,\va,\lambda}^{\gamma}(\cdot,y)\overline{\nabla_1 G_{\rho,\va,\lambda}^{\gamma}(\cdot,y)}-\lambda\int_{\om}\varphi^2|G_{\rho,\va,\lambda}^{\gamma}(\cdot,y)|^2\\
&\quad\quad\quad\quad=-\int_{\om}2\varphi A_{\va}(\cdot)\nabla_1 G_{\rho,\va,\lambda}^{\gamma}(\cdot,y)\nabla\varphi\overline{G_{\rho,\va,\lambda}^{\gamma}(\cdot,y)},
\end{aligned}\label{ww19}
\ee
where $ A_{\va}(x)=A(x/\va) $ if $ \va>0 $ and $ A_{0}(x)=\widehat{A} $. For $ \operatorname{Re}\lambda\geq 0 $, we can take the imaginary parts of both sides of $ \eqref{ww19} $ and get that
\be
\begin{aligned}
\|\varphi G_{\rho,\va,\lambda}^{\gamma}(\cdot,y)\|_{L^2(\om)}^2&\leq \frac{C_{\theta_0}}{|\lambda|}\int_{\om}|\nabla\varphi(\cdot)||G_{\rho,\va,\lambda}^{\gamma}(\cdot,y)||\varphi(\cdot)||\nabla_1 G_{\rho,\va,\lambda}^{\gamma}(\cdot,y)|\\
&\leq \frac{C_{\theta_0}}{|\lambda|}\left\{\|\nabla\varphi G_{\rho,\va,\lambda}^{\gamma}(\cdot,y)\|_{L^2(\om)}^2+\delta\|\varphi \nabla_1G_{\rho,\va,\lambda}^{\gamma}(\cdot,y)\|_{L^2(\om)}^2\right\},
\end{aligned}\label{ww20}
\ee
where $ \delta>0 $ is sufficiently small. For $ \operatorname{Re}\lambda<0 $, we can take the real parts of both sides of $ \eqref{ww19} $ and obtain that $ \eqref{ww20} $ is also true. Then $ \eqref{ww20} $ is true for any $ \lambda\in\Sigma_{\theta_0} $. Owing to $ \eqref{ww19} $, $ \eqref{ww20} $ and properties of $ \varphi $, it is obvious that 
\be
\begin{aligned}
\|\nabla_1 G_{\rho,\va,\lambda}^{\gamma}(\cdot,y)\|_{L^2(\om\backslash B(y,2R))}^2&\leq C_{\theta_0}\|\nabla\varphi G_{\rho,\va,\lambda}^{\gamma}(\cdot,y)\|_{L^2(\om)}^2\\
&\leq \frac{C_{\theta_0}}{R^2}\|G_{\rho,\va,\lambda}^{\gamma}(\cdot,y)\|_{L^2(B(y,2R)\backslash B(y,R))}^2\\
&\leq \frac{C_{\theta_0}}{R^2}\int_{B(y,2R)\backslash B(y,R)}\frac{1}{|z-y|^{2d-4}}dz\leq C_{\theta_0}R^{2-d}.
\end{aligned}\label{ww21}
\ee
On the other hand, if $ \rho>\frac{R}{4} $, according to $ \eqref{ww11} $, it can be obtained that
\begin{align}
\int_{\om\backslash B(y,2R)}|\nabla_1 G_{\rho,\va,\lambda}^{\gamma}(\cdot,y)|^2\leq \int_{\om}|\nabla_1 G_{\rho,\va,\lambda}^{\gamma}(\cdot,y)|^2\leq C_{\theta_0}R^{2-d}.\label{ww22}
\end{align}
$ \eqref{ww21} $, together with $ \eqref{ww22} $, yields that
\begin{align}
\int_{\om\backslash B(y,2R)}|\nabla_1 G_{\rho,\va,\lambda}^{\gamma}(\cdot,y)|^2\leq C_{\theta_0}R^{2-d},\label{ww23}
\end{align}
for any $ \rho>0 $ and $ R<\frac{1}{4}\delta(y) $. Let $ \psi\in C_0^1(\om;\R) $ such that $ 0\leq\psi\leq 1 $, $ \psi\equiv 0 $ in $ B(y,2R) $, $ \psi\equiv 1 $ in $ \om\backslash B(y,\frac{5}{2}R) $ and $ |\nabla\psi|\leq\frac{C}{R} $. With the help of the definition of $ \psi $ and Sobolev embedding theorem, we can deduce that for any $ \rho<\frac{R}{4} $ and $ R<\frac{1}{4}\delta(y) $,
\be
\begin{aligned}
\|\psi G_{\rho,\va,\lambda}^{\gamma}(\cdot,y)\|_{L^{\frac{2d}{d-2}}(\om)}^{\frac{2d}{d-2}}&\leq C\|\nabla(\psi G_{\rho,\va,\lambda}^{\gamma}(\cdot,y))\|_{L^{2}(\om)}^{\frac{2d}{d-2}}\\
&\leq C\|\nabla \psi G_{\rho,\va,\lambda}^{\gamma}(\cdot,y)\|_{L^2(\om)}^{\frac{2d}{d-2}}+C\|\psi\nabla_1 G_{\rho,\va,\lambda}^{\gamma}(\cdot,y)\|_{L^2(\om)}^{\frac{2d}{d-2}}\leq C_{\theta_0}R^{-d},
\end{aligned}\label{psiGvala}
\ee
where for the third inequality, we have used $ \eqref{ww18} $ and $ \eqref{ww23} $. Combining $ \eqref{psiGvala} $ and properties of $ \psi $, it can be easily inferred that
\begin{align}
\int_{\om\backslash B(y,\frac{5}{2}R)}| G_{\rho,\va,\lambda}^{\gamma}(\cdot,y)|^{\frac{2d}{d-2}}\leq C_{\theta_0}R^{-d}\text{ for any }\rho<\frac{R}{4}\text{ and } R<\frac{1}{4}\delta(y).\nonumber
\end{align} 
For the case that $ \rho\geq \frac{R}{4} $, due to $ \eqref{ww11} $, one can find that
\begin{align}
\int_{\om\backslash B(y,\frac{5}{2}R)}| G_{\rho,\va,\lambda}^{\gamma}(\cdot,y)|^{\frac{2d}{d-2}}\leq \left(\int_{\om}|\nabla_1  G_{\rho,\va,\lambda}^{\gamma}(\cdot,y)|^{2}\right)^{\frac{d}{d-2}}\leq C_{\theta_0}R^{-d}.\nonumber
\end{align}
We now address ourselves to the uniform estimates of $ G_{\rho,\va,\lambda}^{\gamma}(\cdot,y) $ and $ \nabla_1 G_{\rho,\va,\lambda}^{\gamma}(\cdot,y) $ with respect to parameter $ \rho>0 $. In the case of $ t>(\frac{\delta(y)}{4})^{1-d} $, we can choose $ R=t^{-\frac{1}{d-1}}<\frac{1}{4}\delta(y) $ and obtain that for any $ \rho>0 $,
\be
\begin{aligned}
&\left|\left\{x\in\om:|\nabla_1 G_{\rho,\va,\lambda}^{\gamma}(\cdot,y)|>t\right\}\right|\leq CR^d+t^{-2}\int_{\om\backslash B(y,2R)}|\nabla_1 G_{\rho,\va,\lambda}^{\gamma}(\cdot,y)|^2\\
&\quad\quad\quad\quad\quad\quad\leq CR^d+C_{\theta_0}t^{-2}R^{2-d}\leq Ct^{-\frac{d}{d-1}}+C_{\theta_0}t^{-2}t^{\frac{d-2}{d-1}}\leq C_{\theta_0}t^{-\frac{d}{d-1}}.
\end{aligned}\label{ww31}
\ee
When $ t>(\frac{\delta(y)}{4})^{2-d} $, similarly, we can choose $ R=t^{-\frac{1}{d-2}} $ and obtain that for any $ \rho>0 $,
\be
\begin{aligned}
&\left|\left\{x\in\om:|G_{\rho,\va,\lambda}^{\gamma}(\cdot,y)|>t\right\}\right|\leq CR^d+t^{-\frac{2d}{d-2}}\int_{\om\backslash B(y,R)}|G_{\rho,\va,\lambda}^{\gamma}(\cdot,y)|^{\frac{2d}{d-2}}\\
&\quad\quad\quad\quad\quad\quad\leq CR^d+C_{\theta_0}t^{-\frac{2d}{d-2}}R^{-d}\leq Ct^{-\frac{d}{d-2}}+C_{\theta_0}t^{-\frac{2d}{d-2}}t^{\frac{d}{d-2}}\leq C_{\theta_0}t^{-\frac{d}{d-2}}.
\end{aligned}\label{ww32}
\ee
Then in view of $ \eqref{ww31} $ and $ \eqref{ww32} $, it can be shown by simple calculations that
\begin{align}
\int_{\om}|G_{\rho,\va,\lambda}^{\gamma}(\cdot,y)|^s&\leq C[\delta(y)]^{s(2-d)}+C_{\theta_0}\int_{(\delta(y)/4)^{2-d}}t^{s-1}t^{-\frac{d}{d-2}}dt\nonumber\\
&\leq C_{\theta_0}\left\{[\delta(y)]^{s(2-d)}+[\delta(y)]^{s(2-d)+d}\right\},\nonumber
\end{align}
for any $ s\in [1,\frac{d}{d-2}) $ and
\begin{align}
\int_{\om}|\nabla_1 G_{\rho,\va,\lambda}^{\gamma}(\cdot,y)|^s\leq C_{\theta_0}\left\{[\delta(y)]^{s(1-d)}+[\delta(y)]^{s(1-d)+d}\right\},\nonumber
\end{align}
for any $ s\in [1,\frac{d}{d-1}) $. From the uniform estimates above, it follows that there exists a subsequence of $ \{G_{\rho_n,\va,\lambda}^{\gamma}(\cdot,y)\}_{n=1}^{\infty} $ and $ G_{\va,\lambda}^{\gamma}(\cdot,y) $ such that for any $ s\in (1,\frac{d}{d-1}) $ and $ 1\leq\gamma\leq m $,
\begin{align}
G_{\rho_n,\va,\lambda}^{\gamma}(\cdot,y)\rightharpoonup G_{\va,\lambda}^{\gamma}(\cdot,y)\text{ weakly in }W_0^{1,s}(\om;\mathbb{C}^m)\text{ as }n\to\infty.\label{ww33}
\end{align}
Hence, we have, for any $ \phi\in W_0^1(\om;\mathbb{C}^m) $ with $ p>d $
\begin{align}
B_{\va,\lambda,\om}[G_{\va,\lambda}^{\gamma}(\cdot,y),\phi(\cdot)]=\lim_{n\to\infty}B_{\va,\lambda,\om}[G_{\rho_n,\va,\lambda}^{\gamma}(\cdot,y),\phi(\cdot)]=\lim_{n\to\infty}\dashint_{\om(y,{\rho_n})}\phi^{\gamma}(\cdot)=\phi^{\gamma}(y),\nonumber
\end{align}
where we have used the definition of the approximating Green matrix. We note that there exists a weak solution $ u_{\va,\lambda}\in W_0^{1,p}(\om;\mathbb{C}^m) $ satisfying $ (\mathcal{L}_{\va}-\overline{\lambda} I)(u_{\va,\lambda})=F $ in $ \om $ and $ u_{\va,\lambda}=0 $ on $ \pa\om $ for any $ F\in L^{\frac{p}{2}}(\om;\mathbb{C}^m) $ with $ p>d $ ($ \frac{pd}{p+d}<\frac{p}{2} $). Thus we can obtain that
\begin{align}
u_{\va,\lambda}^{\gamma}(y)&=B_{\va,\lambda,\om}[G_{\va,\lambda}^{\gamma}(\cdot,y),u_{\va,\lambda}(\cdot)]=\overline{B_{\va,\overline{\lambda},\om}[u_{\va,\lambda}(\cdot),G_{\va,\lambda}^{\gamma}(\cdot,y)]}=\int_{\om}G_{\va,\lambda}^{\gamma}(\cdot,y)\overline{F(\cdot)}.\nonumber
\end{align} 
We now verify the uniqueness. If $ \widetilde{G}_{\va,\lambda}^{\gamma}(\cdot,y) $ is another Green matrix, we can also derive by representation formula that $ u_{\va,\lambda}^{\gamma}(y)=\int_{\om}\widetilde{G}_{\va,\lambda}^{\gamma}(\cdot,y)\overline{F(\cdot)} $. It follows from the uniqueness of the weak solution that
\begin{align}
\int_{\om}[\widetilde{G}_{\va,\lambda}^{\gamma}(\cdot,y)-G_{\va,\lambda}^{\gamma}(\cdot,y)]\overline{F(\cdot)}=0\text{ for any }F\in L^{\frac{p}{2}}(\om;\mathbb{C}^m).\nonumber
\end{align}
Then $ \widetilde{G}_{\va,\lambda}^{\gamma}(\cdot,y)=G_{\va,\lambda}^{\gamma}(\cdot,y) $ a.e. in $ \om $ due to the arbitrariness of $ F $. Next, let $ G_{\tau,\va,\overline{\lambda}}(\cdot,x) $ denote the approximating for Green functions of the operator $ \mathcal{L}_{\va}-\overline{\lambda}I $, which satisfy
\begin{align}
B_{\va,\overline{\lambda},\om}[G_{\tau,\va,\overline{\lambda}}^{\xi}(\cdot,x),u(\cdot)]=\dashint_{\om(x,\tau)}u^{\xi}(\cdot)\text{ for any }u\in H_0^1(\om;\mathbb{C}^m).\label{ww34}
\end{align}
By the same argument, we can derive the existence and uniqueness of $ G_{\va,\overline{\lambda}}^{\xi}(\cdot,x) $. Thus for any $ \tau,\rho>0 $ and $ 1\leq\gamma,\xi\leq m $, it is obvious to see that
\begin{align}
\dashint_{\om(y,\rho)}G_{\tau,\va,\overline{\lambda}}^{\gamma\xi}(\cdot,x)&=B_{\va,\lambda,\om}[G_{\rho,\va,\lambda}^{\gamma}(\cdot,y),G_{\va,\overline{\lambda}}^{\xi}(\cdot,x)]\nonumber\\
&=\overline{B_{\va,\overline{\lambda},\om}[G_{\va,\overline{\lambda}}^{\xi}(\cdot,x),G_{\rho,\va,\lambda}^{\gamma}(\cdot,y)]}=\dashint_{\om(x,\tau)}\overline{G_{\rho,\va,\lambda}^{\xi\gamma}(\cdot,y)}.\nonumber
\end{align}
Note that $ (\mathcal{L}_{\va}-\overline{\lambda} I)(G_{\tau,\va,\overline{\lambda}}^{\xi}(\cdot,x))=0 $ in $ \om\backslash B(x,\tau) $ and $ (\mathcal{L}_{\va}-\lambda I)(G_{\rho,\va,\lambda}^{\gamma}(\cdot,y))=0 $ in $ \om\backslash B(y,\rho) $. In view of the $ W^{1,p} $ estimates, $ G_{\tau,\va,\overline{\lambda}}^{\xi}(\cdot,x) $ and $ G_{\rho,\va,\lambda}^{\gamma}(\cdot,y) $ are locally Hölder continuous. Therefore, by taking $ \tau,\rho\to\infty $, we have $ \overline{G_{\va,\lambda}^{\xi\gamma}(x,y)}=G_{\va,\overline{\lambda}}^{\gamma\xi}(y,x) $, which implies that $ G_{\va,\lambda}(x,y)=[\overline{G_{\va,\overline{\lambda}}(y,x)}]^T $ for any $ x,y\in\om $ such that $ x\neq y $. This gives the proof of $ \eqref{duality} $.

Finally, we will prove $ \eqref{Greene} $. Let $ r=|x-y| $ and $ F\in C_0^{\infty}(\om(x,{\frac{r}{3}});\mathbb{C}^m) $. Assume that $ u_{\va,\lambda} $ is the solution of $ (\mathcal{L}_{\va}-\overline{\lambda}I)(u_{\va,\lambda})=F $ in $ \om $ and $ u_{\va,\lambda}=0 $ on $ \pa\om $. Then $ u_{\va,\lambda}(y)=\int_{\om}\overline{F(\cdot)}G_{\va,\lambda}(\cdot,y) $. Since $ (\mathcal{L}_{\va}-\overline{\lambda} I)(u_{\va,\lambda})=0 $ in $ \om\backslash \om(x,\frac{r}{3}) $ and $ \om(y,\frac{r}{3})\subset\om\backslash \om(x,\frac{r}{3}) $, it follows from $ \eqref{u2dd+2} $ and $ \eqref{Linfty} $ that for any $ k\in\mathbb{N}_+ $,
\begin{align}
|u_{\va,\lambda}(y)|&\leq \frac{C_{k,\theta_0}}{(1+|\lambda|r^2)^k}\left(\dashint_{\om(y,\frac{r}{3})}|u_{\va,\lambda}|^{2}\right)^{\frac{1}{2}}\leq \frac{C_{k,\theta_0}}{(1+|\lambda|r^2)^k}r^{1-\frac{d}{2}}\|u_{\va,\lambda}\|_{L^{\frac{2d}{d-2}}(\om(y,\frac{r}{3}))}\nonumber\\
&\leq \frac{C_{k,\theta_0}}{(1+|\lambda|r^2)^k}r^{1-\frac{d}{2}}\|u_{\va,\lambda}\|_{L^{\frac{2d}{d-2}}(\om)}\leq \frac{C_{k,\theta_0}}{(1+|\lambda|r^2)^k}r^{1-\frac{d}{2}}\|\nabla u_{\va,\lambda}\|_{L^{2}(\om)}\nonumber\\
&\leq\frac{C_{k,\theta_0}}{(1+|\lambda|r^2)^{k}}r^{1-\frac{d}{2}}\|F\|_{L^{\frac{2d}{d+2}}(\om)}=\frac{C_{k,\theta_0}}{(1+|\lambda|r^2)^{k}}r^{1-\frac{d}{2}} \|F\|_{L^{\frac{2d}{d+2}}(\om(x,\frac{r}{3}))}\nonumber\\
&\leq\frac{C_{k,\theta_0}}{(1+|\lambda|r^2)^{k}}r^{2-\frac{d}{2}} \|F\|_{L^{2}(\om(x,\frac{r}{3}))}.\nonumber
\end{align}
In view of duality arguments, we can infer that for any $ k\in\mathbb{N}_+ $,
\begin{align}
\left(\dashint_{\om(x,\frac{r}{3})}|G_{\va,\lambda}(\cdot,y)|^2\right)^{\frac{1}{2}}\leq\frac{C_{k,\theta_0}}{(1+|\lambda|r^2)^{k}}r^{2-d}.\label{ww56}
\end{align}
Note that $ (\mathcal{L}_{\va}-\lambda I)(G_{\va,\lambda}^{\gamma}(\cdot,y))=0 $ in $ \om\backslash B(y,r) $ for any $ r>0 $. So in the case of $ \frac{r}{6}\leq \delta(x) $, it follows from $ \eqref{Linfty} $ and $ \eqref{ww56} $ that for any $ k\in\mathbb{N}_+ $,
\begin{align}
|G_{\va,\lambda}(x,y)|\leq C_{\theta_0}\left(\dashint_{\om(x,\frac{r}{3})}|G_{\va,\lambda}(\cdot,y)|^{2}\right)^{\frac{1}{2}}\leq \frac{C_{k,\theta_0}}{(1+|\lambda||x-y|^2)^{k}}\frac{1}{|x-y|^{d-2}}.\nonumber   
\end{align}
When $ \frac{r}{6}>\delta(x) $, on the other hand, in view of $ \eqref{Holder} $ and $ \eqref{ww56} $, for any $ \sigma_1\in (0,1) $ and $ k\in\mathbb{N}_+ $, we have,
\begin{align}
|G_{\va,\lambda}(x,y)|&=|G_{\va,\lambda}(x,y)-G_{\va,\lambda}(\overline{x},y)|\leq [G_{\va,\lambda}(\cdot,y)]_{C^{0,\sigma_1}(\om(x,\frac{r}{6}))}|x-\overline{x}|^{\sigma_1}\nonumber\\
&\leq C_{\theta_0}\left(\frac{|x-\overline{x}|}{r}\right)^{\sigma_1}\left(\dashint_{\om(x,\frac{r}{3})}|G_{\va,\lambda}(\cdot,y)|^{2}\right)^{\frac{1}{2}}\leq \frac{C_{k,\theta_0}}{(1+|\lambda||x-y|^2)^{k}}\frac{[\delta(x)]^{\sigma_1}}{|x-y|^{d-2+\sigma_1}},\nonumber
\end{align}
where $ \overline{x}\in\pa\om $ is the point such that $ |x-\overline{x}|=\delta(x) $. It is easy to complete the proof of $ \eqref{Greene} $ by considering the same estimates for $ G_{\va,\overline{\lambda}}(x,y) $ and using property $ \eqref{duality} $.
\end{proof}

\subsection{Two dimensional Green functions}

\begin{thm}[Green functions of $ \mathcal{L}_{\va}-\lambda I $ with $ d=2 $]\label{Green2}
Suppose that $ A $ satisfies $ \eqref{sy} $, $ \eqref{el} $, $ \eqref{pe} $, $ \eqref{VMO} $, $ \va\geq 0 $ and $ \om $ is a bounded $ C^{1} $ domain in $ \mathbb{R}^2 $. If $ \lambda\in\Sigma_{\theta_0}\cup\{0\} $ with $ \theta_0\in(0,\frac{\pi}{2}) $, then there exists a unique Green function $ G_{\va,\lambda}=(G_{\va,\lambda}^{\al\beta}(\cdot,\cdot)):\om\times\om\to \mathbb{C}^{m^2}\cup\left\{\infty\right\} $ with $ 1\leq\al,\beta\leq m $, such that
\be
\begin{aligned}
G_{\va,\lambda}(\cdot,y)\in \operatorname{BMO}(\om;\mathbb{C}^{m^2}) \text{ i.e. } \left\|G_{\va,\lambda}(\cdot,y)\right\|_{\BMO(\om)}\leq C_{\theta_0}\text{ uniformly for }y\in\om.
\end{aligned}\label{BMO estimates}
\ee
Moreover, for all $ u_{\va,\lambda} $ being the weak solution for the Dirichlet problem $ (\mathcal{L}_{\va}-\lambda I)(u_{\va,\lambda})=F $ in $ \om $ and $ u_{\va,\lambda}=0 $ on $ \pa\om $, where $ F\in L^p(\om;\mathbb{C}^m) $, $ (p>1) $, we have
\begin{align}
u_{\va,\lambda}(x)=\int_{\om}G_{\va,\lambda}(x,y)\overline{F(y)}dy.\label{Representation formula}
\end{align}
Furthermore, for Green functions $ G_{\va,\overline{\lambda}}(x,y) $ corresponding to the operators of $ \mathcal{L}_{\va}-\overline{\lambda} I $, we have $ G_{\va,\lambda}(x,y)=\overline{[G_{\va,\overline{\lambda}}(y,x)]}^T $. For all $ \sigma_1,\sigma_2,\sigma\in (0,1) $ and $ k\in\mathbb{N}_+ $, Green functions satisfy pointwise estimates
\be
|G_{\va,\lambda}(x,y)|\leq\frac{C_{k,\theta_0}}{(1+|\lambda||x-y|^2)^k}\left(\frac{R_0}{|x-y|}\right)^{\sigma}\text{ for any }x,y\in\om,\label{preliminary}
\ee
\be
|G_{\va,\lambda}(x,y)|\leq C_{k,\theta_0}\left\{1+|\lambda|^{\frac{\sigma}{2}}R_0^{\sigma}+\ln\left(\frac{R_0}{|x-y|}\right)\right\}\text{ for any }x,y\in\om,\label{Pointwise estimates for Green functions 4}
\ee
\begin{align}
|G_{\va,\lambda}(x,y)|\leq\frac{C_{k,\theta_0}[\delta(x)]^{\sigma_1}}{(1+|\lambda||x-y|^2)^k|x-y|^{\sigma_1}}   &\text{ if }   \delta(x)<\frac{1}{4}|x-y|,\label{Pointwise estimates for Green functions 1}\\
|G_{\va,\lambda}(x,y)|\leq\frac{C_{k,\theta_0}[\delta(y)]^{\sigma_2}}{(1+|\lambda||x-y|^2)^k|x-y|^{\sigma_2}}  & \text{ if }   \delta(y)<\frac{1}{4}|x-y|,\label{Pointwise estimates for Green functions 2}\\
|G_{\va,\lambda}(x,y)|\leq\frac{C_{k,\theta_0}[\delta(x)]^{\sigma_1}[\delta(y)]^{\sigma_2}}{(1+|\lambda||x-y|^2)^k|x-y|^{\sigma_1+\sigma_2}}&\text{ if } \min\left\{\delta(x),\delta(y)\right\}<\frac{1}{4}|x-y|,\label{Pointwise estimates for Green functions 3}
\end{align}
where $ \delta(x)=\operatorname{dist}(x,\pa\om) $ denotes the distance from $ x $ to the boundary of $ \om $, $ x\neq y $ and $ C_{k,\theta_0} $ are constants depending only on $ \sigma_1,\sigma_2,\sigma,\mu,\omega(t),k,\theta_0,m $ and $ \om $.
\end{thm}

\begin{proof}
For any $ y\in\om $ and $ \rho>0 $, in view of Theorem \ref{existencethm}, there exists a matrix-valued function $ G_{\rho,\va,\lambda}(\cdot,y)=(G_{\rho,\va,\lambda}^{\alpha\beta}(\cdot,y)):\om\to\mathbb{C}^{m^2} $, such that for any $ 1\leq\gamma\leq m $, $ G_{\rho,\va,\lambda}^{\gamma}(\cdot,y)\in H_0^1(\om;\mathbb{C}^m) $ satisfies
\begin{align}
B_{\va,\lambda,\om}[G_{\rho,\va,\lambda}^{\gamma}(\cdot,y),u(\cdot)]=\dashint_{\om(y,\rho)}u^{\gamma}(\cdot)\text{ for any }u\in H_0^1(\om;\mathbb{C}^m).\nonumber
\end{align}
For any atom function in $ \om $ denoted by $ a=a(\cdot) $, s.t.
\begin{align}
\operatorname{supp}(a)\subset\om(y,\rho)\text{ and }\left\|a\right\|_{L^{\infty}(\om)}\leq\frac{1}{|\om(y,\rho)|},\nonumber
\end{align}
we can get $ v_{\va,\lambda}\in H_0^1(\om;\mathbb{R}^m) $ such that $ (\mathcal{L}_{\va}-\overline{\lambda} I)(v_{\va,\lambda})=a $ in $ \om $ and $ v_{\va,\lambda}=0 $ on $ \pa\om $. Then
\begin{align}
\dashint_{\om(y,\rho)}v_{\va,\lambda}^{\gamma}(\cdot)&=B_{\va,\lambda,\om}[G_{\rho,\va,\lambda}^{\gamma}(\cdot,y),v_{\va,\lambda}(\cdot)]=\overline{B_{\va,\overline{\lambda},\om}[v_{\va,\lambda}(\cdot),G_{\rho,\va,\lambda}^{\gamma}(\cdot,y)]}=\int_{\om}G_{\rho,\va,\lambda}^{\alpha\gamma}(\cdot,y)\overline{a^{\alpha}(\cdot)}.\nonumber
\end{align}
This, together with $ \eqref{Linftyatom} $, implies the following estimate
\begin{align}
\left|\int_{\om}G_{\rho,\va,\lambda}(\cdot,y)\overline{a(\cdot)}\right|\leq \left|\dashint_{\om(y,\rho)}v_{\va,\lambda}(\cdot)\right|\leq\|v_{\va,\lambda}\|_{L^{\infty}(\om(y,\rho))} \leq\|v_{\va,\lambda}\|_{L^{\infty}(\om)}\leq C_{\theta_0},\nonumber
\end{align}
where $ C_{\theta_0} $ depends only on $ \mu,\omega(t),\theta_0,m $ and $ \om $. According to the fact that $ \mathcal{H}^1 $ is the dual space of $ \operatorname{BMO} $ space, we can derive that $ G_{\rho,\va,\lambda}(\cdot,y) $ has a uniform boundedness $ C_{\theta_0} $ in $ \operatorname{BMO} $ space, where $ C_{\theta_0} $ depends only on $ \mu,\omega(t),m,\theta_0 $ and $ \om $. In view of Banach-Alaoglu theorem, we have, for all $ y\in\om $, there exists a sequence $ \{\rho_j\} $ such that $ \rho_j\to 0 $ when $ j\to\infty $ and functions $ G_{\va,\lambda}^{\alpha\beta}(\cdot,y)\in \operatorname{BMO}(\om) $ such that $ G_{\rho_j,\va,\lambda}^{\alpha\beta}(\cdot,y) $ converge to $ G_{\va,\lambda}^{\alpha\beta}(\cdot,y) $ in the space $ \operatorname{BMO}(\om) $ with the sense of the weak*-topology. For $ F\in L^q(\om;\mathbb{R}^m) $ where $ q>1 $, we can choose $ 1<q_1<q $, $ p=\frac{2q_1}{2-q_1}>2 $ and $ u_{\va,\lambda}\in W_0^{1,p}(\om;\mathbb{R}^m) $, such that $ (\mathcal{L}_{\va}-\overline{\lambda}I)(u_{\va,\lambda})=F $ and $ u_{\va,\lambda}=0 $ on $ \pa\om $. Then it can be obtained that
\be
\begin{aligned}
\dashint_{\om(y,\rho)}u_{\va,\lambda}^{\gamma}(\cdot)&=B_{\va,\lambda,\om}[G_{\rho,\va,\lambda}^{\gamma}(\cdot,y),u_{\va,\lambda}(\cdot)]\\
&=\overline{B_{\va,\overline{\lambda},\om}[u_{\va,\lambda}(\cdot),G_{\rho,\va,\lambda}^{\gamma}(\cdot,y)]}=\int_{\om}G_{\rho,\va,\lambda}^{\al\gamma}(\cdot,y)\overline{F^{\al}(\cdot)}.
\end{aligned}\label{ww181}
\ee
In view of Remark \ref{remimp}, Poincaré's inequality and Hölder's inequality, we have
\begin{align}
\left\|u_{\va,\lambda}\right\|_{W^{1,p}(\om)}\leq C\|\nabla u_{\va,\lambda}\|_{L^p(\om)}\leq C_{\theta_0}\left\|F\right\|_{L^{q_1}(\om)}\leq C_{\theta_0}\left\|F\right\|_{L^q(\om)},\nonumber
\end{align}
where $ C_{\theta_0} $ is a constant depending only on $ \mu,\omega(t),\theta_0,m,p,q $ and $ \om $. Using Sobolev embedding theorem, we see that $ u_{\va,\lambda} $ is continuous. Letting $ \rho_j\to 0 $, the left hand side of $ \eqref{ww181} $ converges to $ u_{\va,\lambda}(y) $. On the other hand, according to the embedding theorem $ L^p\subset \mathcal{H}^1 $, one can obtain that
\begin{align}
u_{\va,\lambda}^{\gamma}(y)=\lim_{j\to\infty}\dashint_{\om(y,\rho_j)}u_{\va,\lambda}^{\gamma}(\cdot)=\lim_{j\to\infty}\int_{\om}G_{\rho_j,\va,\lambda}^{\al\gamma}(\cdot,y)\overline{F^{\al}(\cdot)}=\int_{\om}G_{\va,\lambda}^{\al\gamma}(\cdot,y)\overline{F^{\al}(\cdot)}.\label{ww191}
\end{align}
Since the representation formula, uniqueness of Green functions and duality property follow from almost the same arguments for the case $ d\geq 3 $, here we will not repeat the proofs of them.

Finally, we will prove the pointwise estimates of Green functions. Namely, we will prove $ \eqref{preliminary} $-$ \eqref{Pointwise estimates for Green functions 3} $. Letting $ x_0,y_0\in\om $ with $ x_0\neq y_0 $ and assuming that $ \delta(x_0)<\frac{1}{2}|x_0-y_0|=\frac{1}{2}r $, we have $ \om(x_0,\frac{1}{2}r)\subset\om\backslash\left\{y_0\right\} $. According to the definition of $ G_{\va,\lambda}^{\gamma}(\cdot,y_0) $, we have $ (\mathcal{L}_{\va}-\lambda I)(G_{\va,\lambda}^{\gamma}(\cdot,y_0))=0 $ in $ \om(x_0,\frac{1}{2}r) $ and $ G_{\va,\lambda}^{\gamma}(\cdot,y_0)=0 $ on $ \pa\om\cap B(x_0,\frac{1}{2}r) $. Then by $\eqref{Linfty} $, we can obtain that for any $ k\in\mathbb{N}_+ $,
\begin{align}
|G_{\va,\lambda}(x_0,y_0)|\leq C\left\|G_{\va,\lambda}(\cdot,y_0)\right\|_{L^{\infty}(\om(x_0,\frac{r}{4}))}\leq \frac{C_{k,\theta_0}}{(1+|\lambda|r^2)^k}\dashint_{\om(x_0,\frac{1}{2}r)}|G_{\va,\lambda}(\cdot,y_0)|.\nonumber
\end{align}
According to the definition of $ \BMO(\om) $ and the observation that $ G_{\va,\lambda}(\cdot,y_0)_{x_0,\frac{r}{2}}=0 $ by $ \eqref{ww35} $, it is easy to see that for any $ k\in\mathbb{N}_+ $,
\begin{align}
|G_{\va,\lambda}(x_0,y_0)|&\leq \frac{C_{k,\theta_0}}{(1+|\lambda|r^2)^k}\dashint_{\om(x_0,\frac{r}{2})}|G_{\va,\lambda}(z,y_0)-G_{\va,\lambda}(z,y_0)_{x_0,\frac{r}{2}}|dz\nonumber\\
&\leq \frac{C_{k,\theta_0}}{(1+|\lambda|r^2)^k}\left\|G_{\va,\lambda}(\cdot,y_0)\right\|_{\BMO(\om)}\leq \frac{C_{k,\theta_0}}{(1+|\lambda|r^2)^k},\nonumber
\end{align}
where $ C_{k,\theta_0} $ depends only on $ \mu,\omega(t),k,\theta_0,m $ and $ \om $. With the help of estimates above, we can find that if $ \delta(x)<\frac{1}{2}|x-y| $, then
\begin{align}
|G_{\va,\lambda}(x,y)|\leq \frac{C_{k,\theta_0}}{(1+|\lambda|r^2)^k}\leq \frac{C_{k,\theta_0}}{(1+|\lambda||x-y|^2)^k}.\label{Boundedness estimates}
\end{align}
Assume that $ \delta(x_0)<\frac{1}{4}|x_0-y_0|=\frac{1}{4}r $ and $ z_0\in\pa\om $ is chosen such that $ |x_0-z_0|=\delta(x_0) $, we note that $ (\mathcal{L}_{\va}-\lambda I)(G_{\va,\lambda}^{\gamma}(\cdot,y_0))=0 $ in $ \om(x_0,\frac{r}{2}) $ and $ G_{\va,\lambda}^{\gamma}(\cdot,y_0)=0 $ on $ \pa\om\cap B(z_0,\frac{r}{2}) $. In view of the localized boundary H\"{o}lder estimates $ \eqref{Holder} $, then for all $ \sigma_1\in (0,1) $, there is
\be
\begin{aligned}
|G_{\va,\lambda}^{\gamma}(x_0,y_0)|&= |G_{\va,\lambda}^{\gamma}(x_0,y_0)-G_{\va,\lambda}^{\gamma}(z_0,y_0)|\\
&\leq |x_0-z_0|^{\sigma_1}[G_{\va,\lambda}^{\gamma}(\cdot,y_0)]_{C^{0,\sigma_1}(\om(z_0,\frac{3r}{8}))}\\
&\leq C_{\theta_0}\left(\frac{[\delta(x_0)]}{r}\right)^{\sigma_1}\left(\dashint_{\om(z_0,\frac{r}{2})}|G_{\va,\lambda}^{\gamma}(\cdot,y_0)|^2\right)^{\frac{1}{2}}.
\end{aligned}\label{Gvalamcthe}
\ee
Simple obervations give that for all $ x\in \om(z_0,\frac{r}{2}) $, we have $ \delta(x)<\frac{1}{2}r $. Then due to $ \eqref{Boundedness estimates} $ and $ \eqref{Gvalamcthe} $
\begin{align}
|G_{\va,\lambda}(x_0,y_0)|\leq\frac{C_{k,\theta_0}[\delta(x_0)]^{\sigma_1}}{(1+|\lambda||x_0-y_0|^2)^k|x_0-y_0|^{\sigma_1}}.\nonumber
\end{align}
This proves $ \eqref{Pointwise estimates for Green functions 1} $. On the other hand, $ \eqref{Pointwise estimates for Green functions 2} $ can be obtained naturally by considering the same estimates for Green functions of $ \mathcal{L}_{\va}-\overline{\lambda}I $. In addition, we can see from the above proof that $ \frac{1}{4} $ is not an essential constant in $ \eqref{Pointwise estimates for Green functions 1} $-$ \eqref{Pointwise estimates for Green functions 2} $. That is, $ \eqref{Pointwise estimates for Green functions 1} $-$ \eqref{Pointwise estimates for Green functions 2} $ are also true if we change $ \frac{1}{4} $ to any constants $ C_0 $ such that $ 0<C_0<\frac{1}{2} $ by using almost the same arguments. Next, let us prove $ \eqref{Pointwise estimates for Green functions 3} $. We can assume that $ \delta(x_0)<\frac{1}{4}|x_0-y_0| $ and $ \delta(y_0)<\frac{1}{4}|x_0-y_0| $. Otherwise, $ \eqref{Pointwise estimates for Green functions 3} $ follows from $ \eqref{Pointwise estimates for Green functions 1} $-$ \eqref{Pointwise estimates for Green functions 2} $ directly. By using almost the same methods, we have, for any $ 1\leq\gamma\leq m $,
\begin{align}
|G_{\va,\lambda}^{\gamma}(x_0,y_0)|&\leq C_{\theta_0}\left(\frac{[\delta(x_0)]}{r}\right)^{\sigma_1}\left(\dashint_{\om(z_0,\frac{7r}{16})}|G_{\va,\lambda}^{\gamma}(\cdot,y_0)|^2\right)^{\frac{1}{2}}.\label{ww10000}
\end{align}
For all $ y\in B(z_0,\frac{7r}{16})\cap\om $, it is esay to find by triangular inequality that $ |y-y_0|\geq |x_0-y_0|-|x_0-y|\geq\frac{9}{16}r $. Meanwhile, it is not hard to obtain that 
\begin{align}
|G_{\va,\lambda}^{\gamma}(z,y_0)|\leq\frac{C_{k,\theta_0}[\delta(y_0)]^{\sigma_2}}{(1+|\lambda||z-y_0|^2)^k|z-y_0|^{\sigma_2}}\text{ for any }k\in\mathbb{N}_+\text{ and }z\in \om\left(z_0,\frac{7r}{16}\right).\nonumber
\end{align}
This, together with $ \eqref{ww10000} $, gives the proof of $ \eqref{Pointwise estimates for Green functions 3} $. Indeed, for any $ k\in\mathbb{N}_+ $ and $ \sigma_1,\sigma_2\in(0,1) $,
\begin{align}
|G_{\va,\lambda}^{\gamma}(x_0,y_0)|&\leq C_{k,\theta_0}\left(\frac{[\delta(x_0)]}{r}\right)^{\sigma_1}\left(\dashint_{B(z_0,\frac{7r}{16})\cap\om}\frac{[\delta(y_0)]^{2\sigma_2}}{(1+|\lambda||y-y_0|^2)^{2k}|y-y_0|^{2\sigma_2}}dy\right)^{\frac{1}{2}}\nonumber\\
&\leq\frac{C_{k,\theta_0}[\delta(x_0)]^{\sigma_1}[\delta(y_0)]^{\sigma_2}}{(1+|\lambda||x_0-y_0|^2)^{k}|x_0-y_0|^{\sigma_1+\sigma_2}}.\nonumber
\end{align}
At last, if $ \delta(x_0)\geq\frac{1}{4}|x_0-y_0| $ and $ \delta(y_0)\geq\frac{1}{4}|x_0-y_0| $, we choose $ F\in C_0^{\infty}(\om(x_0,\frac{r}{4});\mathbb{C}^m) $. Obviously, here, $ \om(x_0,\frac{r}{4})=B(x_0,\frac{r}{4}) $. Let $ w_{\va,\lambda}\in H_0^1(\om;\mathbb{C}^m) $ satisfy $ (\mathcal{L}_{\va}-\overline{\lambda} I)(w_{\va,\lambda})=F $ in $ \om $ and $ w_{\va,\lambda}=0 $ on $ \pa\om $. In view of the representation theorem, we have that $ w_{\va,\lambda}(y)=\int_{\om}\overline{G_{\va,\lambda}(\cdot,y)}F(\cdot) $. Since $ F\equiv 0 $ in $ \om\backslash\om(x_0,\frac{r}{4}) $, then $ (\mathcal{L}_{\va}-\overline{\lambda} I)(w_{\va,\lambda})=0 $ in $ \om\backslash\om(x_0,\frac{r}{4}) $. For all $ p>2 $ and $ 1<q<2 $, by using $ \eqref{Linfty} $ and Sobolev embedding theorem $ H_0^1(\om)\subset L^p(\om) $ with $ p>1 $, it is not hard to get that
\begin{align}
|w_{\va,\lambda}(y_0)|&\leq  C_{\theta_0}\left(\dashint_{\om(y_0,\frac{r}{4})}|w_{\va,\lambda}|^2\right)^{\frac{1}{2}}\leq C_{\theta_0}\left(\dashint_{\om(y_0,\frac{r}{4})}|w_{\va,\lambda}|^p\right)^{\frac{1}{p}}=C_{\theta_0}R_0^{\frac{2}{p}}r^{-\frac{2}{p}}\left(\int_{\om}|\nabla w_{\va,\lambda}|^2\right)^{\frac{1}{2}}\nonumber\\
&\leq C_{\theta_0}R_0^{\frac{2}{p}+2-\frac{2}{q}}r^{-\frac{2}{p}}\left(\int_{\om}|F|^q\right)^{\frac{1}{q}}\leq C_{\theta_0}R_0^{\frac{2}{p}+2-\frac{2}{q}}r^{-\frac{2}{p}+\frac{2}{q}-1}\left(\int_{\om(x_0,\frac{r}{4})}|F|^2\right)^{\frac{1}{2}},\nonumber
\end{align}
where for the third inequality, we have used $ \eqref{uq} $. This implies that
\begin{align}
\left|\int_{\om(x_0,\frac{r}{4})}G_{\va,\lambda}(\cdot,y_0)\overline{F(\cdot)}\right|\leq C_{\theta_0}r\left(\frac{R_0}{r}\right)^{\frac{2}{p}-\frac{2}{q}+2}\left(\int_{\om(x_0,\frac{r}{4})}|F|^2\right)^{\frac{1}{2}}.\label{ww211}
\end{align}
Owing to duality arguments, $ \eqref{ww211} $ implies that for all $ p>2 $ and $ 1<q<2 $
\begin{align}
\left(\dashint_{\om(x_0,\frac{r}{4})}|G_{\va,\lambda}(\cdot,y_0)|^2\right)^{\frac{1}{2}}\leq C_{\theta_0}\left(\frac{R_0}{r}\right)^{\frac{2}{p}-\frac{2}{q}+2}.\label{ww2222}
\end{align}
For any $ \sigma\in(0,1) $, we can choose special $ p,q $ such that $ -\frac{2}{p}+\frac{2}{q}-2=-\sigma $. Then $ \eqref{ww2222} $, together with $ \eqref{Linfty} $, gives the proof of $ \eqref{preliminary} $. Of course, there is still a certain distance between this and $ \eqref{Pointwise estimates for Green functions 4} $, so we need to make more precise estimates.

For $ x_0,y_0\in\om $, letting $ r_1=\frac{1}{2}|x_0-y_0| $, it can be seen that if $ \delta(x_0)<\frac{1}{2}|x_0-y_0| $, then by $ \eqref{Boundedness estimates} $, we have $ |G_{\va,\lambda}(x_0,y_0)|\leq C_{\theta_0} $. If $ \delta(x_0)\geq\frac{1}{2}|x_0-y_0| $, we need to consider the sequence of subsets of $ \om $ denoted as $ \om_j=\om(x_0,2^jr_1) $ with $ j=0,1,...,N $ such that $ 2^Nr_1\geq \delta(x_0) $ and $ 2^{N-1}r_1<\delta(x_0) $. Obviously, we can bound $ N $ by the inequality
\begin{align}
N\leq C\left\{1+\ln\left(\frac{ R_0}{|x_0-y_0|}\right)\right\}.\label{lnN}
\end{align}
According to the fact that $ G_{\va,\lambda}(\cdot,y_0)\in \operatorname{BMO}(\om) $, we have, if $ 1\leq j\leq N-2 $, then
\be
\begin{aligned}
&|G_{\va,\lambda}(\cdot,y_0)_{x_0,2^{j}r_1}-G_{\va,\lambda}(\cdot,y_0)_{x_0,2^{j+1}r_1}|\\
&\quad\quad=\left|\dashint_{\om(x_0,2^jr_1)}G_{\va,\lambda}(\cdot,y_0)-\dashint_{\om(x_0,2^{j+1}r_1)}G_{\va,\lambda}(\cdot,y_0)\right|\\
&\quad\quad\leq C\dashint_{\om(x_0,2^{j+1}r_1)}|G_{\va,\lambda}(x,y_0)-G_{\va,\lambda}(\cdot,y_0)_{x_0,2^{j+1}r_1}|dx\\
&\quad\quad\leq C\left\|G_{\va,\lambda}(\cdot,y_0)\right\|_{\BMO(\om)}\leq C_{\theta_0}.
\end{aligned}\label{Average difference estimates}
\ee
With the choice of $ N $, we can get that
\be
\begin{aligned}
\dashint_{\om_{N-1}}|G_{\va,\lambda}(\cdot,y_0)|&\leq C\dashint_{\om_{N}}|G_{\va,\lambda}(\cdot,y_0)|\leq C\left\|G_{\va,\lambda}(\cdot,y_0)\right\|_{\BMO(\om)}\leq C_{\theta_0}. 
\end{aligned}\label{Large scale integral}
\ee
In the view of $ \eqref{Linfty} $, if $ p>2 $ and $ q=\frac{2p}{p+2} $, then for any $ k\in\mathbb{N}_+ $,
\begin{align}
|G_{\va,\lambda}(x_0,y_0)-G_{\va,\lambda}(\cdot,y_0)_{x_0,r_1}|&\leq\frac{C_{k,\theta_0}}{(1+|\lambda|r_1^2)^k}\dashint_{B(x_0,r_1)}|G_{\va,\lambda}(x,y_0)-G_{\va,\lambda}(\cdot,y_0)_{x_0,r_1}|dx
\nonumber\\
&\quad+C_{k,\theta_0}(1+|\lambda|r_1^2)^{n}|\lambda|r_1^2|G_{\va,\lambda}(\cdot,y_0)_{x_0,r_1}|\label{The second important estimates}\\
&\leq \frac{C_{k,\theta_0}}{(1+|\lambda|r_1^2)^k}+C_{k,\theta_0}(1+|\lambda|r_1^2)^{n}|\lambda|r_1^2|G_{\va,\lambda}(\cdot,y_0)_{x_0,r_1}|,\nonumber
\end{align}
where the second inequality in $ \eqref{The second important estimates} $ is derived from $ \eqref{Average difference estimates} $ and the fact that
\begin{align}
\dashint_{B(x_0,r_1)}|G_{\va,\lambda}(x,y_0)-G_{\va,\lambda}(\cdot,y_0)_{x_0,r_1}|dx\leq \left\|G_{\va,\lambda}(\cdot,y_0)\right\|_{\BMO(\om)}\leq C_{\theta_0}.\nonumber
\end{align}
In view of $ \eqref{Average difference estimates} $, $\eqref{Large scale integral} $ and $ \eqref{The second important estimates} $, it is not hard to see that
\be
\begin{aligned}
|G_{\va,\lambda}(x_0,y_0)|&\leq  C\sum_{j=-1}^{N-2}\left|\dashint_{\om_j}G_{\va,\lambda}(\cdot,y_0)-\dashint_{\om_{j+1}}G_{\va,\lambda}(\cdot,y_0)\right|+\dashint_{\om_{N-1}}|G_{\va,\lambda}(\cdot,y_0)|\\
&\leq C_{\theta_0}\left(1+\ln\left(\frac{ R_0}{|x_0-y_0|}\right)\right)+C|G_{\va,\lambda}(x_0,y_0)-G_{\va,\lambda}(\cdot,y_0)_{x_0,r_1}|\\
&\leq C_{\theta_0}\left(1+\ln\left(\frac{ R_0}{|x_0-y_0|}\right)\right)+C_{\theta_0}(1+|\lambda|r_1^2)^{n}|\lambda|r_1^2|G_{\va,\lambda}(\cdot,y_0)_{x_0,r_1}|,
\end{aligned}\label{lnzuiy}
\ee
where we denote $ \dashint_{\om_{-1}}G_{\va,\lambda}(\cdot,y_0)=G_{\va,\lambda}(x_0,y_0) $ and $ C_{\theta_0} $ depends only on $ \mu,\omega(t),\theta_0,m $ and $ \om $. Setting $ \sigma\in (0,1) $ and using $ \eqref{preliminary} $, then for any $ k\in\mathbb{N}_+ $,
\begin{align}
|G_{\va,\lambda}(\cdot,y_0)_{x_0,r_1}|&= \left|\dashint_{\om_0}G_{\va,\lambda}(\cdot,y_0)\right|\leq\dashint_{\om_0}|G_{\va,\lambda}(\cdot,y_0)|\nonumber\\
&\leq \dashint_{B(x_0,r_1)}\frac{C_{k,\theta_0}R_0^{\sigma}}{(1+|\lambda||x-y_0|^2)^k|x-y_0|^{\sigma}}dx\leq\frac{C_{k,\theta_0}R_0^{\sigma}}{(1+|\lambda||x_0-y_0|^2)^k|x_0-y_0|^{\sigma}}.\nonumber
\end{align}
Choosing $ k\in\mathbb{N}_+ $ such that $ k>n $, we have
\begin{align}
C_{\theta_0}(1+|\lambda|r_1^2)^{n}|\lambda|r_1^2|G_{\va,\lambda}(\cdot,y_0)_{x_0,r_1}|\leq C_{\theta_0} |\lambda|^{\frac{\sigma}{2}}R_0^{\sigma}.\nonumber
\end{align}
This, together with $ \eqref{lnzuiy} $, completes the proof.
\end{proof}

\subsection{More estimates of Green functions}

\begin{thm}\label{lipgg}
For $ \va\geq 0 $ and $ d\geq 2 $, $ \lambda\in\Sigma_{\theta_0}\cup\{0\} $ with $ \theta_0\in(0,\frac{\pi}{2}) $, let $ \om $ be a bounded $ C^{1,\eta} $ $ (0<\eta<1) $ domain in $ \mathbb{R}^d $. Suppose that $ A $ satisfies $ \eqref{sy} $, $ \eqref{el} $, $ \eqref{pe} $ and $ \eqref{Hol} $. Then for any $ k\in\mathbb{N}_+ $, the Green functions of $ (\mathcal{L}_{\va}-\lambda I) $ satisfy the uniform pointwise estimates
\begin{align}
|\nabla_1 G_{\va,\lambda}(x,y)|+|\nabla_2 G_{\va,\lambda}(x,y)|&\leq \frac{C_{k,\theta_0} }{(1+|\lambda||x-y|^2)^{k}|x-y|^{d-1}},\label{Green Lipschitz 11}\\
|\nabla_1G_{\va,\lambda}(x,y)|&\leq \frac{C_{k,\theta_0} \delta(y)}{(1+|\lambda||x-y|^2)^{k}|x-y|^{d}},\label{Green Lipschitz 12}\\
|\nabla_2G_{\va,\lambda}(x,y)|&\leq \frac{C_{k,\theta_0}\delta(x)}{(1+|\lambda||x-y|^2)^{k}|x-y|^{d}},\label{Green Lipschitz 13}\\
|\nabla_1\nabla_2G_{\va,\lambda}(x,y)|&\leq\frac{C_{k,\theta_0}}{(1+|\lambda||x-y|^2)^{k}|x-y|^{d}},\label{Green Lipschitz 14}
\end{align}
for $ x,y\in\om $ and $ x\neq y $, where $ C_{k,\theta_0} $ depends only on $ \mu,d,m,\tau,\nu,k,\theta_0,\eta $ and $ \om $.
\end{thm}
\begin{rem}
The Lipschitz estimates of Green functions for operators $ \mathcal{L}_{\va}-\lambda I $ given above are sharp. These are of great significance in this paper. Further proofs rely heavily on these estimates. In fact, we point out that after proving these estimates, the proofs of theorems on convergence rates for Green functions are standard and the reason why the standard proofs can be carried out later is that we calculated the impact of $ \lambda $ in Lipschitz estimates to the best. 
\end{rem}
\begin{rem}
The proof of Theorem \ref{lipgg} is quite different from which of Lipschitz estimates of Green functions for the operator $ \mathcal{L}_{\va} $, especially when $ d=2 $. Recall that in \cite{Shen2}, the case $ d=2 $ was discussed by using the fact that $ \mathcal{L}_{\va}(c)=0 $ for any constant vector $ c\in\mathbb{C}^m $. However operator $ \mathcal{L}_{\va}-\lambda I $ do not have such property. In this point of view, we need to modify the proofs in \cite{Shen2} and use some technical calculations here.
\end{rem}
\begin{proof}[Proof of Theorem \ref{lipgg}]
Firstly, we consider the case $ d\geq 3 $. Under the condition that $ A $ is Hölder continuous, we can improve the estimates $ \eqref{Greene} $ to the case that $ \sigma_1=\sigma_2=1 $. Setting $ R=|x-y| $ and applying $ \eqref{Greene} $, it is apparent that if $ \frac{R}{6}\leq \delta(x) $, then
\begin{align}
|G_{\va,\lambda}(x,y)|\leq \frac{C_{k,\theta_0}}{(1+|\lambda||x-y|^2)^{k}}\frac{1}{|x-y|^{d-2}},\nonumber
\end{align}
for any $ k\in\mathbb{N}_+ $. If $ \frac{R}{6}>\delta(x) $, we can choose $ \overline{x}\in\pa\om $ such that $ |x-\overline{x}|=\delta(x) $. Then it follows from $ \eqref{Lipesu} $ that for any $ k\in\mathbb{N}_+ $,
\begin{align}
|G_{\va,\lambda}(x,y)|&=|G_{\va,\lambda}(x,y)-G_{\va,\lambda}(\overline{x},y)|\leq|x-\overline{x}|\|\nabla_1G_{\va,\lambda}(\cdot,y)\|_{L^{\infty}(\om(x,\frac{R}{6}))}\nonumber\\
&\leq\frac{C_{k,\theta_0}\delta(x)}{(1+|\lambda|R^2)^kR}\left(\dashint_{\om(x,\frac{R}{3})}|G_{\va,\lambda}(\cdot,y)|^2\right)^{\frac{1}{2}}.\nonumber
\end{align}
This, together with $ \eqref{ww56} $, implies that for any $ k\in\mathbb{N}_+ $,
\begin{align}
|G_{\va,\lambda}(x,y)|\leq\frac{C_{k,\theta_0}\delta(x)}{(1+|\lambda||x-y|^2)^{k}|x-y|^{d-1}}.\nonumber
\end{align}
By applying the same arguments on $ G_{\va,\overline{\lambda}}(x,y) $ and using $ \eqref{duality} $, we have
\begin{align}
|G_{\va,\lambda}(x,y)|\leq\frac{C_{k,\theta_0}}{(1+|\lambda||x-y|^2)^{k}|x-y|^{d-2}}\min\left\{1,\frac{\delta(x)}{|x-y|},\frac{\delta(y)}{|x-y|}\right\}.\label{ww57}
\end{align}
Noticing $ (\mathcal{L}_{\va}-\lambda I)(G_{\va,\lambda}^{\gamma}(\cdot,y))=0 $ for any $ 1\leq\gamma\leq m $ in $ \om\backslash\{y\} $, it follows from  $ \eqref{Lipesu} $ and $ \eqref{ww56} $ that
\be
\begin{aligned}
|\nabla_1G_{\va,\lambda}(x,y)|&\leq \frac{C_{\theta_0}}{R}\left(\dashint_{\om(x,\frac{1}{2}R)}|G_{\va,\lambda}(\cdot,y)|^{2}\right)^{\frac{1}{2}}\leq\frac{C_{k,\theta_0}}{(1+|\lambda|R^2)^{k}R^{d-1}}\\
&\leq \frac{C_{k,\theta_0}}{(1+|\lambda||x-y|^2)^{k}|x-y|^{d-1}}.
\end{aligned}\label{weds}
\ee
for any $ k\in\mathbb{N}_+ $. Again, in view of $ \eqref{duality} $, we can see that $ \eqref{Green Lipschitz 11} $ is true. Similarly, if we use the estimate
\begin{align}
|G_{\va,\lambda}(x,y)|\leq\frac{C_{k,\theta_0}\delta(y)}{(1+|\lambda||x-y|^2)^{k}|x-y|^{d-1}}\nonumber
\end{align}
for the second inequality of $ \eqref{weds} $, it is easy to show that
\begin{align}
|\nabla_1G_{\va,\lambda}(x,y)|\leq \frac{C_{k,\theta_0}\delta(y)}{(1+|\lambda||x-y|^2)^{k}|x-y|^{d}},\label{ww70}
\end{align}
which gives the proof of $ \eqref{Green Lipschitz 12} $. Similarly, $ \eqref{Green Lipschitz 13} $ can be derived by using $ \eqref{duality} $ and $ \eqref{ww70} $. Now, we only need to show $ \eqref{Green Lipschitz 14} $. Obviously, for any $ 1\leq\gamma\leq m $,
\begin{align}
(\mathcal{L}_{\va}-\lambda I)(\nabla_2G_{\va,\lambda}^{\gamma}(\cdot,y))=0\quad\text{in}\quad\om\backslash\{y\}. \nonumber
\end{align} 
For any $ k_1,k_2\in\mathbb{N}_+ $, $ k=k_1+k_2 $, using $ \eqref{Green Lipschitz 11} $, we have
\begin{align}
|\nabla_1\nabla_2G_{\va,\lambda}(x,y)|&\leq\frac{C_{k,\theta_0}}{(1+|\lambda|R^2)^{k_2}R}\left(\dashint_{\om(x,\frac{1}{2}R)}|\nabla_2G_{\va,\lambda}(\cdot,y)|^{2}\right)^{\frac{1}{2}}\nonumber\\
&\leq\frac{C_{k,\theta_0}}{(1+|\lambda|R^2)^{k_2}R}\frac{C_{k,\theta_0}}{(1+|\lambda|R^2)^{k_1}R^{d-1}}\leq\frac{C_{k,\theta_0}}{(1+|\lambda||x-y|^2)^{k}|x-y|^{d}},\nonumber
\end{align}     
which completes the proof for the case $ d\geq 3 $. Next, we consider the case $ d=2 $. Set $ r=|x_0-y_0| $ with $ x_0\neq y_0 $ and $ x_0,y_0\in\om $. For $ F\in C_0^{\infty}(\om(x_0,\frac{r}{4});\mathbb{C}^m) $, let $ v_{\va,\lambda} $ satisfy $ (\mathcal{L}_{\va}-\overline{\lambda} I)(v_{\va,\lambda})=F $ in $ \om $ and $ v_{\va,\lambda}=0 $ on $ \pa\om $. Then $ v_{\va,\lambda}(y)=\int_{\om}\overline{G_{\va,\lambda}(\cdot,y)}F(\cdot) $ due to representation formula $ \eqref{Representation formula} $. Since $ F\equiv 0 $ in $ \om\backslash\om(x_0,\frac{r}{4}) $, we have, $ (\mathcal{L}_{\va}-\overline{\lambda} I)(v_{\va,\lambda})=0 $ in $ \om\backslash\om(x_0,\frac{r}{4}) $. Owing to $ \eqref{Linfty} $, $ \eqref{impoestim} $, Poincaré's inequality and Hölder's inequality, we can get, for any $ p>2 $,
\begin{align}
|v_{\va,\lambda}(y_0)|&\leq C_{\theta_0}\left(\dashint_{\om(y_0,\frac{r}{4})}|v_{\va,\lambda}|^2\right)^{\frac{1}{2}}\leq C_{\theta_0}\left(\dashint_{\om(y_0,\frac{r}{4})}|v_{\va,\lambda}|^p\right)^{\frac{1}{p}}=C_{\theta_0}\delta(y_0)r^{-\frac{2}{p}}\left(\int_{\om}|\nabla v_{\va,\lambda}|^p\right)^{\frac{1}{p}}\nonumber\\
&\leq C_{\theta_0}\delta(y_0)r^{-\frac{2}{p}}\|F\|_{L^{\frac{2p}{p+2}}(\om)}\leq C_{\theta_0}\delta(y_0)r^{-\frac{2}{p}}\|F\|_{L^{\frac{2p}{p+2}}(\om(x_0,\frac{r}{4}))} \leq C_{\theta_0}\delta(y_0)\|F\|_{L^{2}(\om(x_0,\frac{r}{4}))}.\nonumber
\end{align}
By duality arguments, this implies that
\begin{align}
\left(\dashint_{\om(x_0,\frac{r}{4})}|G_{\va,\lambda}(\cdot,y_0)|^2\right)^{\frac{1}{2}}\leq \frac{C_{\theta_0}\delta(y_0)}{r}\leq \frac{C_{\theta_0}\delta(y_0)}{|x_0-y_0|}.\label{ww22122}
\end{align}
According to the Lipschitz estimates $ \eqref{Lipesu} $ and $ \eqref{ww22122} $, then for any $ k\in\mathbb{N}_+ $,
\begin{align}
|\nabla_1G_{\va,\lambda}(x_0,y_0)|&\leq \frac{C_{k,\theta_0}}{(1+|\lambda|r^2)^kr}\left(\dashint_{\om(x_0,\frac{r}{4})}|G_{\va,\lambda}(\cdot,y_0)|^2\right)^{\frac{1}{2}}\leq \frac{C_{k,\theta_0}\delta(y_0)}{(1+|\lambda||x_0-y_0|^2)^k|x_0-y_0|^2}.\nonumber
\end{align}
Then $ \eqref{Green Lipschitz 12} $ is proved and $ \eqref{Green Lipschitz 13} $ follows directly by considering same estimates of $ G_{\va,\overline{\lambda}}(x,y) $. Moreover, for $ v_{\va,\lambda} $, we have $ (\mathcal{L}_{\va}-\overline{\lambda} I)(v_{\va,\lambda}-(v_{\va,\lambda})_{y_0,\frac{r}{4}})=\overline{\lambda} (v_{\va,\lambda})_{y_0,\frac{r}{4}} $ in $ \om $. By using $ \eqref{L2uva} $, $ \eqref{Lipesu} $, $ \eqref{impoestim} $, Hölder's inequality and Poincaré's inequality, it can be obtained that
\begin{align}
|\nabla v_{\va,\lambda}(y_0)|&\leq \frac{C_{\theta_0}}{r}\left(\dashint_{\om(y_0,\frac{r}{4})}|v_{\va,\lambda}-(v_{\va,\lambda})_{y_0,\frac{r}{4}}|^2\right)^{\frac{1}{2}}+C_{\theta_0}(1+|\lambda|r^2)^n|\lambda|r|(v_{\va,\lambda})_{y_0,\frac{r}{4}}|\nonumber\\
&\leq C_{\theta_0}\left(\dashint_{\om(y_0,\frac{r}{4})}|\nabla v_{\va,\lambda}|^2\right)^{\frac{1}{2}}+C_{\theta_0}(1+|\lambda|r^2)^n|\lambda|r\left(\dashint_{\om(y_0,\frac{r}{4})}|v_{\va,\lambda}|^2\right)^{\frac{1}{2}}\nonumber\\
&=C_{\theta_0}r^{-\frac{2}{p}}\left(\int_{\om}|\nabla v_{\va,\lambda}|^p\right)^{\frac{1}{p}}+C_{\theta_0}(1+|\lambda|r^2)^n\|F\|_{L^{2}(\om(x_0,\frac{r}{4}))}\nonumber\\
&\leq C_{\theta_0}r^{-\frac{2}{p}}\|F\|_{L^{\frac{2p}{p+2}}(\om(x_0,\frac{r}{4}))}+C_{\theta_0}(1+|\lambda|r^2)^n\|F\|_{L^{2}(\om(x_0,\frac{r}{4}))}\nonumber\\
&\leq C_{\theta_0}(1+|\lambda|r^2)^n\|F\|_{L^{2}(\om(x_0,\frac{r}{4}))},\nonumber
\end{align}
when $ 2<p<\infty $. In view of the definition of $ v_{\va,\lambda} $, we can get that
\begin{align}
\left|\int_{\om(x_0,\frac{r}{4})}\nabla_2G_{\va,\lambda}(\cdot,y_0)\overline{F(\cdot)}\right|\leq C_{\theta_0}(1+|\lambda|r^2)^n\|F\|_{L^{2}(\om(x_0,\frac{r}{4}))}.\label{ww2112222}
\end{align}
Again, by duality arguments, it can be inferred that
\begin{align}
\left(\dashint_{\om(x_0,\frac{r}{4})}|\nabla_2G_{\va,\lambda}(\cdot,y_0)|^2\right)^{\frac{1}{2}}\leq \frac{C_{\theta_0}(1+|\lambda|r^2)^n}{r}.\label{ww22122667}
\end{align}
Similarly, with the help of Lipschitz estimates $ \eqref{Lipesu} $ and $ \eqref{ww22122667} $, for any $ k\in\mathbb{N}_+ $,
\begin{align}
|\nabla_1\nabla_2G_{\va,\lambda}(x_0,y_0)|&\leq \frac{C_{k,\theta_0}}{(1+|\lambda|r^2)^{k+n}r}\left(\dashint_{\om(x_0,\frac{r}{4})}|\nabla_2G_{\va,\lambda}(z,y_0)|^2dz\right)^{\frac{1}{2}}\nonumber\\
&\leq \frac{C_{k,\theta_0}(1+|\lambda|r^2)^n}{(1+|\lambda|r^2)^{k+n}r^2}\leq \frac{C_{k,\theta_0}}{(1+|\lambda||x_0-y_0|^2)^k|x_0-y_0|^2}.\nonumber
\end{align}
This completes the proof of $ \eqref{Green Lipschitz 14} $. Then we only need to show $ \eqref{Green Lipschitz 11} $. To begin with, we note that if $ |y-z|<\frac{1}{2}|x-y| $, then
\be
\begin{aligned}
|\nabla_1(G_{\va,\lambda}(x,y)-G_{\va,\lambda}(x,z))|&\leq\|\nabla_1\nabla_2G_{\va,\lambda}(x,\cdot)\|_{L^{\infty}(\om(y,\frac{|x-y|}{2}))}|y-z|\\
&\leq \frac{C_{k,\theta_0}|y-z|}{(1+|\lambda||x-y|^2)^k|x-y|^2}\leq \frac{C_{k,\theta_0}}{(1+|\lambda||x-y|^2)^k|x-y|},
\end{aligned}\label{lianggexiangjian}
\ee
for any $ k\in\mathbb{N}_+ $. If $ \frac{1}{4}|x_0-y_0|\leq\delta(y_0) $, we get that
\begin{align}
|x-x_0|\leq \frac{1}{2}|x-y_0|\text{ for any }x\in B\left(x_0,\frac{1}{4}|x_0-y_0|\right).\label{ww39}
\end{align}
At first, there exists a point $ \overline{y}\in\partial\Omega $ (see Figure 1), such that
\begin{align}
 (x_0-\overline{y})//(x_0-y_0)\quad\text{and}\quad (x_0-\overline{y})\cdot(x_0-y_0)>0.\nonumber
\end{align}
Then according to the fact that $ \frac{1}{4}|x_0-y_0|\leq\delta(y_0)\leq |y_0-\overline{y}| $, there always exists a positive integer $ N\in\mathbb{N}_+ $ and a sequence of points $ \left\{y_j\right\}_{j=1}^{N}\subset\om $ such that
\begin{align}
&y_{j}=x_0+\frac{5}{4}(y_{j-1}-x_0),\ j=1,\ldots, N,\nonumber\\
&\frac{1}{4}|x_0-y_{N-1}|\leq |y_{N-1}-\overline{y}|\quad\text{and}\quad \frac{1}{4}|x_0-y_{N}|>|y_{N}-\overline{y}|.\nonumber
\end{align}
Moreover, in view of the definition of $ \left\{y_j\right\}_{j=1}^{N} $, one can obtain that
\begin{equation}
|x_0-y_{j}|=\frac{5}{4}|x_0-y_{j-1}|=\ldots=\left(\frac{5}{4}\right)^j|x_0-y_0|,\ \ j=1,\ldots, N.
\end{equation}

\centerline{
\begin{tikzpicture}
\filldraw[color=black, fill=gray!50] (-5,0) circle (0.5);
\filldraw[color=black, fill=black](-5,0) circle (0.05);
\filldraw[color=black, fill=black](-3,0) circle (0.05);
\filldraw[color=black, fill=black](-2.5,0) circle (0.05);
\filldraw[color=black, fill=black](-1.875,0) circle (0.05);
\filldraw[color=black, fill=black](-1.093,0) circle (0.05);
\filldraw[color=black, fill=black](5.125,0) circle (0.05);
\filldraw[color=black, fill=black](8,0) circle (0.05);
\node[color=black] at (-5,0.2) {$ x_0 $};
\node[color=black] at (-3,0.2) {$ y_0 $};
\node[color=black] at (-2.5,0.2) {$ y_1 $};
\node[color=black] at (-1.875,0.2) {$ y_2 $};
\node[color=black] at (-1.093,0.2) {$ y_3 $};
\node[color=black] at (5.125,0.2) {$ y_N $};
\node[color=black] at (8.2,0.2) {$ \overline{y} $};
\draw [color=black, thick][-](-5,0) -- (8,0);
\draw [color=black, thick][dotted](3,0.2) -- (4.06,0.2);
\draw [color=black, thick][dashed](8,1.5) -- (8,-0.5);
\node[color=black] at (8.4,0.9) {$ \partial\Omega $};
\end{tikzpicture}
}
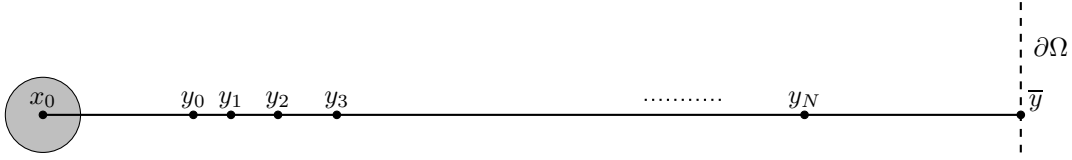
\centerline{Figure 1.}
\noindent
By applying $ \eqref{lianggexiangjian} $, we can obtain that for any $ k\in\mathbb{N}_+ $,
\begin{align}
|\nabla_1 (G_{\va,\lambda}(x_0,y_0)-G_{\va,\lambda}(x_0,y_1))|\leq \frac{C_{k,\theta_0}}{(1+|\lambda||x_0-y_0|^2)^k|x_0-y_0|}.\label{ww431}
\end{align}
Similarly, since $ |y_{j+1}-y_{j}|=\frac{1}{4}|x_0-y_j|<\frac{1}{2}|x_0-y_j| $, we have
\begin{equation}
|\nabla_1 (G_{\va,\lambda}(x_0,y_j)-G_{\va,\lambda}(x_0,y_{j+1}))|\leq \frac{C_{k,\theta_0}}{(1+|\lambda||x_0-y_j|^2)^k|x_0-y_j|},\label{ww43-i}
\end{equation}
for any $ j=1,\ldots,N-1 $. Finally, owing to the simple observation that 
\be
\frac{1}{4}|x_0-y_N|>|y_N-\overline{y}|\geq\delta(y_N), \nonumber
\ee
it can be inferred that
\begin{equation}
\begin{aligned}
|\nabla_1 G_{\va,\lambda}(x_0,y_N)|&\leq\frac{C_{k,\theta_0}\delta(y_N)}{(1+|\lambda||x_0-y_N|^2)^k|x_0-y_N|^2}\leq\frac{C_{k,\theta_0}}{(1+|\lambda||x_0-y_N|^2)^k|x_0-y_N|},\\
\end{aligned}\label{ww43-N}
\end{equation}
where we have used $ \eqref{Green Lipschitz 12} $. From \eqref{ww431}-\eqref{ww43-N}, we can obtain that
\begin{align}
|\nabla_1 G_{\va,\lambda}(x_0,y_0)|&\leq \sum_{j=0}^N\frac{C_{k,\theta_0}}{(1+|\lambda||x_0-y_j|^2)^k|x_0-y_j|}\nonumber\\
&\leq \sum_{j=0}^N\left(\frac{4}{5}\right)^j\frac{C_{k,\theta_0}}{(1+|\lambda|\left(\frac{5}{4}\right)^{2j}|x_0-y_0|^2)^k|x_0-y_0|}\nonumber\\
&\leq \frac{C_{k,\theta_0}}{(1+|\lambda||x_0-y_0|^2)^k|x_0-y_0|},\nonumber
\end{align}
for any $ k\in\mathbb{N}_+ $ and $ \frac{1}{4}|x_0-y_0|\leq\delta(y_0) $. This, together with $ \eqref{Green Lipschitz 12} $ and $ \eqref{duality} $, gives $ \eqref{Green Lipschitz 11} $.
\end{proof}
\begin{rem}
Under the same assumptions of Theorem \ref{Green2}, if we further assume that $ A $ satisfies $ \eqref{Hol} $, we can improve $ \eqref{Pointwise estimates for Green functions 4} $ to the sharp estimate. That is, for any $ k\in\mathbb{N}_+ $,
\begin{align}
|G_{\va,\lambda}(x,y)|\leq\frac{C_{k,\theta_0}}{(1+|\lambda||x-y|^2)^k}\left\{1+\ln\left(\frac{R_0}{|x-y|}\right)\right\},\label{sharppointwise}
\end{align}
where $ C_{k,\theta_0} $ depends only on $ \mu,m,\tau,\nu,k,\theta_0,\eta $ and $ \om $. The proof is similar to Theorem \ref{lipgg}. Like what we have done in $ \eqref{lianggexiangjian} $, one can find that if $ |y-z|<\frac{1}{2}|x-y| $, then
\be
\begin{aligned}
|G_{\va,\lambda}(x,y)-G_{\va,\lambda}(x,z)|&\leq\|\nabla_2G_{\va,\lambda}(x,\cdot)\|_{L^{\infty}(\om(y,\frac{|x-y|}{2}))}|y-z|\\
&\leq \frac{C_{k,\theta_0}|y-z|}{(1+|\lambda||x-y|^2)^k|x-y|}\leq \frac{C_{k,\theta_0}}{(1+|\lambda||x-y|^2)^k},
\end{aligned}\label{lianggexiangjian2}
\ee
where we have used $ \eqref{Green Lipschitz 11} $. Then using almost the same arguments, we can obtain that
\begin{align}
|G_{\va,\lambda}(x_0,y_j)-G_{\va,\lambda}(x_0,y_{j+1})|\leq \frac{C_{k,\theta_0}}{(1+|\lambda||x_0-y_0|^2)^k},\label{ww431111}
\end{align}
for any $ j=0,1,2,\ldots,N-1 $. Also, it can be obtained that
\begin{equation}
\begin{aligned}
|G_{\va,\lambda}(x_0,y_N)|&\leq\frac{C_{k,\theta_0}}{(1+|\lambda||x_0-y_N|^2)^k},\\
\end{aligned}\label{ww43-N111}
\end{equation}
for any $ k\in\mathbb{N}_+ $, where we have used $ \eqref{Boundedness estimates} $. Noticing that the boundedness of $ N $, $ \eqref{lnN} $, is also true in this case, it follows from direct computations that
\begin{align}
|G_{\va,\lambda}(x_0,y_0)|&\leq \sum_{j=0}^N\frac{C_{k,\theta_0}}{(1+|\lambda||x_0-y_j|^2)^k}\leq \frac{C_{k,\theta_0}}{(1+|\lambda||x_0-y_0|^2)^k}\left\{1+\ln\left(\frac{R_0}{|x_0-y_0|}\right)\right\},\nonumber
\end{align}
which completes the proof.
\end{rem}
\begin{lem}
For $ \va\geq 0 $ and $ d\geq 2 $, let $ \lambda\in\Sigma_{\theta_0}\cup\{0\} $ with $ \theta_0\in(0,\frac{\pi}{2}) $ and $ \om $ be a bounded $ C^1 $ domain in $ \R^d $. Suppose that $ A $ satisfies $ \eqref{sy} $, $ \eqref{el} $, $ \eqref{pe} $ and $ \eqref{VMO} $. Then for any $ x\in\om $, $ \sigma\in(0,1) $ and $ k\in\mathbb{N}_+ $, 
\begin{align}
\int_{\om}|G_{\va,\lambda}(x,y)|dy&\leq 
C_{\theta_0}(R_0^{-2}+|\lambda|)^{-1}\cdot\left\{\begin{matrix}
1&\text{if}&d\geq 3,\\
(1+|\lambda|R_0^{2})^{\sigma}&\text{if}&d=2,
\end{matrix}\right.\label{*}\\
\int_{\om}|\nabla_2G_{\va,\lambda}(x,y)|dy&\leq C_{\theta_0}(R_0^{-2}+|\lambda|)^{-\frac{1}{2}},\label{**}
\end{align}
where $ C_{\theta_0} $ depends only on $ \mu,d,m,\sigma,\omega(t),\theta_0 $ and $ \om $. If we further assume that $ A $ satisfies $ \eqref{Hol} $ and $ \om $ is a $ C^{1,\eta} $ domain with $ 0<\eta<1 $, then
\begin{align}
\int_{\om}|\nabla_1G_{\va,\lambda}(x,y)|dy&\leq C_{\theta_0}(R_0^{-2}+|\lambda|)^{-\frac{1}{2}},\label{***}
\end{align}
where $ C $ depends only on $ \mu,d,m,\tau,\nu,\theta_0,\eta $ and $ \om $.
\end{lem}
\begin{proof}
By using $ \eqref{Greene} $ with $ k=2 $, it is easy to infer that 
\be
\begin{aligned}
\int_{\om}|G_{\va,\lambda}(x,y)|dy&\leq \int_{B(x,R_0)}\frac{C_{\theta_0}}{(1+|\lambda||x-y|^2)^{2}}\frac{1}{|x-y|^{d-2}}dy=\int_{0}^{R_0}\frac{C_{\theta_0}\rho}{(1+|\lambda|\rho^2)^{2}}d\rho\\
&=-\left.\frac{C_{\theta_0}}{|\lambda|(1+|\lambda|\rho^2)}\right|_0^{R_0}=C_{\theta_0}\left\{\frac{1}{|\lambda|}-\frac{1}{|\lambda|(1+|\lambda|R_0^2)}\right\}\leq \frac{C_{\theta_0}}{|\lambda|}.
\end{aligned}\label{ww25} 
\ee
On the other hand, it can be easily seen that,
\begin{align}
\int_{\om}|G_{\va,\lambda}(x,y)|dy&\leq \int_{B(x,R_0)}\frac{C_{\theta_0}}{(1+|\lambda||x-y|^2)^{2}}\frac{1}{|x-y|^{d-2}}dy\leq\int_{0}^{R_0}C_{\theta_0}\rho d\rho\leq C_{\theta_0}R_0^2.\nonumber
\end{align}
This, together with $ \eqref{ww25} $, implies $ \eqref{*} $ with $ d\geq 3 $. For $ d=2 $, owing to $ \eqref{preliminary} $, we can get that, for any $ \sigma\in(0,1) $,
\begin{align}
\int_{\om}|G_{\va,\lambda}(x,y)|dy&\leq\sum_{j=0}^{\infty}\int_{\om(x,2^{-j}R_0)\backslash\om(x,2^{-j-1}R_0)}\frac{C_{\theta_0}R_0^{\sigma}}{(1+|\lambda||x-y|^2)^2|x-y|^{\sigma}}dy\nonumber\\
&\leq\sum_{j=0}^{\infty}\frac{C_{\theta_0}(2^{-j}R_0)^2R_0^{\sigma}}{(1+|\lambda|(2^{-j}R_0)^2)^{1-\sigma}(2^{-j}R_0)^{\sigma}}\leq C_{\theta_0}(1+|\lambda|R_0^{2})^{\sigma}(R_0^{-2}+|\lambda|)^{-1}.\nonumber
\end{align}
For $ R>0 $, consider the annulus $ \om(x,2R)\backslash \om(x,R) $. One can use small $ d $ dimensional balls, whose radius are $ \frac{5}{8}R $ and centers are on $ \pa B(x,\frac{3}{2}R)\cap\om $, to cover the annulus $ \om(x,2R)\backslash \om(x,R) $. We denote these small balls as $ \left\{B(x_i,\frac{5}{8}R)\right\}_{i=1}^{N} $. Obviously, $ N\leq C\frac{\pi((2R)^d-R^d)}{\pi(\frac{5}{8} R)^d}\leq C $, where $ C $ is a constant, independent of $ R $. Then by using Hölder's inequality, we have
\begin{align}
\int_{\om(x,2R)\backslash \om(x,R)}|\nabla_2 G_{\va,\lambda}(x,y)|dy&\leq C\left(\int_{\om(x,2R)\backslash \om(x,R)}|\nabla_2 G_{\va,\lambda}(x,y)|^2dy\right)^{\frac{1}{2}}R^{\frac{d}{2}}\nonumber\\
&\leq C\sum_{i=1}^{N}\left(\int_{\om(x_i,\frac{5}{8} R)}|\nabla_2 G_{\va,\lambda}(x,y)|^2dy\right)^{\frac{1}{2}}R^{\frac{d}{2}}.\nonumber
\end{align}
In view of Caccioppoli's inequality $ \eqref{Cau2} $ and $ \eqref{Greene} $, if $ d\geq 3 $, then for any $ k\in\mathbb{N}_+ $,
\begin{align}
\left(\int_{\om(x_i,\frac{5}{8} R)}|\nabla_2 G_{\va,\lambda}(x,y)|^2dy\right)^{\frac{1}{2}}&\leq\frac{C_{k,\theta_0}}{(1+|\lambda|R^2)^kR}\left(\int_{\om(x_i,\frac{2}{3}R)}|G_{\va,\lambda}(x,y)|^2dy\right)^{\frac{1}{2}}\nonumber\\
&\leq\frac{C_{k,\theta_0}R^{\frac{d}{2}}}{(1+|\lambda|R^2)^kR^{d-1}}\leq \frac{C_{k,\theta_0}}{(1+|\lambda|R^2)^kR^{\frac{d}{2}-1}}.\nonumber
\end{align}
This implies that
\begin{align}
\int_{\om(x,2R)\backslash \om(x,R)}|\nabla_2 G_{\va,\lambda}(x,y)|dy&\leq \frac{C_{k,\theta_0}R}{(1+|\lambda|R^2)^k}.\nonumber
\end{align}
Thus
\begin{align}
\int_{\om}|\nabla_2 G_{\va,\lambda}(x,y)|dy&\leq\sum_{i=0}^{\infty}\int_{\om(x,2^{-j}R_0)\backslash \om(x,2^{-j-1}R_0)}|\nabla_2 G_{\va,\lambda}(x,y)|dy\leq\sum_{i=0}^{\infty}\frac{C_{\theta_0}(2^{-j}R_0)}{(1+|\lambda|(2^{-j}R_0)^2)^{\frac{3}{2}}}\nonumber\\
&\approx\int_{B(0,2R_0)}\frac{1}{(1+|\lambda||x|^2)^{\frac{3}{2}}}\frac{1}{|x|^{d-1}}dy\leq C(R_0^{-2}+|\lambda|)^{-\frac{1}{2}},\nonumber
\end{align}
which gives the proof of $ \eqref{**} $ with $ d\geq 3 $. If $ d=2 $, to show $ \eqref{**} $, we have
\be
\begin{aligned}
\int_{\om(x,2R)\backslash \om(x,R)}|\nabla_2 G_{\va,\lambda}(x,y)|dy&\leq C\left(\int_{\om(x,2R)\backslash \om(x,R)}|\nabla_2 G_{\va,\lambda}(x,y)|^2dy\right)^{\frac{1}{2}}R\\
&\leq C\sum_{i=1}^{N}\left(\int_{\om(x_i,\frac{5}{8} R)}|\nabla_2 G_{\va,\lambda}(x,y)|^2dy\right)^{\frac{1}{2}}R.
\end{aligned}\label{suibian}
\ee
Choose $ f\in C_{0}^{\infty}(\om(x_i,\frac{5}{8}R);\mathbb{C}^{m\times d}) $, $ w_{\va,\lambda} $ such that $ (\mathcal{L}_{\va}-\lambda I)(w_{\va,\lambda})=\operatorname{div}(f) $ in $ \om $ and $ w_{\va,\lambda}=0 $ in $ \pa\om $. Then $ w_{\va,\lambda}(z)=\int_{\om}\nabla_2G_{\va,\lambda}(z,\cdot)\overline{f(\cdot)} $ due to $ \eqref{Representation formula} $. In view of $ \eqref{Linfty} $, $ \eqref{impoestim2} $, Hölder's inequality, the fact that $ (\mathcal{L}_{\va}-\lambda I)(w_{\va,\lambda})=0 $ in $ \om\backslash\om(x_i,\frac{5}{8}R) $ and $ \om(x,\frac{1}{20}R)\subset\om\backslash\om(x_i,\frac{5}{8}R) $, it is not hard to deduce that for any $ 2<p<\infty $ and $ k\in\mathbb{N}_+ $,
\begin{align}
|w_{\va,\lambda}(x)|&\leq \frac{C_{k,\theta_0}}{(1+|\lambda|R^2)^k}\left(\dashint_{\om(x,\frac{1}{20}R)}|w_{\va,\lambda}|^2\right)^{\frac{1}{2}}\leq \frac{C_{k,\theta_0}R^{-1+\frac{2}{p}}}{(1+|\lambda|R^2)^k}\|w_{\va,\lambda}\|_{L^{\frac{2p}{p-2}}(\om)}\nonumber\\
&\leq \frac{C_{k,\theta_0}R^{-1+\frac{2}{p}}}{(1+|\lambda|R^2)^k}\|f\|_{L^{\frac{p}{p-1}}(\om)}\leq \frac{C_{k,\theta_0}}{(1+|\lambda|R^2)^k}\|f\|_{L^{2}(\om(x_i,\frac{5}{8}R))}.\nonumber
\end{align}
Then by using the representation of $ w_{\va,\lambda} $ given above, we have
\begin{align}
\left|\int_{\om(x_i,\frac{5}{8}R)}\nabla_2G_{\va,\lambda}(x,y)\overline{f(y)}dy\right|\leq \frac{C_{k,\theta_0}}{(1+|\lambda|R^2)^k}\|f\|_{L^{2}(\om(x_i,\frac{5}{8}R))}.\nonumber
\end{align}
By duality arguments, we can infer that
\begin{align}
\left(\int_{\om(x_i,\frac{5}{8}R)}|\nabla_2G_{\va,\lambda}(x,y)|^2dy\right)^{\frac{1}{2}}\leq \frac{C_{k,\theta_0}}{(1+|\lambda|R^2)^k}.\nonumber
\end{align}
With the help of $ \eqref{suibian} $, it can be seen that for any $ k\in\mathbb{N}_+ $,
\begin{align}
\int_{\om(x,2R)\backslash \om(x,R)}|\nabla_2 G_{\va,\lambda}(x,y)|dy&\leq \frac{C_{k,\theta_0} R}{(1+|\lambda|R^2)^k}.\label{prepa}
\end{align}
Then according to the calculations with $ d\geq 3 $,
\begin{align}
\int_{\om}|\nabla_2 G_{\va,\lambda}(x,y)|dy&\leq\sum_{i=0}^{\infty}\int_{\om(x,2^{-j}R_0)\backslash \om(x,2^{-j-1}R_0)}|\nabla_2 G_{\va,\lambda}(x,y)|dy\nonumber\\
&\leq\sum_{i=0}^{\infty}\frac{C_{\theta_0} 2^{-j}R_0}{(1+|\lambda|(2^{-j}R_0)^2)^{\frac{3}{2}}}\leq C_{\theta_0}(R_0^{-2}+|\lambda|)^{-\frac{1}{2}}.\nonumber
\end{align}
For the proof of $ \eqref{***} $, we simply note that by $ \eqref{Green Lipschitz 11} $ with $ k=2 $,
\begin{align}
\int_{\om}|\nabla_1G_{\va,\lambda}(x,y)|dy&\leq\int_{\om}\frac{C_{\theta_0}}{(1+|\lambda||x-y|^2)^2|x-y|^{d-1}}dy\leq C_{\theta_0}(R_0^{-2}+|\lambda|)^{-\frac{1}{2}}.\nonumber
\end{align}
Then we can complete the proof.
\end{proof}

\section{Convergence estimates of resolvents}\label{Estimates of convergence of resolvents}

\subsection{Convergence of Green functions}

Now we turn to estimate $ |G_{\va,\lambda}(x,y)-G_{0,\lambda}(x,y)| $, where $ G_{\va,\lambda} $ and $ G_{0,\lambda} $ are Green functions for operators $ \mathcal{L}_{\va}-\lambda I $ and $ \mathcal{L}_0-\lambda I $ respectively. Such estimates are essential in the proof of Theorem \ref{Lpconres} and \ref{LpW1pconreso}.

\begin{thm}[Convergence of Green functions I]\label{Convergence of Green's functions}
For $ \va\geq 0 $, $ d\geq 2 $, $ \lambda\in\Sigma_{\theta_0}\cup\{0\} $ with $ \theta_0\in(0,\frac{\pi}{2}) $, let $ \om $ be a bounded $ C^{1,1} $ domain in $ \mathbb{R}^d $. Suppose that $ A $ satisfies $ \eqref{sy} $, $ \eqref{el} $, $ \eqref{pe} $ and $ \eqref{Hol} $. Then for any $ k\in\mathbb{N}_+ $ and $ x,y\in\om $ with $ x\neq y $,
\begin{align}
|G_{\va,\lambda}(x,y)-G_{0,\lambda}(x,y)|\leq \frac{C_{k,\theta_0}\va }{(1+|\lambda||x-y|^2)^{k}|x-y|^{d-1}},\label{Convergence of Green's functions formula}
\end{align}
where $ C_{k,\theta_0} $ depends only on $ \mu,d,m,\nu,\tau,k,\theta_0 $ and $ \om $.
\end{thm}

For continuous function $ u $ in $ \om $, the nontangential maximal function is defined by
\begin{align}
(u)^*(y)=\sup\left\{|u(x)|:x\in\om\text{ and }|x-y|<C_0\delta(x)\right\}\label{nontangential maximal function}
\end{align}
for $ y\in\pa\om $, where $ C_0=C_0(\om)>1 $ is sufficiently large depending on $ \om $.

\begin{thm}[Nontangential-maximal-function estimates]
For $ \va\geq 0 $ and $ d\geq 2 $, $ \lambda\in\Sigma_{\theta_0}\cup\{0\} $ with $ \theta_0\in(0,\frac{\pi}{2}) $, let $ \om $ be a bounded $ C^{1,\eta} $ domain in $ \mathbb{R}^d $ with $ 0<\eta<1 $. Suppose that $ A $ satisfies $ \eqref{sy} $, $ \eqref{el} $, $ \eqref{pe} $, $ \eqref{Hol} $ and $ 1<p\leq\infty $. For $ g\in L^p(\pa\om;\mathbb{C}^m) $, let $ u_{\va,\lambda} $ be the unique solution to the Dirichlet problem $ (\mathcal{L}_{\va}-\lambda I)(u_{\va,\lambda})=0 $ in $ \om $ and $ u_{\va,\lambda}=g $ on $ \pa\om $ with the property $ (u_{\va,\lambda})^*\in L^p(\pa\om) $. Then
\begin{align}
\|(u_{\va,\lambda})^*\|_{L^p(\pa\om)}\leq C_{\theta_0}\|g\|_{L^p(\pa\om)}.\label{Lp for nontan}
\end{align}
where $ C_{\theta_0} $ depends only on $ \mu,d,m,\nu,\tau,\theta_0,\eta $ and $ \om $. In the case that $ p=\infty $, $ \eqref{Lp for nontan} $ implies,
\begin{align}
\|u_{\va,\lambda}\|_{L^{\infty}(\om)}\leq  C_{\theta_0}\|g\|_{L^{\infty}(\pa\om)}.\label{maximal principle}
\end{align}
\end{thm}
\begin{proof}
By using the definition of Green functions, we can represent $ u_{\va,\lambda} $ by the formula
\begin{align}
u_{\va,\lambda}(x)=\int_{\pa\om}P_{\va,\lambda}(x,y)\overline{g(y)}dS(y),\nonumber
\end{align}
where the Poisson kernel $ P_{\va,\lambda}(x,y)=(P_{\va,\lambda}^{\al\beta}(x,y)) $ for $ \mathcal{L}_{\va}-\lambda I $ in $ \om $ is given by 
\begin{align}
P_{\va,\lambda}^{\al\beta}(x,y)=-n_i(y)a_{ji}^{\gamma\beta}(y/\va)\frac{\pa }{\pa y_j} \{G_{\va,\lambda}^{\al\gamma}(x,y)\}\nonumber
\end{align}
for $ x\in\om $, $ y\in\pa\om $ and $ n(x)=(n_1(x),n_2(x),...,n_d(x)) $ denotes the outward unit normal to $ \pa\om $. In view of the estimate $ \eqref{Green Lipschitz 13} $, we have, for any $ k\in\mathbb{N}_+ $,
\begin{align}
|P_{\va,\lambda}(x,y)|\leq \frac{C_{k,\theta_0}\delta(x)}{(1+|\lambda||x-y|^2)^{k}|x-y|^d}\leq \frac{C_{\theta_0}\delta(x)}{|x-y|^d}.\nonumber
\end{align}
Hence, we can obtain that
\begin{align}
|u_{\va,\lambda}(x)|\leq \int_{\pa\om}\frac{C_{\theta_0}\delta(x)}{|x-y|^d}|g(y)|dS(y)\text{ for any }x\in\om.\nonumber
\end{align}
The rest of the proof is standard, which can be obtained in Theorem 4.6.5 of \cite{Shen2}.
\end{proof}

\begin{lem}
For $ \va\geq 0 $, $ d\geq 2 $, $ \lambda\in\Sigma_{\theta_0}\cup\{0\} $ with $ \theta_0\in(0,\frac{\pi}{2}) $, let $ \om $ be a bounded $ C^{1,1} $ domain in $ \mathbb{R}^d $. Suppose that $ A $ satisfies $ \eqref{sy} $, $ \eqref{el} $, $ \eqref{pe} $ and $ \eqref{Hol} $. If $ \Delta(x_0,3R)\neq\emptyset $, assume that $ u_{\va,\lambda}\in H^1(\om(x_0,3R);\mathbb{C}^m) $ is a solution of the boundary problem
\begin{align}
(\mathcal{L}_{\va}-\lambda I)(u_{\va,\lambda})=0\text{ in }\om(x_0,3R)\quad\text{and}\quad u_{\va,\lambda}=f\text{ on } \Delta(x_0,3R),\nonumber
\end{align}
with $ \|f\|_{L^{\infty}(\om)}<\infty $. If $ \Delta(x_0,3R)=\emptyset $, assume that $ u_{\va,\lambda}\in H^1(B(x_0,3R);\mathbb{C}^m) $ is the weak solution of the interior problem
\begin{align}
(\mathcal{L}_{\va}-\lambda I)(u_{\va,\lambda})=0\text{ in }B(x_0,3R).\nonumber
\end{align}
Then for any $ k\in\mathbb{N}_+ $,
\begin{align}
\|u_{\va,\lambda}\|_{L^{\infty}(\om(x_0,R))}\leq C_{k,\theta_0}\|f\|_{L^{\infty}(\Delta(x_0,3R))}+\frac{C_{k,\theta_0}}{(1+|\lambda|R^2)^{k}}\dashint_{\om(x_0,{3R})}|u_{\va,\lambda}|,\label{ww65}
\end{align}
where $ C_{k,\theta_0} $ depends only on $ \mu,d,m,\theta_0,k,\nu,\tau $ and $ \om $.
\end{lem}
\begin{proof}
We only need to show the case that $ \Delta(x_0,3R)\neq\emptyset $ since the other follows directly from $ \eqref{Linfty} $. If $ f\equiv 0 $, the estimate is also a consequence of $ \eqref{Linfty} $. To treat the general case, let $ v_{\va,\lambda} $ be the solution to $ (\mathcal{L}_{\va}-\lambda I)(v_{\va,\lambda})=0 $ in $ \widetilde{\om} $ with the Dirichlet condition $ v_{\va,\lambda}=f $ on $ \pa\widetilde{\om}\cap\pa\om $ and $ v_{\va,\lambda}=0 $ on $ \pa\widetilde{\om}\backslash\pa\om $, where $ \widetilde{\om} $ is a $ C^{1,1} $ domain such that $ \om(x_0,2R)\subset\widetilde{\om}\subset \om(x_0,3R) $. By the maximum principle $ \eqref{maximal principle} $, we have,
\begin{align} 
\|v_{\va,\lambda}\|_{L^{\infty}(\widetilde{\om})}\leq C_{\theta_0}\|v_{\va,\lambda}\|_{L^{\infty}(\pa\widetilde{\om})}\leq C_{\theta_0}\|f\|_{L^{\infty}(\Delta(x_0,3R))}.\nonumber
\end{align}
This, together with the fact that for any $ k\in\mathbb{N}_+ $,
\begin{align}
\|u_{\va,\lambda}-v_{\va,\lambda}\|_{L^{\infty}(\om_R)}&\leq \frac{C_{k,\theta_0}}{(1+|\lambda|R^2)^{k}}\dashint_{\om_{2R}}|u_{\va,\lambda}-v_{\va,\lambda}|\nonumber\\
&\leq \frac{C_{k,\theta_0}}{(1+|\lambda|R^2)^{k}}\dashint_{\om_{3R}}|u_{\va,\lambda}|+C_{k,\theta_0}\|f\|_{L^{\infty}(\Delta_{3R})},\nonumber
\end{align}
gives the results. 
\end{proof}

\begin{lem}\label{ww74}
For $ \va\geq 0 $ and $ d\geq 2 $, $ \lambda\in\Sigma_{\theta_0}\cup\{0\} $ with $ \theta_0\in(0,\frac{\pi}{2}) $, let $ \om $ be a bounded $ C^{1,1} $ domain in $ \mathbb{R}^d $. Suppose that $ A $ satisfies $ \eqref{sy} $, $ \eqref{el} $, $ \eqref{pe} $ and $ \eqref{Hol} $. Let $ u_{\va,\lambda}\in H^1(\om(x_0,4R);\mathbb{C}^m) $ and $ u_{0,\lambda}\in W^{2,p}(\om(x_0,4R);\mathbb{C}^m) $ for some $ d<p<\infty $. If $ \Delta(x_0,4R)\neq\emptyset $, assume that 
\begin{align}
(\mathcal{L}_{\va}-\lambda I)(u_{\va,\lambda})=(\mathcal{L}_{0}-\lambda I)(u_{0,\lambda})\text{ in }\om(x_0,4R)\quad\text{and}\quad u_{\va,\lambda}=u_{0,\lambda}\text{ on }\Delta(x_0,4R),\nonumber
\end{align}
and if $ \Delta(x_0,4R)=\emptyset $, assume that
\begin{align}
(\mathcal{L}_{\va}-\lambda I)(u_{\va,\lambda})=(\mathcal{L}_{0}-\lambda I)(u_{0,\lambda})\text{ in }B(x_0,4R).\nonumber
\end{align}
Then for any $ k\in\mathbb{N}_+ $,
\be
\begin{aligned}
\|u_{\va,\lambda}-u_{0,\lambda}\|_{L^{\infty}(\om_R)}&\leq 
\frac{C_{k,\theta_0}}{(1+|\lambda|R^2)^k}\dashint_{\om_{4R}}|u_{\va,\lambda}-u_{0,\lambda}|\\
&\quad+C_{k,\theta_0}\va\left\{ R^{1-\frac{d}{p}}\|\nabla^2u_{0,\lambda}\|_{L^p(\om_{4R})}+(1+|\lambda|R^2)\|\nabla u_{0,\lambda}\|_{L^{\infty}(\om_{4R})}\right\},
\end{aligned}\label{67}
\ee
where $ C_{k,\theta_0} $ depends on $ \mu,d,m,\nu,\tau,p,k,\theta_0 $ and $ \om $.
\end{lem}
\begin{proof}
Firstly, we consider the case that $ \Delta(x_0,3R)\neq\emptyset $. Choose a domain $ \widetilde{\om} $, which is $ C^{1,1} $ such that $ \om(x_0,3R)\subset\widetilde{\om}\subset \om(x_0,4R) $. Consider
\begin{align}
w_{\va,\lambda}(x)=u_{\va,\lambda}(x)-u_{0,\lambda}(x)-\va\chi_{j}^{\beta}(x/\va)\frac{\pa u_{0,\lambda}^{\beta}}{\pa x_j}=w_{\va,\lambda}^{(1)}(x)+w_{\va,\lambda}^{(2)}(x)\quad\text{in }\widetilde{\om},\nonumber
\end{align} 
where $ w_{\va,\lambda}^{(1)} $ and $ w_{\va,\lambda}^{(2)} $ are weak solutions for the Drichlet problems
\begin{align}
(\mathcal{L}_{\va}-\lambda I)(w_{\va,\lambda}^{(1)})&=(\mathcal{L}_{\va}-\lambda I)(w_{\va,\lambda})\text{ in }\widetilde{\om}\quad\text{and}\quad  w_{\va,\lambda}^{(1)}\in H_0^1(\widetilde{\om};\mathbb{C}^m),\label{ww63}\\
(\mathcal{L}_{\va}-\lambda I)(w_{\va,\lambda}^{(2)})&=0\text{ in }\widetilde{\om}\quad\text{and}\quad w_{\va,\lambda}^{(2)}=w_{\va,\lambda}\quad\text{on }\pa\widetilde{\om}.\label{ww64}
\end{align}
Since $ w_{\va,\lambda}^{(2)}=w_{\va,\lambda}=-\va\chi(x/\va)\nabla u_{0,\lambda} $ on $ \Delta(x_0,3R) $ and $ \|\chi\|_{L^{\infty}(\om)}\leq C $, it follows from $ \eqref{ww65} $ that for any $ k\in\mathbb{N}_+ $,
\begin{align}
\|w_{\va,\lambda}^{(2)}\|_{L^{\infty}(\om_R)}&\leq C_{k,\theta_0}\va \|\nabla u_{0,\lambda}\|_{L^{\infty}(\Delta_{3R})}+\frac{C_{k,\theta_0}}{(1+|\lambda|R^2)^k}\dashint_{\om_{3R}}|w_{\va,\lambda}^{(2)}|\nonumber\\
&\leq C_{k,\theta_0}\va\|\nabla u_{0,\lambda}\|_{L^{\infty}(\Delta_{3R})}+\frac{C_{k,\theta_0}}{(1+|\lambda|R^2)^k}\dashint_{\om_{3R}}|w_{\va,\lambda}^{(1)}|+\frac{C_{k,\theta_0}}{(1+|\lambda|R^2)^k}\dashint_{\om_{3R}}|w_{\va,\lambda}|\nonumber\\
&\leq \frac{C_{k,\theta_0}}{(1+|\lambda|R^2)^k}\left\{\dashint_{\om_{3R}}|u_{\va,\lambda}-u_{0,\lambda}|+\|w_{\va,\lambda}^{(1)}\|_{L^{\infty}(\om_{3R})}\right\}+C_{k,\theta_0}\va \|\nabla u_{0,\lambda}\|_{L^{\infty}(\om_{3R})}.\nonumber
\end{align}
By using definitions of $ w_{\va,\lambda} $, $ w_{\va,\lambda}^{(1)} $ and $ w_{\va,\lambda}^{(2)} $, this gives
\be
\begin{aligned}
\|u_{\va,\lambda}-u_{0,\lambda}\|_{L^{\infty}(\om_R)}&\leq \frac{C_{k,\theta_0}}{(1+|\lambda|R^2)^k}\dashint_{\om_{3R}}|u_{\va,\lambda}-u_{0,\lambda}|\\
&\quad+C_{k,\theta_0}\left\{\|w_{\va,\lambda}^{(1)}\|_{L^{\infty}(\om_{3R})}+\va\|\nabla u_{0,\lambda}\|_{L^{\infty}(\om_{3R})}\right\}.
\end{aligned}\label{ww66}
\ee
To estimate $ w_{\va,\lambda}^{(1)} $ on $ \om(x_0,3R) $, we use the representation formula to obtain that
\begin{align}
w_{\va,\lambda}^{(1)}(x)=\int_{\widetilde{\om}}\widetilde{G}_{\va,\lambda}(x,y)\overline{(\mathcal{L}_{\va}-\lambda I)(w_{\va,\lambda})}dy,\nonumber
\end{align}
where $ \widetilde{G}_{\va,\lambda}(x,y) $ denotes the matrix of the Green function for $ \mathcal{L}_{\va}-\lambda I $ in $ \widetilde{\om} $. Note that
\be
\begin{aligned}
[(\mathcal{L}_{\va}-\lambda I) w_{\va,\lambda})]^{\al}&=-\va\frac{\pa}{\pa x_i}\left\{F_{kij}^{\al\beta}(x/\va)\frac{\pa^2u_{0,\lambda}^{\beta}}{\pa x_j\pa x_k}\right\}+\va\lambda\chi_{j}^{\al\beta}(x/\va)\frac{\pa u_{0,\lambda}^{\beta}}{\pa x_j}\\
&\quad+\va\frac{\pa}{\pa x_i}\left\{a_{ij}^{\al\delta}(x/\va)\chi_k^{\delta\beta}(x/\va)\frac{\pa^2u_{0,\lambda}^{\beta}}{\pa x_j\pa x_k}\right\}.
\end{aligned}\label{ww69}
\ee
Then by using the representation formula of Green functions $ \eqref{repre} $ and $ \eqref{Representation formula} $, 
\begin{align}
w_{\va,\lambda}^{(1)\gamma}(x)&=\va\int_{\widetilde{\om}}\frac{\pa}{\pa y_i}\{\widetilde{G}_{\va,\lambda}^{\gamma\al}(x,y)\}\overline{\left[F_{kij}^{\al\beta}(y/\va)-a_{ij}^{\al\delta}(y/\va)\chi_k^{\delta\beta}(y/\va)\right]\frac{\pa^2 u_{0,\lambda}^{\beta}}{\pa y_j\pa y_k}}dy\nonumber\\
&\quad\quad+\va\overline{\lambda}\int_{\widetilde{\om}}\widetilde{G}_{\va,\lambda}^{\gamma\al}(x,y)\overline{\chi_{j}^{\al\beta}(y/\va)\frac{\pa u_{0,\lambda}^{\beta}}{\pa y_j}}dy.\nonumber
\end{align}
Since $ \|F_{kij}\|_{L^{\infty}(\om)},\|\chi_j\|_{L^{\infty}(\om)}\leq C $ and $ p>d $, we have
\begin{align}
\left|\va\lambda\int_{\widetilde{\om}}\widetilde{G}_{\va,\lambda}^{\gamma\al}(x,y)\overline{\chi_{j}^{\al\beta}(y/\va)\frac{\pa u_{0,\lambda}^{\beta}}{\pa y_j}}dy\right|&\leq C_{\theta_0}\va|\lambda|\|\nabla u_{0,\lambda}\|_{L^{\infty}(\om_{4R})}\int_{\widetilde{\om}}|\widetilde{G}_{\va,\lambda}^{\gamma\al}(x,y)|dy\nonumber\\
&\leq C_{\theta_0}\va |\lambda|R^2\|\nabla u_{0,\lambda}\|_{L^{\infty}(\om_{4R})},\nonumber
\end{align}
where we have used $ \eqref{*} $ for the second inequality. Then
\begin{align}
|w_{\va,\lambda}^{(1)}(x)|&\leq C\va\int_{\widetilde{\om}}|\nabla_2\widetilde{G}_{\va,\lambda}(x,y)||\nabla^2u_{0,\lambda}(y)|dy+C_{\theta_0}\va |\lambda|R^2\|\nabla u_{0,\lambda}\|_{L^{\infty}(\om_{4R})}\nonumber\\
&\leq C\va\|\nabla^2u_{0,\lambda}\|_{L^p(\om_{4R})}\left(\int_{\widetilde{\om}}|\nabla_2\widetilde{G}_{\va,\lambda}(x,y)|^{p'}dy\right)^{\frac{1}{p'}}+C_{\theta_0}\va |\lambda|R^2\|\nabla u_{0,\lambda}\|_{L^{\infty}(\om_{4R})}.\nonumber
\end{align}
In view of $ \eqref{Green Lipschitz 11} $ with $ k=0 $, we have $ \left(\int_{\widetilde{\om}}|\nabla_2\widetilde{G}_{\va,\lambda}(x,y)|^{p'}dy\right)^{\frac{1}{p'}}\leq C_{\theta_0}R^{1-\frac{d}{p}} $ and then
\begin{align}
|w_{\va,\lambda}^{(1)}(x)|\leq C_{\theta_0}\va\left\{R^{1-\frac{d}{p}}\|\nabla^2u_{0,\lambda}\|_{L^p(\om_{4R})}+|\lambda|R^2\|\nabla u_{0,\lambda}\|_{L^{\infty}(\om_{4R})}\right\}.\nonumber
\end{align}
This, together with $ \eqref{ww66} $, gives the proof for the case that $ \Delta(x_0,3R)\neq\emptyset $. If $ \Delta(x_0,3R)=\emptyset $, we can choose $ w_{\va,\lambda}^{(1)}$ and $ w_{\va,\lambda}^{(2)} $ by
\begin{align}
(\mathcal{L}_{\va}-\lambda I)(w_{\va,\lambda}^{(1)})&=(\mathcal{L}_{\va}-\lambda I)(w_{\va,\lambda})\text{ in }B\left(x_0,3R\right)\quad\text{and}\quad  w_{\va,\lambda}^{(1)}\in H_0^1\left(B\left(x_0,3R\right);\mathbb{C}^m\right),\nonumber\\
(\mathcal{L}_{\va}-\lambda I)(w_{\va,\lambda}^{(2)})&=0\text{ in }B\left(x_0,3R\right)\quad\text{and}\quad  w_{\va,\lambda}^{(2)}=w_{\va,\lambda}\text{ on }\pa\left(B\left(x_0,3R\right)\right).\nonumber
\end{align}
By using $ \eqref{Linfty} $, it can be seen that for any $ k\in\mathbb{N}_+ $,
\begin{align}
\|w_{\va,\lambda}^{(2)}\|_{L^{\infty}(B_R)}&\leq \frac{C_{k,\theta_0}}{(1+|\lambda|R^2)^k}\dashint_{B_{3R}}|w_{\va,\lambda}^{(2)}|\leq \frac{C_{k,\theta_0}}{(1+|\lambda|R^2)^k}\left\{\dashint_{B_{3R}}|w_{\va,\lambda}^{(1)}|+\dashint_{B_{3R}}|w_{\va,\lambda}|\right\}\nonumber\\
&\leq \frac{C_{k,\theta_0}}{(1+|\lambda|R^2)^k}\left\{\dashint_{B_{3R}}|u_{\va,\lambda}-u_{0,\lambda}|+\|w_{\va,\lambda}^{(1)}\|_{L^{\infty}(B_{3R})}+\va \|\nabla u_{0,\lambda}\|_{L^{\infty}(B_{3R})}\right\}.\nonumber
\end{align}
Then $ \eqref{ww66} $ is still true. Choosing $ \widetilde{\om}=B(x_0,3R) $ and using almost the same arguments for the case that $ \Delta(x_0,3R)\neq\emptyset $, we can complete the proof.
\end{proof}

\begin{proof}[Proof of Theorem \ref{Convergence of Green's functions}]
Fix $ x_0,y_0\in\om $ and $ R=\frac{|x_0-y_0|}{16}>0 $. For $ F\in C_0^{\infty}(\om(y_0,R);\mathbb{C}^m) $, let
\begin{align}
u_{\va,\lambda}(x)=\int_{\om}G_{\va,\lambda}(x,y)\overline{F(y)}dy\quad\text{and}\quad u_{0,\lambda}(x)=\int_{\om}G_{0,\lambda}(x,y)\overline{F(y)}dy.\nonumber
\end{align}
Then $ (\mathcal{L}_{\va}-\lambda I)(u_{\va,\lambda})=(\mathcal{L}_0-\lambda I)(u_{0,\lambda})=F $ in $ \om $ and $ u_{\va,\lambda}=u_{0,\lambda}=0 $ on $ \pa\om $. For any $ x\in\om $ and $ p>d $, using $ \eqref{Green Lipschitz 11} $ with $ k=0 $, it follows by Hölder's inequality that
\begin{align}
|\nabla u_{0,\lambda}(x)|&\leq\left|\int_{\om}\nabla_1 G_{0,\lambda}(x,y)\overline{F(y)}dy\right|\leq C\|F\|_{L^p(\om(y_0,R))}\left(\int_{\om(y_0,R)}|\nabla_1G_{0,\lambda}(x,y)|^{p'}dy\right)^{\frac{1}{p'}}\nonumber\\
&\leq C_{\theta_0}\left(\int_{\om(y_0,R)}\frac{1}{|x-y|^{(d-1)p'}}dy\right)^{\frac{1}{p'}}\|F\|_{L^p(\om(y_0,R))}\leq C_{\theta_0}R^{1-\frac{d}{p}}\|F\|_{L^p(\om(y_0,R))}.\nonumber
\end{align}
To this end, we can obtain the $ L^{\infty} $ estimates of $ \nabla u_{0,\lambda} $, that is,
\begin{align}
\|\nabla u_{0,\lambda}\|_{L^{\infty}(\om)}\leq C_{\theta_0}R^{1-\frac{d}{p}}\|F\|_{L^p(\om(y_0,R))}.\label{nabla infty}
\end{align}
To estimate $ \|u_{\va,\lambda}-u_{0,\lambda}\|_{L^{\infty}(\om(y_0,R))} $, set $ w_{\va,\lambda} $ by
\begin{align}
w_{\va,\lambda}(x)=u_{\va,\lambda}(x)-u_{0,\lambda}(x)-\va\chi_{j}^{\beta}(x/\va)\frac{\pa u_{0,\lambda}^{\beta}}{\pa x_j}=v_{\va,\lambda}(x)+z_{\va,\lambda}(x),\nonumber
\end{align}
where $ v_{\va,\lambda}\in H_0^1(\om;\mathbb{C}^m) $ and $ (\mathcal{L}_{\va}-\lambda I)(v_{\va,\lambda})=(\mathcal{L}_{\va}-\lambda I)(w_{\va,\lambda}) $ in $ \om $. By using $ \eqref{L2uva} $, $ \eqref{W2p} $ and $ \eqref{ww69} $ with $ p=2 $, we can obtain
\begin{align}
\|\nabla v_{\va,\lambda}\|_{L^2(\om)}&\leq C_{\theta_0}\va\|\nabla^2u_{0,\lambda}\|_{L^2(\om)}+C_{\theta_0}\va|\lambda|^{\frac{1}{2}}\|\nabla u_{0,\lambda}\|_{L^2(\om)}\nonumber\\
&\leq C_{\theta_0}\va\|F\|_{L^2(\om)}\leq C_{\theta_0}\va\|F\|_{L^2(\om(y_0,R))}.\nonumber
\end{align}
By Hölder's and Sobolev's inequalities, this implies that if $ d\geq 3 $, we have
\be
\begin{aligned}
\|v_{\va,\lambda}\|_{L^2(\om(y_0,R))}&\leq CR\|v_{\va,\lambda}\|_{L^{\frac{2d}{d-2}}(\om(y_0,R))}\leq CR\|\nabla v_{\va,\lambda}\|_{L^2(\om)}\\
&\leq C_{\theta_0}\va R\|F\|_{L^2(\om(y_0,R))}\leq C_{\theta_0}\va R^{1+\frac{d}{2}-\frac{d}{p}}\|F\|_{L^p(\om(y_0,R))},
\end{aligned}\label{zhongjianchanwu}
\ee
where $ p>d $. If $ d=2 $, one can use the following estimate
\begin{align}
\|v_{\va,\lambda}\|_{L^2(\om(x_0,R))}\leq C\va R\|F\|_{L^2(\om(y_0,R))}.\nonumber
\end{align}
in place of $ \eqref{zhongjianchanwu} $. Indeed, for any $ 2<q<\infty $, by using $ \eqref{ww69} $, Hölder's inequality and Sobolev embdeding theorem that $ W_0^{1,\frac{2q}{q+2}}(\om)\subset L^q(\om) $, it can be got that
\be
\begin{aligned}
\|v_{\va,\lambda}\|_{L^2(\om(y_0,R))}&\leq CR^{1-\frac{2}{q}}\|v_{\va,\lambda}\|_{L^q(\om(y_0,R))}\leq CR^{1-\frac{2}{q}}\|v_{\va,\lambda}\|_{L^q(\om)}\leq CR^{1-\frac{2}{q}}\|\nabla v_{\va,\lambda}\|_{L^{\frac{2q}{q+2}}(\om)}\\
&\leq C_{\theta_0}\va R^{1-\frac{2}{q}}\left\{\|\nabla^2u_{0,\lambda}\|_{L^{\frac{2q}{q+2}}(\om)}+|\lambda|^{\frac{1}{2}}\|\nabla u_{0,\lambda}\|_{L^{\frac{2q}{q+2}}(\om)}\right\}\\
&\leq C_{\theta_0}\va R^{1-\frac{2}{q}}\|F\|_{L^{\frac{2q}{q+2}}(\om(y_0,R))}\leq C_{\theta_0}R\|F\|_{L^{2}(\om(y_0,R))},
\end{aligned}\label{1111}
\ee
where for the forth and fifth inequalities, we have used Theorem \ref{Lp estimates of resolventsf} and $ \eqref{W2p} $. Since $ (\mathcal{L}_{\va}-\lambda I)(z_{\va,\lambda})=0 $ in $ \om $ and $ z_{\va,\lambda}=w_{\va,\lambda} $ on $ \pa\om $, by maximum principle $ \eqref{maximal principle} $,
\begin{align}
\|z_{\va,\lambda}\|_{L^{\infty}(\om)}\leq C_{\theta_0} \|z_{\va,\lambda}\|_{L^{\infty}(\pa\om)}\leq C_{\theta_0}\va \|\nabla u_{0,\lambda}\|_{L^{\infty}(\pa\om)}.\nonumber
\end{align}
In view of $ \eqref{nabla infty} $, we obtain
\begin{align}
\|u_{\va,\lambda}-u_{0,\lambda}\|_{L^2(\om(x_0,R))}&\leq \|w_{\va,\lambda}\|_{L^2(\om(x_0,R))}+C\va R^{\frac{d}{2}}\|\nabla u_{0,\lambda}\|_{L^{\infty}(\om)}\nonumber\\
&\leq\|v_{\va,\lambda}\|_{L^2(\om(x_0,R))}+\|z_{\va,\lambda}\|_{L^2(\om(x_0,R))}+C\va R^{\frac{d}{2}}\|\nabla u_{0,\lambda}\|_{L^{\infty}(\om)}\nonumber\\
&\leq\|v_{\va,\lambda}\|_{L^2(\om(x_0,R))}+C_{\theta_0}\va R^{\frac{d}{2}}\|\nabla u_{0,\lambda}\|_{L^{\infty}(\om)}\nonumber\\
&\leq C_{\theta_0}\va R^{1+\frac{d}{2}-\frac{d}{p}}\|F\|_{L^p(\om(y_0,R))},\nonumber
\end{align}
where $ p>d $ and $ d\geq 3 $. Similarly, if $ d=2 $, we have, for $ p>2 $,
\begin{align}
\|u_{\va,\lambda}-u_{0,\lambda}\|_{L^2(\om(x_0,R))}&\leq\|v_{\va,\lambda}\|_{L^2(\om(x_0,R))}+\|z_{\va,\lambda}\|_{L^2(\om(x_0,R))}+C\va R\|\nabla u_{0,\lambda}\|_{L^{\infty}(\om)}\nonumber\\
&\leq C_{\theta_0}\va R^{2-\frac{2}{p}}\|F\|_{L^p(\om(y_0,R))}.\nonumber
\end{align}
This, together with Lemma \ref{ww74} and $ \eqref{W2p} $, gives
\begin{align}
|u_{\va,\lambda}(x_0)-u_{0,\lambda}(x_0)|\leq C_{\theta_0}\va R^{1-\frac{d}{p}}\|F\|_{L^p(\om(y_0,R))}.\nonumber
\end{align}
Then it follows by duality arguments that
\begin{align}
\left(\int_{\om(y_0,R)}|G_{\va,\lambda}(x_0,y)-G_{0,\lambda}(x_0,y)|^{p'}dy\right)^{\frac{1}{p'}}\leq C_{\theta_0}\va R^{1-\frac{d}{p}}\text{ for any }p>d.\label{ww76}
\end{align}
Finally, since $ (\mathcal{L}_{\va}-\overline{\lambda} I)(G_{\va,\lambda}^{\gamma}(x_0,\cdot))=(\mathcal{L}_0-\overline{\lambda} I)(G_{0,\lambda}^{\gamma}(x_0,\cdot))=0 $ in $ \om(y_0,R) $ for any $ 1\leq \gamma\leq m $, we may invoke Lemma \ref{ww74} again to conclude that for any $ k\in\mathbb{N}_+ $,
\be
\begin{aligned}
|G_{\va,\lambda}(x_0,y_0)-G_{0,\lambda}(x_0,y_0)|&\leq \frac{C_{k,\theta_0}}{(1+|\lambda|R^2)^k}\dashint_{\om(y_0,R)}|G_{\va,\lambda}(x_0,y)-G_{0,\lambda}(x_0,y)|dy\\
&\quad+C_{k,\theta_0}\va R^{1-\frac{d}{p}}\|\nabla_2^2 G_{0,\lambda}(x_0,\cdot)\|_{L^p(\om(y_0,R))}\\
&\quad+C_{k,\theta_0}\va(1+|\lambda|R^2)\|\nabla_2G_{0,\lambda}(x_0,\cdot)\|_{L^{\infty}(\om(y_0,R))}.
\end{aligned}\label{ww75}
\ee
Firstly for the first term of $ \eqref{ww75} $, in view of $ \eqref{ww76}, $ we can obtain that it is bounded by 
\be 
C_{k,\theta_0}\va(1+|\lambda|R^2)^{-k}R^{1-d}.\label{zjbddx}
\ee
For the third term of $ \eqref{ww75} $, using $ \eqref{Green Lipschitz 11} $, it is also bounded by $ \eqref{zjbddx} $. For the second term, if $ d\geq 3 $, to obtain the same boundedness, we can use the $ W^{2,p} $ estimates for $ \mathcal{L}_0-\lambda I $, $ \eqref{W2p local} $, that is,
\begin{align}
\left(\dashint_{\om(y_0,R)}|\nabla_2^2G_{0,\lambda}(x_0,y)|^pdy\right)^{\frac{1}{p}}\leq \frac{C_{k,\theta_0}}{(1+|\lambda|R^2)^kR^2}\left(\dashint_{\om(y_0,2R)}|G_{0,\lambda}(x_0,y)|^2dy\right)^{\frac{1}{2}}.\label{GreenW2p}
\end{align}
This, together with $ \eqref{Greene} $, completes the proof for the case that $ d\geq 3 $. If $ d=2 $, we divide the proof into two cases. If $ \Delta(y_0,3R)\neq\emptyset $, then it follows directly from $ \eqref{Boundedness estimates} $ and $ \eqref{GreenW2p} $. If $ \Delta(y_0,3R)=\emptyset $, we can obtain the same result by using the interior Lipschitz estimate for the function $ \nabla_2G_{0,\lambda}(x_0,y) $,
\begin{align}
|\nabla_2^2G_{0,\lambda}(x_0,y_0)|\leq \frac{C_{k,\theta_0}}{(1+|\lambda|R^2)^kR}\left(\dashint_{B(y_0,2R)}|\nabla_2G_{0,\lambda}(x_0,y)|^2dy\right)^{\frac{1}{2}}
\end{align}
and $ \eqref{Green Lipschitz 11} $.
\end{proof}

\begin{thm}[Convergence of Green functions II]\label{Convergence of Green's functions 2t}
For $ \va>0 $ and $ d\geq 2 $, $ \lambda\in\Sigma_{\theta_0}\cup\{0\} $ with $ \theta_0\in(0,\frac{\pi}{2}) $, let $ \om $ be a bounded $ C^{2,1} $ domain in $ \mathbb{R}^d $. Suppose that $ A $ satisfies $ \eqref{sy} $, $ \eqref{el} $, $ \eqref{pe} $ and $ \eqref{Hol} $. Then for any $ k\in\mathbb{N}_+ $, $ 1\leq\al\leq m $ and $ x,y\in\om $ with $ x\neq y $,
\be
\begin{aligned}
\left|\frac{\pa }{\pa x_i}\{G_{\va,\lambda}^{\al\beta}(x,y)\}-\frac{\pa}{\pa x_i}\{\Phi_{\va,j}^{\al\beta}(x)\}\frac{\pa}{\pa x_j}\{G_{0,\lambda}^{\beta\gamma}(x,y)\}\right|\leq\frac{C_{k,\theta_0}\va \ln[\va^{-1}|x-y|+2]}{(1+|\lambda||x-y|^2)^{k}|x-y|^d},
\end{aligned}\label{Convergence of Green's functions 2}
\ee
where $ C_{k,\theta_0} $ depends only on $ \mu,d,m,\nu,\tau,k,\theta_0 $ and $ \om $.
\end{thm}

\begin{lem}
For $ \va>0 $ and $ d\geq 2 $, $ \lambda\in\Sigma_{\theta_0}\cup\{0\} $ with $ \theta_0\in(0,\frac{\pi}{2}) $, let $ \om $ be a bounded $ C^{2,1} $ domain in $ \mathbb{R}^d $. Suppose that $ A $ satisfies $ \eqref{sy} $, $ \eqref{el} $, $ \eqref{pe} $ and $ \eqref{Hol} $. Assume that $ u_{\va,\lambda}\in H^1(\om(x_0,4R);\mathbb{C}^m) $ and $ u_{0,\lambda}\in C^{2,\rho}(\om(x_0,4R);\mathbb{C}^m) $ for some $ 0<\rho<1 $. If $ \Delta(x_0,4R)\neq\emptyset $, assume that 
\begin{align}
(\mathcal{L}_{\va}-\lambda I)(u_{\va,\lambda})=(\mathcal{L}_0-\lambda I)(u_{0,\lambda})\text{ in } \om(x_0,4R)\quad\text{and}\quad u_{\va,\lambda}=u_{0,\lambda}\text{ on }\Delta(x_0,4R).\nonumber 
\end{align} 
If $ \Delta(x_0,4R)=\emptyset $, assume that
\begin{align}
(\mathcal{L}_{\va}-\lambda I)(u_{\va,\lambda})=(\mathcal{L}_0-\lambda I)(u_{0,\lambda})\text{ in } B(x_0,4R).\nonumber 
\end{align} 
Then for any $ 0<\va<r $, $ 1\leq\al\leq m $ and $ k\in\mathbb{N}_+ $,
\begin{align}
\|\frac{\pa u_{\va,\lambda}^{\al}}{\pa x_i}-\frac{\pa \Phi_{\va,j}^{\al\beta}}{\pa x_i}\frac{\pa u_{0,\lambda}^{\beta}}{\pa x_j}\|_{L^{\infty}(\om_R)}&\leq\frac{C_{k,\theta_0}}{(1+|\lambda|R^2)^kR}\dashint_{\om_{4R}}|u_{\va,\lambda}-u_{0,\lambda}|\nonumber\\
&\quad+C_{k,\theta_0}\va \left\{ (1+|\lambda|R^2)R^{-1}\|\nabla u_{0,\lambda}\|_{L^{\infty}(\om_{4R})}\right.\label{ww79}\\
&\quad+\ln[\va^{-1}R+2]\|\nabla^2u_{0,\lambda}\|_{L^{\infty}(\om_{4R})}+\left. R^{\rho}[\nabla^2 u_{0,\lambda}]_{C^{0,\rho}(\om_{4R})}\right\},\nonumber
\end{align}
where $ C_{k,\theta_0} $ depends on $ \mu,d,m,k,\theta_0,\nu,\tau,p,\rho $ and $ \om $.
\end{lem}
\begin{proof}
We only prove the case that $ \Delta(x_0,3R)\neq\emptyset $ and the other is similar in view of the proof of Lemma \ref{ww74}. We start by choosing a $ C^{2,1} $ domain $ \widetilde{\om} $ such that $ \om(x_0,3R)\subset\widetilde{\om}\subset \om(x_0,4R) $. Let
\begin{align}
w_{\va,\lambda}(x)=u_{\va,\lambda}(x)-u_{0,\lambda}(x)-[\Phi_{\va,j}^{\beta}(x)-P_j^{\beta}(x)]\frac{\pa u_{0,\lambda}^{\beta}}{\pa x_j}.\nonumber
\end{align}
Simple calculations imply that
\be
\begin{aligned}
\left\{(\mathcal{L}_{\va}-\lambda I)(w_{\va,\lambda})\right\}^{\alpha}&=-\va\frac{\pa}{\pa x_i}\left\{ F_{jik}^{\al\gamma}(x/\varepsilon)\frac{\pa^2u_{0,\lambda}^{\gamma}}{\pa x_j\pa x_k}\right\}\\
&\quad+a_{ij}^{\al\beta}(x/\va)\frac{\pa}{\pa x_j}[\Phi_{\va,k}^{\beta\gamma}(x)-x_k\delta^{\beta\gamma}-\va\chi_k^{\beta\gamma}(x/\va)]\frac{\pa^2 u_{0,\lambda}^{\gamma}}{\pa x_i\pa x_k}\\
&\quad+\frac{\pa}{\pa x_i}\left\{a_{ij}^{\al\beta}(x/\va)[\Phi_{\va,k}^{\beta\gamma}(x)-x_k\delta^{\beta\gamma}]\frac{\pa^2u_{0,\lambda}^{\gamma}}{\pa x_j\pa x_k}\right\}\\
&\quad+\lambda[\Phi_{\va,k}^{\al\beta}(x)-x_k\delta^{\al\beta}]\frac{\pa u_{0,\lambda}^{\beta}}{\pa x_k}
\end{aligned}\label{Equality 1}
\ee
Note that $ w_{\va,\lambda}=0 $ on $ \Delta(x_0,4R) $. Write $ w_{\va,\lambda}=v_{\va,\lambda}+z_{\va,\lambda} $ in $ \widetilde{\om} $, where $ v_{\va,\lambda}\in H_0^1(\widetilde{\om};\mathbb{C}^m) $ and $ (\mathcal{L}_{\va}-\lambda I)(v_{\va,\lambda})=(\mathcal{L}_{\va}-\lambda I)(w_{\va,\lambda}) $ in $ \widetilde{\om} $. Since $ (\mathcal{L}_{\va}-\lambda I)(z_{\va,\lambda})=0 $ in $ \widetilde{\om} $ and $ z_{\va,\lambda}=w_{\va,\lambda}=0 $ on $ \Delta(x_0,3R) $, it follows from the boundary Lipschitz estimate $ \eqref{Linfty} $ and $ \eqref{Lipesu} $ that for any $ k\in\mathbb{N}_+ $, 
\begin{align}
\|\nabla z_{\va,\lambda}\|_{L^{\infty}(\om_R)}&\leq \frac{C_{\theta_0}}{R}\left(\dashint_{\om_{3/2R}}|z_{\va,\lambda}|^{2}\right)^{\frac{1}{2}}\leq \frac{C_{\theta_0}}{R}\|z_{\va,\lambda}\|_{L^{\infty}(\om_{3/2R})}\nonumber\\
&\leq \frac{C_{k,\theta_0}}{(1+|\lambda|R^2)^kR}\dashint_{\om_{2R}}|z_{\va,\lambda}|\leq\frac{C_{k,\theta_0}}{(1+|\lambda|R^2)^kR}\dashint_{\om_{2R}}|w_{\va,\lambda}|+\frac{C_{k,\theta_0}}{R}\|v_{\va,\lambda}\|_{L^{\infty}(\om_{2R})}\nonumber\\
&\leq\frac{C_{k,\theta_0}}{(1+|\lambda|R^2)^kR}\dashint_{\om_{2R}}|u_{\va,\lambda}-u_{0,\lambda}|+\frac{C_{k,\theta_0}}{R}\left\{\|v_{\va,\lambda}\|_{L^{\infty}(\om_{2R})}+\va\|\nabla u_{0,\lambda}\|_{L^{\infty}(\om_{2R})}\right\},\nonumber 
\end{align}
where we have used $ \eqref{Estimate for Dirichlet correctors} $. This implies that
\begin{align}
&\|\nabla w_{\va,\lambda}\|_{L^{\infty}(\om_R)}\leq\|\nabla v_{\va,\lambda}\|_{L^{\infty}(\om_R)}+\|\nabla z_{\va,\lambda}\|_{L^{\infty}(\om_R)}\nonumber\\
&\quad\quad\leq\frac{C_{k\theta_0}}{(1+|\lambda|R^2)^kR}\dashint_{\om_{2R}}|u_{\va,\lambda}-u_{0,\lambda}|+\frac{C_{k,\theta_0}}{R}\left\{R\|\nabla v_{\va,\lambda}\|_{L^{\infty}(\om_{2R})}+\va\|\nabla u_{0,\lambda}\|_{L^{\infty}(\om_{2R})}\right\},\nonumber
\end{align}
where we have used $ \|v_{\va,\lambda}\|_{L^{\infty}(\om_{2R})}\leq CR\|\nabla v_{\va,\lambda}\|_{L^{\infty}(\om_{2R})} $ since $ v_{\va,\lambda}=0 $ on $ \pa\widetilde{\om} $. Owing to
\begin{align}
\frac{\pa w_{\va,\lambda}^{\al}}{\pa x_i}=\frac{\pa u_{\va,\lambda}^{\al}}{\pa x_i}-\frac{\pa \Phi_{\va,j}^{\al\beta}}{\pa x_i}\frac{\pa u_{0,\lambda}^{\beta}}{\pa x_j}-[\Phi_{\va,j}^{\al\beta}(x)-x_j\delta^{\al\beta}]\frac{\pa^2u_{0,\lambda}^{\beta}}{\pa x_i\pa x_j},\nonumber
\end{align}
we can obtain that for any $ 1\leq\al\leq m $ and $ k\in\mathbb{N}_+ $,
\be
\begin{aligned}
&\|\frac{\pa u_{\va,\lambda}^{\al}}{\pa x_i}-\frac{\pa \Phi_{\va,j}^{\al\beta}}{\pa x_i}\frac{\pa u_{0,\lambda}^{\beta}}{\pa x_j}\|_{L^{\infty}(\om_R)}\leq \|\nabla w_{\va,\lambda}\|_{L^{\infty}(\om_R)}+\|[\Phi_{\va}-P]\nabla^2u_{0,\lambda}\|_{L^{\infty}(\om_R)}\\
&\quad\quad\leq \|\nabla w_{\va,\lambda}\|_{L^{\infty}(\om_R)}+C\va\|\nabla^2 u_{0,\lambda}\|_{L^{\infty}(\om_R)}\\
&\quad\quad\leq\frac{C_{k\theta_0}}{(1+|\lambda|R^2)^kR}\dashint_{\om_{2R}}|u_{\va,\lambda}-u_{0,\lambda}|\\
&\quad\quad\quad\quad+C_{k,\theta_0}\left\{\va R^{-1}\|\nabla u_{0,\lambda}\|_{L^{\infty}(\om_{2R})}
+\|\nabla v_{\va,\lambda}\|_{L^{\infty}(\om_{2R})}+\va\|\nabla^2 u_{0,\lambda}\|_{L^{\infty}(\om_{2R})}\right\}.
\end{aligned}\label{ww78}
\ee
It remains to estimate $ \nabla v_{\va,\lambda} $ in $ \om_{2R} $. To this end, we use the representation formula
\begin{align}
v_{\va,\lambda}(x)=\int_{\widetilde{\om}}\widetilde{G}_{\va,\lambda}(x,y)\overline{(\mathcal{L}_{\va}-\lambda I)(w_{\va,\lambda})(y)}dy,\nonumber
\end{align}
where $ \widetilde{G}_{\va,\lambda}(x,y) $ is the Green function for $ \mathcal{L}_{\va}-\lambda I $ in the $ C^{2,1} $ domain $ \widetilde{\om} $. Let
\begin{align}
f_i^{\al}(x)=-\va F_{jik}^{\al\gamma}(x/\varepsilon)\frac{\pa^2u_{0,\lambda}^{\gamma}}{\pa x_j\pa x_k}+a_{ij}^{\al\beta}(x/\va)[\Phi_{\va,k}^{\beta\gamma}(x)-x_k\delta^{\beta\gamma}]\frac{\pa^2u_{0,\lambda}^{\gamma}}{\pa x_j\pa x_k}.\nonumber
\end{align}
In view of $ \eqref{Equality 1} $, we can obtain
\begin{align}
v_{\va,\lambda}(x)&=-\int_{\widetilde{\om}}\frac{\pa}{\pa y_i}\{\widetilde{G}_{\va,\lambda}(x,y)\}\overline{[f_i(y)-f_i(x)]}dy\nonumber\\
&\quad\quad+\int_{\widetilde{\om}}\widetilde{G}_{\va,\lambda}(x,y)\overline{a_{ij}(y/\va)\frac{\pa}{\pa y_i}[\Phi_{\va,k}(y)-P_k(y)-\va\chi_k(y/\va)]\frac{\pa^2 u_{0,\lambda}}{\pa y_i\pa y_k}}dy\nonumber\\
&\quad\quad+\overline{\lambda}\int_{\widetilde{\om}}\widetilde{G}_{\va,\lambda}(x,y)\overline{[\Phi_{\va,k}(y)-P_k(y)]\frac{\pa u_{0,\lambda}}{\pa y_k}}dy.\nonumber
\end{align}
It follows that
\be
\begin{aligned}
|\nabla v_{\va,\lambda}(x)|&\leq\int_{\widetilde{\om}}|\nabla_1\nabla_2\widetilde{G}_{\va,\lambda}(x,y)||f(y)-f(x)|dy\\
&\quad+C\|\nabla^2u_{0,\lambda}\|_{L^{\infty}(\om_{4R})}\int_{\widetilde{\om}}|\nabla_1\widetilde{G}_{\va,\lambda}(x,y)||\nabla[\Phi_{\va}(y)-P(y)-\va\chi(y/\va)]|dy\\
&\quad+C|\lambda|\|\nabla u_{0,\lambda}\|_{L^{\infty}(\om_{4R})}\int_{\widetilde{\om}}|\widetilde{G}_{\va,\lambda}(x,y)||\Phi_{\va}(y)-P(y)|dy.
\end{aligned}\label{ww77}
\ee
In view of the Lipshcitz estimates of Green functions, $ \eqref{Green Lipschitz 14} $, we have
\begin{align}
|\nabla_1\nabla_2\widetilde{G}_{\va,\lambda}(x,y)|\leq\frac{C_{k,\theta_0}}{(1+|\lambda||x-y|^2)^{k}|x-y|^d}\text{ for any }k\in\mathbb{N}_+.\label{yyd}
\end{align}
Meanwhile, simple calculations and $ \eqref{Estimate for Dirichlet correctors} $ give the $ L^{\infty} $ and Hölder estimates of $ f $, that is
\begin{align}
\|f\|_{L^{\infty}(\om_{4R})}&\leq C\va\|\nabla^2u_{0,\lambda}\|_{L^{\infty}(\om_{4R})},\nonumber\\
|f(x)-f(y)|&\leq C|x-y|^{\rho}\left\{\va^{1-\rho}\|\nabla^2u_{0,\lambda}\|_{L^{\infty}(\om_{4R})}+\va[\nabla^2u_{0,\lambda}]_{C^{0,\rho}(\om_{4R})}\right\}.\nonumber
\end{align}
Choosing $ k=0 $ in $ \eqref{yyd} $, it can be obtained that
\begin{align}
&\int_{\widetilde{\om}}|\nabla_1\nabla_2\widetilde{G}_{\va,\lambda}(x,y)||f(y)-f(x)|dy\nonumber\\
&\quad\quad\leq C_{\theta_0}\va\|\nabla^2u_{0,\lambda}\|_{L^{\infty}(\om_{4R})}\int_{\widetilde{\om}\backslash B(x,\va)}\frac{dy}{|x-y|^{d}}\nonumber\\
&\quad\quad\quad+C_{\theta_0} \left\{\va^{1-\rho}\|\nabla^2u_{0,\lambda}\|_{L^{\infty}(\om_{4R})}+\va[\nabla^2u_{0,\lambda}]_{C^{0,\rho}(\om_{4R})}\right\}\int_{\widetilde{\om}\cap B(x,\va)}\frac{dy}{|x-y|^{d-\rho}}\nonumber\\
&\quad\quad\leq C_{\theta_0} \left\{\va\ln[\va^{-1}R+2]\|\nabla^2u_{0,\lambda}\|_{L^{\infty}(\om_{4R})}+C\va^{1+\rho}[\nabla^2u_{0,\lambda}]_{C^{0,\rho}(\om_{4R})}\right\}.\nonumber
\end{align}
Finally, using the estimates
\begin{align}
|\nabla_1\widetilde{G}_{\va,\lambda}(x,y)|&\leq\frac{C_{k,\theta_0}}{(1+|\lambda||x-y|^2)^{k}|x-y|^{d-1}}\min\left\{1,\frac{\dist(y,\pa\widetilde{\om})}{|x-y|}\right\}\text{ for any }k\in\mathbb{N}_+,\nonumber
\end{align}
as well as the observation that for any $ 1\leq j\leq d $ and $ 1\leq \beta\leq m $,
\begin{align}
\left|\nabla\left\{\Phi_{\va,j}^{\beta}(x)-P_{j}^{\beta}(x)-\va \chi_{j}^{\beta}(x/\va)\right\}\right| \leq C \min \left\{1,\va[\operatorname{dist}(x,\pa\widetilde{\Omega})]^{-1}\right\},\nonumber
\end{align}
we can bound the second term in the right hand side of $ \eqref{ww77} $ by
\begin{align}
&C_{\theta_0}\|\nabla^2u_{0,\lambda}\|_{L^{\infty}(\om_{4R})}\left\{\va\int_{\widetilde{\om}\backslash B(x,\va)}\frac{dy}{|x-y|^{d}}+\int_{\widetilde{\om}\cap B(x,\va)}\frac{dy}{|x-y|^{d-1}}\right\}\nonumber\\
&\quad\quad\leq C_{\theta_0}\va \ln[\va^{-1}R+2]\|\nabla^2u_{0,\lambda}\|_{L^{\infty}(\om_{4R})}.\nonumber
\end{align}
For the third term of the right hand side of $ \eqref{ww77} $, by rescaling and $ \eqref{*} $, we can obtain that this term is bounded by
\begin{align}
C_{\theta_0}\va (1+|\lambda|R^2)R^{-1}\|\nabla u_{0,\lambda}\|_{L^{\infty}(\om_{4R})}.\nonumber
\end{align}
As a result, we have proved that
\begin{align}
\|\nabla v_{\va,\lambda}\|_{L^{\infty}(\om_{3R})}&\leq  C_{\theta_0}\va (1+|\lambda|R^2) R^{-1}\|\nabla u_{0,\lambda}\|_{L^{\infty}(\om_{4R})}\nonumber\\
&\quad\quad+C_{\theta_0} \left\{\va\ln[\va^{-1}R+2]\|\nabla^2u_{0,\lambda}\|_{L^{\infty}(\om_{4R})}+\va R^{\rho}[\nabla^2u_{0,\lambda}]_{C^{0,\rho}(\om_{4R})}\right\}.\nonumber
\end{align}
This, together with $ \eqref{ww78} $, completes the proof.
\end{proof}

\begin{proof}[Proof of Theorem \ref{Convergence of Green's functions 2t}] Fix $ x_0,y_0\in\om $ and $ R=\frac{|x_0-y_0|}{16} $. We may assume that $ 0<\va<r $, since the case $ \va\geq r $ is trivial and follows directly from the size estimates of $ |\nabla_1G_{\va,\lambda}(x,y)| $, $ |\nabla_2G_{\va,\lambda}(x,y)| $ (see $ \eqref{Green Lipschitz 11} $) and $ \eqref{Estimate for Dirichlet correctors} $. For any $ 1\leq\gamma\leq m $, let $ u_{\va,\lambda}(x)=G_{\va,\lambda}^{\gamma}(x,y_0) $ and $ u_{0,\lambda}(x)=G_{0,\lambda}^{\gamma}(x,y_0) $. Observe that $ (\mathcal{L}_{\va}-\lambda I)(u_{\va,\lambda})=(\mathcal{L}_0-\lambda I)(u_{0,\lambda})=0 $ in $ \om(x_0,4R) $ and $ u_{\va,\lambda}=u_{0,\lambda}=0 $ on $ \Delta(x_0,4R) $ (if $ \Delta(x_0,4R)\neq\emptyset $). By Theorem \ref{Convergence of Green's functions}, it can be obtained,
\begin{align}
\|u_{\va,\lambda}-u_{0,\lambda}\|_{L^{\infty}(\om(x_0,4R))}\leq C_{\theta_0}\va R^{1-d}.\nonumber
\end{align}
Also, since $ \om $ is $ C^{2,1} $, we have, for any $ k\in\mathbb{N}_+ $,
\begin{align}
\|\nabla u_{0,\lambda}\|_{L^{\infty}(\om_{4R})}&\leq \frac{C_{k,\theta_0} }{(1+|\lambda|R^2)^kR^{d-1}},\nonumber\\
\|\nabla^2 u_{0,\lambda}\|_{L^{\infty}(\om_{4R})}\leq \frac{C_{k,\theta_0} }{(1+|\lambda|R^2)^kR^d}\quad&\text{and}\quad\|\nabla^2 u_{0,\lambda}\|_{C^{0,\rho}(\om_{4R})}\leq \frac{C_{k,\theta_0} }{(1+|\lambda|R^2)^kR^{d+\rho}}.\nonumber
\end{align}
Here, we have used $ \eqref{ww80} $, $ \eqref{ww81} $, $ \eqref{Green Lipschitz 11} $ and arguments in the proof of Theorem \ref{Convergence of Green's functions}. Hence, in view of $ \eqref{ww79} $, we can complete the proof.
\end{proof}

\subsection{Proof of Theorem \ref{Approximation 1}, Theorem \ref{Lpconres} and \ref{LpW1pconreso}}
\begin{proof}[Proof of Theorem \ref{Approximation 1}]
For $ u_{\va,\lambda}=R(\lambda,\mathcal{L}_{\va})F $ and $ u_{0,\lambda}=R(\lambda,\mathcal{L}_{0})F $, let 
\begin{align}
w_{\va,\lambda}(x)=u_{\va,\lambda}(x)-u_{0,\lambda}(x)-[\Phi_{\va,j}^{\beta}(x)-P_{j}^{\beta}(x)]\frac{\pa u_{0,\lambda}^{\beta}}{\pa x_j}.\nonumber
\end{align}
Using the equality $ \eqref{Equality 1} $ and choosing $ w_{\va,\lambda} $ as the test function, we can obtain that
\begin{align}
B_{\va,\lambda,\om}[w_{\va,\lambda},w_{\va,\lambda}]=\int_{\om}A(x/\va)\nabla w_{\va,\lambda}\overline{\nabla w_{\va,\lambda}}dx-\lambda\int_{\om}|w_{\va,\lambda}|^2dx=J_{\va,\lambda,\om}[u_{0,\lambda},w_{\va,\lambda}],\quad\va>0,\label{Equation 1}
\end{align}
where the bilinear form $ J_{\va,\lambda,\om}[\cdot,\cdot]:H_0^1(\om;\mathbb{C}^m)\times H_0^1(\om;\mathbb{C}^m)\to\mathbb{C} $ is defined by
\be
\begin{aligned}
J_{\va,\lambda,\om}[u,v]&=-\int_{\om}\va  F_{jik}^{\al\gamma}(x/\varepsilon)\frac{\pa^2 u^{\gamma}}{\pa x_j\pa x_k}\overline{\frac{\pa v^{\al}}{\pa x_i}}dx+\lambda\int_{\om}[\Phi_{\va,k}^{\al\beta}(x)-x_k\delta^{\al\beta}]\frac{\pa u^{\beta}}{\pa x_k}\overline{v^{\al}}dx\\
&\quad\quad-\int_{\om}a_{ij}^{\al\beta}(x/\va)[\Phi_{\va,k}^{\beta\gamma}(x)-x_k\delta^{\beta\gamma}]\frac{\pa^2 u^{\beta}}{\pa x_j\pa x_k}\overline{\frac{\pa v^{\al}}{\pa x_i}}dx\\
&\quad\quad+\int_{\om}a_{ij}^{\al\beta}(x/\va)\frac{\pa}{\pa x_j}[\Phi_{\va,k}^{\beta\gamma}(x)-x_k\delta^{\beta\gamma}-\va\chi_k^{\beta\gamma}(x/\va)]\frac{\pa^2 u^{\gamma}}{\pa x_i\pa x_k}\overline{v^{\al}}dx.
\end{aligned}\label{J linear}
\ee
Taking $ u=u_{0,\lambda} $ and  $ v=w_{\va,\lambda} $ and using $ \eqref{Estimate for Dirichlet correctors} $, it can be inferred that
\begin{align}
|J_{\va,\lambda,\om}[u_{0,\lambda},w_{\va,\lambda}]|&\leq C\va \int_{\om}|\nabla^2u_{0,\lambda}||\nabla w_{\va,\lambda}|dx+C\va|\lambda|\int_{\om}|\nabla u_{0,\lambda}||w_{\va,\lambda}|dx\nonumber\\
&\quad+C\int_{\om}|\nabla[\Phi_{\va}(x)-P(x)-\va\chi(x/\va)]||\nabla^2u_{0,\lambda}||w_{\va,\lambda}|dx\nonumber\\
&\leq C\va|\lambda|\|\nabla u_{0,\lambda}\|_{L^2(\om)}\|w_{\va,\lambda}\|_{L^2(\om)}+C\va\|\nabla^2 u_{0,\lambda}\|_{L^2(\om)}\|\nabla w_{\va,\lambda}\|_{L^2(\om)}\nonumber\\
&\quad+C\|\nabla[\Phi_{\va}(\cdot)-P(\cdot)-\va\chi(\cdot/\va)]w_{\va,\lambda}\|_{L^2(\om)}\|\nabla^2 u_{0,\lambda}\|_{L^2(\om)}.\nonumber
\end{align}
To estimate $ |J_{\va,\lambda,\om}[u_{0,\lambda},w_{\va,\lambda}]| $, we first claim that for any $ 1\leq j\leq d $ and $ 1\leq \beta\leq m $,
\begin{align}
\|\nabla[\Phi_{\va,j}^{\beta}(\cdot)-P_j^{\beta}(\cdot)-\va\chi_j^{\beta}(\cdot/\va)]w_{\va,\lambda}\|_{L^2(\om)}\leq C\va\|\nabla w_{\va,\lambda}\|_{L^2(\om)}.\label{Claim 1}
\end{align}
To see $ \eqref{Claim 1} $, we fix $ 1\leq \beta_0\leq m $, $ 1\leq j_0\leq d $ and let
\begin{align}
h_{\va}(x)=\Phi_{\va,j_0}^{\beta_0}(x)-P_{j_0}^{\beta_0}(x)-\va\chi_{j_0}^{\beta_0}(x/\va),\text{ where }x\in\om.\nonumber
\end{align}
Note that $ h_{\va}\in H^1(\Omega;\mathbb{C}^m)\cap L^{\infty}(\Omega;\mathbb{C}^m) $ and $ \mathcal{L}_{\va}(h_{\va})=0 $ in $ \om $. It follows that
\begin{align}
\mu\int_{\om}|\nabla h_{\va}|^2|w_{\va,\lambda}|^2&\leq\int_{\om}a_{ij}^{\al\beta}(x/\va)\frac{\pa h_{\va}^{\beta}}{\pa x_j}\overline{\frac{\pa  h_{\va}^{\al}}{\pa x_i}}|w_{\va,\lambda}|^2dx\nonumber\\
&=-\int_{\om}\overline{h_{\va}^{\al}}a_{ij}^{\al\beta}(x/\va)\frac{\pa h_{\va}^{\beta}}{\pa x_j}\overline{\frac{\pa  w_{\va,\lambda}^{\gamma}}{\pa x_i}}w_{\va,\lambda}^{\gamma}dx-\int_{\om}\overline {h_{\va}^{\al}}a_{ij}^{\al\beta}(x/\va)\frac{\pa h_{\va}^{\beta}}{\pa x_j}\frac{\pa  w_{\va,\lambda}^{\gamma}}{\pa x_i}\overline{w_{\va,\lambda}^{\gamma}}dx,\nonumber
\end{align}
where we have used integration by parts. Hence
\begin{align}
\int_{\om}|\nabla h_{\va}|^2|w_{\va,\lambda}|^2dx\leq C\int_{\om}|h_{\va}||\nabla h_{\va}||\nabla w_{\va,\lambda}||w_{\va,\lambda}|dx,\nonumber
\end{align}
where $ C $ depends on $ d,m $ and $ \mu $. This directly implies the claim by noticing that $ \|h_{\va}\|_{L^{\infty}(\om)}\leq C\va $ and using the inequality
\begin{align}
\int_{\om}|\nabla h_{\va}|^2|w_{\va,\lambda}|^2dx\leq \frac{1}{2}\int_{\om}|\nabla h_{\va}|^2|w_{\va,\lambda}|^2dx+C\int_{\om}| h_{\va}|^2|\nabla w_{\va,\lambda}|^2dx,\nonumber
\end{align}
where we have used the $ \eqref{inte} $. Then the claim implies that
\begin{align}
|J_{\va,\lambda,\om}[u_{0,\lambda},w_{\va,\lambda}]|\leq C\va|\lambda|\|\nabla u_{0,\lambda}\|_{L^2(\om)}\|w_{\va,\lambda}\|_{L^2(\om)}+C\va\|\nabla^2 u_{0,\lambda}\|_{L^2(\om)}\|\nabla w_{\va,\lambda}\|_{L^2(\om)}.\label{Estimate of J}
\end{align}
In view of $ \eqref{ABCDE} $, $ \eqref{inte} $ and $ \eqref{Equation 1} $, it can be easily shown that
\begin{align}
\|w_{\va,\lambda}\|_{L^2(\om)}^2&\leq \frac{Cc(\lambda,\theta)}{R_0^{-2}+|\lambda|}|B_{\va,\lambda,\om}[w_{\va,\lambda},w_{\va,\lambda}]|\leq \frac{Cc(\lambda,\theta)}{R_0^{-2}+|\lambda|}|J_{\va,\lambda,\om}[u_{0,\lambda},w_{\va,\lambda}]|\nonumber\\
&\leq \frac{C\va c(\lambda,\theta)}{R_0^{-2}+|\lambda|}\left\{|\lambda|\|\nabla u_{0,\lambda}\|_{L^2(\om)}\|w_{\va,\lambda}\|_{L^2(\om)}+\|\nabla^2 u_{0,\lambda}\|_{L^2(\om)}\|\nabla w_{\va,\lambda}\|_{L^2(\om)}\right\}\nonumber\\
&\leq C\va^2 c^2(\lambda,\theta)\|\nabla u_{0,\lambda}\|_{L^2(\om)}^2+\frac{1}{2}\|w_{\va,\lambda}\|_{L^2(\om)}^2+\frac{C\va c(\lambda,\theta)}{R_0^{-2}+|\lambda|}\|\nabla^2u_{0,\lambda}\|_{L^2(\om)}\|\nabla w_{\va,\lambda}\|_{L^2(\om)}.\nonumber
\end{align}
Then it is not hard to obtain the $ L^2 $ estimate of $ w_{\va,\lambda} $, that is,
\begin{align}
\|w_{\va,\lambda}\|_{L^2(\om)}^2\leq C\va^2 c^2(\lambda,\theta)\|\nabla u_{0,\lambda}\|_{L^2(\om)}^2+\frac{C\va c(\lambda,\theta)}{R_0^{-2}+|\lambda|}\|\nabla^2u_{0,\lambda}\|_{L^2(\om)}\|\nabla w_{\va,\lambda}\|_{L^2(\om)}.\label{wvadiyi}
\end{align}
Similarly, owing to $ \eqref{laxmil} $, it can be obtained without difficulty that
\begin{align}
\|\nabla w_{\va,\lambda}\|_{L^2(\om)}^2&\leq Cc(\lambda,\theta)|B_{\va,\lambda,\om}[w_{\va,\lambda},w_{\va,\lambda}]|\leq Cc(\lambda,\theta)|J_{\va,\lambda,\om}[u_{0,\lambda},w_{\va,\lambda}]|\nonumber\\
&\leq C\va c(\lambda,\theta)\left\{|\lambda|\|\nabla u_{0,\lambda}\|_{L^2(\om)}\|w_{\va,\lambda}\|_{L^2(\om)}+\|\nabla^2 u_{0,\lambda}\|_{L^2(\om)}\|\nabla w_{\va,\lambda}\|_{L^2(\om)}\right\}\nonumber\\
&\leq C\va^2 c^2(\lambda,\theta)|\lambda|\|\nabla u_{0,\lambda}\|_{L^2(\om)}^2+|\lambda|\|w_{\va,\lambda}\|_{L^2(\om)}^2+C\va c(\lambda,\theta)\|\nabla^2 u_{0,\lambda}\|_{L^2(\om)}\|\nabla w_{\va,\lambda}\|_{L^2(\om)}.\nonumber
\end{align}
This, together with $ \eqref{inte} $ and $ \eqref{wvadiyi} $, gives that
\begin{align}
\|\nabla w_{\va,\lambda}\|_{L^2(\om)}^2&\leq C\va^2c^2(\lambda,\theta)|\lambda|\|\nabla u_{0,\lambda}\|_{L^2(\om)}^2+C\va c(\lambda,\theta)\|\nabla^2 u_{0,\lambda}\|_{L^2(\om)}\|\nabla w_{\va,\lambda}\|_{L^2(\om)}\nonumber\\
&\leq C\va^2c^2(\lambda,\theta)|\lambda|\|\nabla u_{0,\lambda}\|_{L^2(\om)}^2+C\va^2 c^2(\lambda,\theta)\|\nabla^2 u_{0,\lambda}\|_{L^2(\om)}^2+\frac{1}{2}\|\nabla w_{\va,\lambda}\|_{L^2(\om)}^2.\nonumber
\end{align}
In view of $ \eqref{L2uva}$ and $ \eqref{L2n2u0} $, we can estimate $ \|\nabla u_{0,\lambda}\|_{L^2(\om)} $ and $ \|\nabla^2u_{0,\lambda}\|_{L^2(\om)} $. Then
\begin{align}
\|\nabla w_{\va,\lambda}\|_{L^2(\om)}\leq C\va c^2(\lambda,\theta)\|F\|_{L^2(\om)},\nonumber
\end{align}
which completes the proof of $ \eqref{Convergence rate 1} $. This, together with $\eqref{wvadiyi} $, shows $ \eqref{Convergence rate 11} $.
\end{proof}

\begin{proof}[Proof of Theorem \ref{Lpconres}]
In view of the representation formula $ \eqref{repre} $ and $ \eqref{Convergence of Green's functions formula} $, it can be seen that
\begin{align}
\|R(\lambda,\mathcal{L}_{\va})-R(\lambda,\mathcal{L}_{0})\|_{L^{\infty}(\om)\to L^{\infty}(\om)}&\leq C_{\theta_0}\va(R_0^{-2}+|\lambda|)^{-\frac{1}{2}},\nonumber\\
\|R(\lambda,\mathcal{L}_{\va})-R(\lambda,\mathcal{L}_{0})\|_{L^{1}(\om)\to L^{1}(\om)}&\leq C_{\theta_0}\va(R_0^{-2}+|\lambda|)^{-\frac{1}{2}}.\nonumber
\end{align}
These, together with $ \eqref{Operator estimate 11} $ and the M. Riesz interpolation theorem, give $ \eqref{**-} $.
\end{proof}

\begin{proof}[Proof of Theorem \ref{LpW1pconreso}]
It follows directly from $ \eqref{Convergence of Green's functions 2} $ and the arguments in Theorem 6.5.2 in \cite{Shen2}. We give the proof here for the sake of completeness. Let $ u_{\va,\lambda},u_{0,\lambda}\in H_0^1(\om;\mathbb{C}^m) $ such that $ (\mathcal{L}_{\va}-\lambda I)(u_{\va,\lambda})=F $ and $ (\mathcal{L}_{0}-\lambda I)(u_{0,\lambda})=F $ with $ F\in L^p(\om;\mathbb{C}^m) $. By using $ \eqref{Estimate for Dirichlet correctors} $ and $ \eqref{W2p} $, we only need to show that for any $ 1\leq p\leq \infty $ and $ 1\leq\al\leq m $,
\begin{align}
\|\frac{\pa u_{\va,\lambda}^{\al}}{\pa x_i}-\frac{\pa\Phi_{\va,j}^{\al\beta}}{\pa x_i}\frac{\pa u_{0,\lambda}^{\beta}}{\pa x_j}\|_{L^p(\om)}\leq C_{\theta_0}\va\{\ln[\va^{-1}R_0+2]\}^{4|\frac{1}{2}-\frac{1}{p}|}\|F\|_{L^p(\om)}.\nonumber
\end{align}
In view of Theorem \ref{Convergence of Green's functions 2t}, we can obtain that
\begin{align}
\left|\frac{\pa u_{\va,\lambda}^{\al}}{\pa x_i}-\frac{\pa\Phi_{\va,j}^{\al\beta}}{\pa x_i}\frac{\pa u_{0,\lambda}^{\beta}}{\pa x_j}\right|\leq C_{\theta_0}\int_{\om}K_{\va}(x,y)|F(y)|dy,\nonumber
\end{align}
where the kernel $ K_{\va} $ is defined by
\begin{align}
K_{\va}(x,y)=\left\{\begin{matrix}
\va|x-y|^{-d}\ln[\va^{-1}|x-y|+2]&\text{if}&|x-y|\geq\va,\\
|x-y|^{1-d}&\text{if}&|x-y|<\va.
\end{matrix}\right.\nonumber
\end{align}
Direct computations imply that
\begin{align}
\sup_{x\in\om}\int_{\om}K_{\va}(x,y)dy+\sup_{y\in\om}\int_{\om}K_{\va}(x,y)dy\leq C_{\theta_0}\{\ln[\va^{-1}R_0+2]\}^2.\label{Kva1}
\end{align}
This gives $ \eqref{Convergence rate LpW1p} $ for cases $ p=1 $ and $ \infty $. Thus, by the M. Riesz interpolation theorem, the result is a direct consequence of the case $ p=2 $, which is given by $ \eqref{Convergence rate 1} $.
\end{proof}

\section*{Acknowledgments}
I am grateful to Professor Jun Geng of Lanzhou University for warm guidance on the topics of the homogenization theory for elliptic systems. I am also grateful to Professor Zhifei Zhang of Peking University for some important inspirations on the homogenization theory. I sincerely thank the anonymous reviewers for their constructive revision suggestions.

\end{document}